\RequirePackage{rotating}
\documentclass[a4paper,onepage,reqno,12pt]{amsart}
\usepackage[top=30mm,right=30mm,bottom=30mm,left=30mm]{geometry}

\usepackage{pdfsync}
\usepackage{eurosym}
\usepackage{comment}
\usepackage{tabularx}
\usepackage{graphicx}
\usepackage{amsmath}
\usepackage{amssymb}
\usepackage{amsfonts}
\usepackage{amsthm}
\usepackage{amstext}
\usepackage{amsbsy}
\usepackage{amsopn}
\usepackage{amscd}
\usepackage{enumerate}
\usepackage{color}
\usepackage[colorlinks=true,linkcolor=red,citecolor=blue]{hyperref}
\usepackage[hyperpageref]{backref}
\usepackage[numbers,sort&compress]{natbib}
\usepackage{multirow}
\usepackage{longtable}
\usepackage{float}
\usepackage{lscape}
\usepackage{pdflscape}
\usepackage{url}
\usepackage{adjustbox}
\usepackage{rotating}
\usepackage{longtable}
\usepackage{etoolbox}
\usepackage{lipsum}
\usepackage[]{algorithm2e}

\newtheorem{theorem}{Theorem}[section]
\newtheorem{corollary}[theorem]{Corollary}

\newtheorem{lemma}[theorem]{Lemma}
\newtheorem{proposition}[theorem]{Proposition}

\theoremstyle{definition}

\newtheorem{example}[theorem]{Example}
\numberwithin{equation}{section}

\newcommand{\C}{\mathrm{C}}

\newcommand{\Dmc}{\mathcal{D}}

\renewcommand{\leq}{\leqslant}
\renewcommand{\geq}{\geqslant}

\newcommand{\imod}[1]{ \mkern4mu({\operator@font mod}\,\,#1)}

\usepackage{tabularx}
\makeatletter
\def\hlinewd#1{%
\noalign{\ifnum0=`}\fi\hrule \@height #1 %
\futurelet\reserved@a\@xhline}
\makeatother

\begin{document}
\title[Flag-transitive symmetric designs]{The Flag-Transitive and Point-Imprimitive Symmetric $(v,k,\lambda)$ Designs with $v<100$}

\author[M. Galici]{Mario Galici}
\address{Mario Galici, Dipartimento di Matematica e Fisica ``E. De Giorgi'', University of Salento, Lecce, Italy. }%
\email{mario.galici@unisalento.it}
\author[A. Montinaro]{Alessandro Montinaro}%
\thanks{Corresponding author: Alessandro Montinaro}
\address{Alessandro Montinaro, Dipartimento di Matematica e Fisica ``E. De Giorgi'', University of Salento, Lecce, Italy. }%
\email{alessandro.montinaro@unisalento.it}

\date{\today}%

\begin{abstract}
A complete classification of the flag-transitive point-imprimitive symmetric $2$-$(v,k,\lambda )$ designs with $v<100$ is provided. Apart from the known examples with $\lambda \leq 10$, the complementary design of $PG_{5}(2)$, and the $2$-design $\mathcal{S}^{-}(3)$ constructed by Kantor in \cite{Ka75}, we found two non isomorphic $2$-$(64,28,12)$ designs. They were constructed via computer as developments of $(64,28,12)$-difference sets by AbuGhneim in \cite{OAG}. In the present paper, independently from \cite{OAG}, we construct the aforementioned two $2$-designs and we prove that their full automorhpism group is flag-transitive and point-imprimitive. The construction is theoretical and relies on the the absolutely irreducible $8$-dimensional $\mathbb{F}_{2}$-representation of $PSL_{2}(7)$. Our result, together with that about the flag-transitive point-primitive symmetric $2$-designs with $v<2500$ by Brai\'{c}-Golemac-Mandi\'{c}-Vu\v{c}i\v{c}i\'{c} \cite{BGMV}, provides a complete classification of the flag-transitive $2$-designs with $v<100$.  
\end{abstract}
\maketitle

\section{Introduction and Main Theorem}\label{Sec1}

A $2$-$(v,k,\lambda )$ \emph{design} $\mathcal{D}$ is a pair $(\mathcal{P},%
\mathcal{B})$ with a set $\mathcal{P}$ of $v$ points and a set $\mathcal{B}$
of $b$ blocks such that each block is a $k$-subset of $\mathcal{P}$ and each two distinct points are contained in $\lambda $ blocks. The \emph{replication number} $r$ of $\mathcal{D}$ is the number of blocks containing a given point.
We say $\mathcal{D}$ is \emph{non-trivial} if $2<k<v$, and \emph{symmetric} if $v=b$. 
Given a $2$-$(v,k,\lambda )$ design $\mathcal{D}$, the incidence structure $\overline{\mathcal{D}}=(\mathcal{P},%
\mathcal{B}^{\prime})$, where $\mathcal{B}^{\prime}=\{\mathcal{P}\setminus B: B \in \mathcal{B}\}$, is a $2$-$(v,v-k,b-2r+\lambda)$ design, called the \emph{complementary design} to $\mathcal{D}$. 

An automorphism of $%
\mathcal{D}$ is a permutation of the point set which preserves the block
set. The set of all automorphisms of $\mathcal{D}$ with the composition of
permutations forms a group, denoted by $\mathrm{Aut(\mathcal{D})}$. Clearly, an automorphism of $\mathcal{D}$ is also an automorphism of $\overline{\mathcal{D}}$. For a subgroup $G$ of $\mathrm{Aut(\mathcal{D})}$, $G$ is said to be \emph{%
point-primitive} if $G$ acts primitively on $\mathcal{P}$, and said to be 
\emph{point-imprimitive} otherwise. In this setting, we also say that $%
\mathcal{D}$ is either \emph{point-primitive} or \emph{point-imprimitive}, respectively. A \emph{flag} of $\mathcal{D}$ is a pair $(x,B)$ where $x$ is
a point and $B$ is a block containing $x$. If $G\leq \mathrm{Aut(\mathcal{D})%
}$ acts transitively on the set of flags of $\mathcal{D}$, then we say that $%
G$ is \emph{flag-transitive} and that $\mathcal{D}$ is a \emph{%
flag-transitive design}.

This paper is a contribution to the problem of constructing and classifying designs with a rich automorphism group. Although the original motivation for our investigation was the paper by Praeger and Zhou \cite{PZ} as well as that of  Mandi\'{c}-\v{S}uba\v{s}i\'{c} \cite{MS} on the symmetric $2$-designs with $\lambda \leq 10$,  our starting point is different: we assume that $v$ rather than $\lambda$ has an upper bound. More precisely, we assume $v<100$. In this setting, we provide a complete classification of the pair $(\mathcal{D},G)$ and, apart from the known examples with $\lambda \leq 10$, the $2$-design $\overline{PG_{5}(2)}$ provided by Cameron and Kantor in \cite[Theorem III]{CK}, we show that there are three (up to isomorphism) $2$-$(64,28,12)$ designs admitting a flag-transitive point imprimitive automorphism group, one of them being the $2$-design $\mathcal{S}^{-}(3)$ constructed by Kantor in \cite{Ka75}. The remaining two symmetric $2$-design were constructed via computer as developments of $(64,28,12)$-difference sets by AbuGhneim \cite{OAG}, although he did not prove the flag-transitivity. In the present paper, independently from \cite{OAG}, we construct the aforementioned two inequivalent $2$-designs. The construction is theoretical and relies on the the absolutely irreducible $8$-dimensional $\mathbb{F}_{2}$-representation of $PSL_{2}(7)$ (see \cite{AtMod}). Further, we show that any flag-transitive point-imprimitive automorphism group of any of the two $2$-design is $2^{8}:PSL_{2}(7)$, which also is the full automorphism group of each of the two $2$-designs. More precisely, we obtain the following classification result:

\begin{theorem}\label{main}
Let $\mathcal{D}$ be a symmetric $2$-$(v,k,\lambda )$ design admitting a
flag-transitive point-imprimitive automorphism group $G$. If $v<100$, then $(\mathcal{D},G)$ are as in Table \ref{tavola1}.

    \begin{table}[h!]
	\centering
	\caption{Symmetric $(v,k,\lambda)$ designs $\Dmc$ with $v<100$ admitting a flag-transitive point-imprimitive automorphism group.}\label{tavola1}
		\begin{tabular}{ccccccccr}
			\noalign{\smallskip}\hline\noalign{\smallskip}
			Line &
			$v$ &
			$k$ &
			$\lambda$ &
			Design &
			$G$&
            Isom. classes&
			References \\
			\noalign{\smallskip}\hline\noalign{\smallskip}
			\multirow{1}{*}{$1$} &
			$16$ &
			$6$ &
			$2$ &
			    & 
            $G$ is as in Line 1 of Table \ref{telionos} &
            $1$ &
			\cite{Na,Hu,ORR}
			\\
			\multirow{1}{*}{$2$} &
			$16$ &
			$6$ &
			$2$ &
			      &
			$G$ is as in Line 2 of Table \ref{telionos} &
            $1$ &
			\cite{Na,Hu,ORR}
			\\
			\multirow{1}{*}{3} &
			$45$ & 
			$12$ & 
			$3$ &
                &
			$G$ is as in Line 3 of Table \ref{telionos} &
            $1$ &
			\cite{P,sane}\\
			\multirow{1}{*}{$4$} &
			$15$ &
			$8$ &
			$4$ &
			$\overline{PG_{3}(2)}$   &
			$G$ is as in Line 4 of Table \ref{telionos} &
            $1$ &
			\cite{CK,PZ}
			\\
            \multirow{1}{*}{$5$} &
			$63$ &
			$32$ &
			$16$ &
			$\overline{PG_{5}(2)}$   &
			$\Sigma L_{3}(4) \unlhd G \leq \Gamma L_{3}(4)$ &
            $1$ &
			\cite{CK}
            \\
			\multirow{1}{*}{$6$} &
			$64$ &
			$28$ &
			$12$ &
			 $\mathcal{S}^{-}(3)$   &
			$2^{3}:(2^{3}:PSL_{2}(7))$ & 
            $1$ &
            \cite{Ka75}
            \\
            \multirow{1}{*}{} &
               &
               &
               &
               &
            $2^{6}:PSL_{2}(7),2^{6}.PSL_{2}(7)$& 
			\\
            \multirow{1}{*}{} &
               &
               &
               &
               &
            $(2^{6}.2^{3}):7,(2^{6}.2^{3}):(7:3)$& 
			\\
			\multirow{1}{*}{$7$} &
			$64$ &
			$28$ &
			$12$ &
			    &
			$2^{8}:PSL_{2}(7)$ &
            $2$ &
			Section \ref{Sec2}, \cite{OAG}
            \\
            \multirow{1}{*}{$8$} &
			$96$ &
			$20$ &
			$4$ &
                &
			$G$ is as in Line 5 of Table \ref{telionos} &
            $1$ &
			\cite{OPS,LPR}
			\\
			\multirow{1}{*}{$9$} &
			$96$ &
			$20$ &
			$4$ &
			    &
			$G$ is as in Line 6 of Table \ref{telionos} &
            $1$ &
			\cite{LPR,OPS}
			\\
			\multirow{1}{*}{$10$} &
			$96$ &
			$20$ &
			$4$ &
                &
			$G$ is as in Line 7 of Table \ref{telionos} &
            $1$ &
			\cite{LPR,OPS}
			\\
			\multirow{1}{*}{$11$} &
			$96$ &
			$20$ &
			$4$ & 
                &
			$G$ is as in Line 8 of Table \ref{telionos} &
            $1$ &
			\cite{LPR,OPS}
			\\ 

			\noalign{\smallskip}\hline
		\end{tabular}
\end{table}
\end{theorem}

\bigskip

\begin{corollary}\label{cor1}
The flag-transitive symmetric $2$-$(v,k,\lambda)$ designs with $v<100$ are known.   
\end{corollary}
\bigskip 






\subsection{Structure of the paper and outline of the proof.} 
The paper consists of 5 sections briefly described below. In Section 1, we introduce the problem and state our main results: Theorem \ref{main} and Corollary \ref{cor1}. Section 2 focuses on the construction of two non isomorphic $2$-$(64,28,12)$ design admitting $2^{8}:PSL_{2}(7)$ as a flag-transitive point-imprimitive (full) automorphism group. As mentioned above, they were constructed via computer as developments of $(64,28,12)$-difference sets by AbuGhneim \cite{OAG}, although he did not prove the flag-transitivity. Here,  we construct them theoretically using the absolutely irreducible $8$-dimensional $\mathbb{F}_{2}$-representation of $PSL_{2}(7)$.

In section 3, we introduce our main analysis tool: the Theorem of Camina-Zieschang \cite[Propositions 2.1 and 2.3]{CZ}, which associates with the flag-transitive point-imprimitive $2$-design $\mathcal{D}$ two possibly trivial $2$-designs $\mathcal{D}_{0}$ and $\mathcal{D}_{1}$. The first one is induced on each block of imprimitivity, the second one on the $G$-invariant partition. The restrictions on the parameters of $\mathcal{D}_{0}$, $\mathcal{D}_{1}$ and $\mathcal{D}$ together with $v<100$ lead to precise parameters tuples of the designs $\mathcal{D}_{0}$, $\mathcal{D}_{1}$ and $\mathcal{D}$. 
The complete determination of the possible $2$-designs isomorphic to $\mathcal{D}_{0}$ or $\mathcal{D}_{1}$, together with their corresponding automorphism groups, is given in the appendix (the final section of the paper).
The determination of such designs is obtained by using the classification of the finite primitive groups up to order $100$ provided in \cite[Table B.4]{DM} together with some specific geometry of the classical groups, and in some very few cases \texttt{GAP}\cite{GAP}. So, all the admissible Camina-Zieschang decompositions of $\mathcal{D}$ are recorded in Tables \ref{smaug}, \ref{smeagol}, and \ref{sedici}.

In Section 4, we focus on the case where $\mathcal{D}$ is symmetric, and we filter the admissible cases provided in the aforementioned tables according to this property. The candidates are then listed in Table \ref{symm}. As the cases with $\lambda \leq 10$ are settled in \cite[Theorem 1]{MS} and \cite[Theorem 1.2]{Mo}, we may assume that $\lambda>10$, thus obtaining either a $2$-$(63,32,16)$ design or a $2$-$(64,28,12)$ design. The final part of this section is devoted to identify such $2$-design by means of group-theoretical methods as well as using the package $\texttt{DESIGN}$ \cite{Design} of \texttt{GAP} \cite{GAP}.

Finally, as mentioned above, Section 5 gathers various classification results on $2$-designs with specific numerical parameters, in order to determine $\mathcal{D}_{0}$, $\mathcal{D}_{1}$, and their automorphism groups, thereby completing Tables \ref{smaug}, \ref{smeagol}, and \ref{sedici}, and ultimately Table \ref{symm}.

\bigskip

\section{The two flag-transitive $2$-$(64,28,12)$ designs as in Line 7 of Table \ref{tavola1}}\label{Sec2}
The aim of this section is to provide a theoretic construction of the two non-isomorphic $%
2$-$(64,28,12)$ designs as in Line 7 of Table \ref{tavola1} admitting $2^{8}:PSL_{2}(7)$ as a flag-transitive,
point-imprimitive (full) automorphism group. These $2$-designs are not new. Indeed, by using \cite{GAP}, all $(64,28,12)$ difference sets were determined in \cite{OAG}. In particular, it was shown in \cite{OAG} that, only $259$ out of the $267$ groups of order $64$ admit a $(64,28,12)$ difference set. Here, we show that only $14$ of these groups admit two non-isomorphic $(64,28,12)$ difference sets whose development is flag-transitive and point-imprimitive $2$-design as in Line 7 of Table \ref{tavola1}. Our proof makes use of some geometry of the absolutely
irreducible $8$-dimensional representation of $PSL_{2}(7)$. 
  
\bigskip

In order to construct the examples, we need to recall the following useful facts about the absolutely
irreducible $8$-dimensional $\mathbb{F}_{2}$-representation of $PSL_{2}(7)$. 

\bigskip

The group $GO_{8}^{-}(2)$ acts naturally on $V=V_{8}(2)$, so let $Q$ be its
invariant quadratic form. From \cite[Table 8.53]{BHRD} we now that, $\Omega
_{8}^{-}(2)$ contains a unique conjugacy class of subgroups isomorphic to $%
G_{0}=PSL_{2}(7)$ that acts absolutely irreducibly on $V$. Moreover, $%
N_{GO_{8}^{-}(2)}(G_{0})=PGL_{2}(7)$.

Let $S$ be a Sylow $7$-subgroup $G_{0}$, then $S$ is also a Sylow $7$%
-subgroup of $GO_{8}^{-}(2)$. It can be deduced from \cite{At} that, $S$
preserves exactly two $1$-subspaces, say $\left\langle a_{1}\right\rangle $
and $\left\langle a_{2}\right\rangle $, and exactly two $3$-subspaces of $V$%
, say $V_{1}$ and $V_{2}$, and all these are singular with respect to $Q$.
Let $W_{1}=V_{1}^{\perp }$ ad $W_{2}=V_{2}^{\perp }$ be the orthogonal
complements of $V_{1}$ and $V_{2}$, respectively. \ Then $V_{1}\leq W_{1}$
and $V_{2}\leq W_{2}$ since both $V_{1}$ and $V_{2}$ are singular, and $\dim
W_{1}=\dim W_{2}=5$. Moreover, $W_{1}$ and $W_{2}$ are the unique $S$%
-invariant $5$-subspaces of $V$ since $S$ consists of isometries of $V$, and $%
V_{1}$ and $V_{2}$ are the unique $S$-invariant $3$-subspaces of $V$. Thus, $%
W_{1}=V_{1}\oplus \left\langle a_{1},a_{2}\right\rangle $ and $%
W_{2}=V_{2}\oplus \left\langle a_{1},a_{2}\right\rangle $ and $V=V_{1}\oplus
\left\langle a_{1},a_{2}\right\rangle \oplus V_{2}$. Also, $N_{G_{0}}(S)=S:C$
with $C$ cyclic of order $3$, preserves both $W_{1}$ and $W_{2}$. Therefore, 
$W_{1}^{G_{0}}$ and $W_{2}^{G_{0}}$ are the unique two (distinct) $2$%
-transitive $G_{0}$-orbits both of length $8$ on $5$-subspaces since $G$
consists of isometries of $V$ with respect to $Q$, and $G_{0}<\Omega
_{8}^{-}(2)$ acts irreducibly on $V$. Also, $N_{GO_{8}^{-}(2)}(G_{0})$ fuses 
$W_{1}^{G_{0}}$ and $W_{2}^{G_{0}}$. By using \cite{GAP}, one can see that

\begin{itemize}
\item[(i)] the intersection of any two distinct elements in $W_{1}^{G_{0}}$,
or $W_{2}^{G_{0}}$, is a $2$-space;

\item[(ii)] the intersection of any element in $W_{1}^{G_{0}}$ with any
element in $W_{2}^{G_{0}}$ is a $3$-subspace, unless they are stabilized by
the same Frobenius subgroup of $G_{0}$ of order $21$, in which case their
intersection is a $2$-space.  
\end{itemize}

\bigskip 

All the previously introduced symbols will have that fixed meaning
throughout this section.

\bigskip

\begin{lemma}
\label{Proex}Let $G=T:G_{0}$, where $T$ is the translation group of $V$ and $%
G_{0}=PSL_{2}(7)$, and let $\mathcal{P}=\left\{ W_{1}^{\beta }+x:\beta \in
G_{0},x\in V\right\} $ and $\Sigma =\Delta ^{G}$, where $\Delta
=V/W_{1}=\left\{ W_{1}+x:x\in V\right\} $. Then the following hold:
\end{lemma}

\begin{enumerate}
\item $G$ acts imprimitively on $\mathcal{P}$. In particular, $\Sigma $ is a 
$G$-invariant partition of $\mathcal{P}$ in $2^{3}$ classes of each of size $%
2^{3}$, and $\left\vert \mathcal{P}\right\vert =2^{6}$.

\item $G_{W_{1}}=T_{W_{1}}:(S:C)$. In particular, $G$ has rank $3$ and its
subdegrees are $1,7$ and $56$

\item Let $S^{g}$, with $g\in G_{0}$, be the other Sylow $7$-subgroup
normalized by $C$, and let $K=T_{W_{2}^{g}}:(S^{g}:C)$. Then $%
W_{1}^{gK},W_{1}^{K},(W_{1}+x_{0})^{K}$, where $W_{1}+x_{0}$ is some
suitable element of $\mathcal{P}$ fixed by $C$, are $K$-orbits of length $%
8,28,28$. The three $K$-orbits form a partition of $\mathcal{P}$.   

\item Set $B_{1}=W_{1}^{K}$ and $B_{2}=(W_{1}+x_{0})^{K}$. Then $%
G_{B_{1}}=G_{B_{2}}=K$ and $\left\vert B_{1}^{G}\right\vert =\left\vert
B_{2}^{G}\right\vert =64$.
\end{enumerate}

\begin{proof}
Since $W_{1}^{\beta }+x=W_{1}^{\beta \phi }$, where $\beta \in
G_{0}$ and $\phi :V\longrightarrow
V,u\longmapsto u+x$, it follows that $G$ acts transitively on $\mathcal{P}$.
Then $\Sigma =\Delta ^{G}$, where $\Delta
=V/W_{1}=\left\{ W_{1}+x:x\in V\right\} $, is a covering of $\mathcal{P}$. Assume that $%
\Delta \cap \Delta ^{\gamma }\neq \varnothing $ for some $\gamma \in G$.
Hence, there are $x,y \in V$ such that such that $\left( W_{1}+x\right)
^{\gamma }=W_{1}+y$. Now, $\gamma =\alpha \tau $ for some $\alpha \in G_{0}$
and $\tau \in T$. Hence, $\left( W_{1}+x\right) ^{\gamma }=W_{1}^{\alpha
}+x^{\alpha }+z$ for some $z\in V$. Therefore, one has%
\[
W_{1}^{\alpha }+x^{\alpha }+z=W_{1}+y\text{.}
\]%
Since $\alpha \in G_{0}$, it follows that $0\in W_{i}^{\alpha }$, and hence $%
x^{\alpha }+z+y\in W_{1}$. Then $W_{1}^{\alpha }=W_{1}$ and hence $\Delta
^{\gamma }=\Delta ^{\tau }=\Delta $ since $\Delta =V/W_{1}$. Therefore, $%
\Sigma $ is a $G$-invariant partition of $\mathcal{P}$ in classes each of
size $2^{3}$. Since $T$ stabilizes $\Delta $ and $T\vartriangleleft G$, it
follows that $T$ stabilizes any element of $\Sigma $. Moreover, $S:C$ fixes $%
W_{1}$ and hence stabilizes $\Delta $. Therefore $T:(S:C)$ stabilizes $%
\Delta $, and actually $G_{\Delta }=T:(S:C)$ since $T:(S:C)$ is maximal in $G
$. Therefore $\left\vert \Sigma \right\vert =2^{3}$, and hence $\left\vert 
\mathcal{P}\right\vert =\left\vert \Sigma \right\vert \left\vert \Delta
\right\vert =2^{6}$. This proves (1).

Let $H=T_{W_{1}}:(S:C)$, then $H\leq G_{W_{1}}$. Furthermore, $\left\vert
G_{W_{1}}\right\vert =2^{5}\cdot 3\cdot 7$ since $G$ acts transitively on $%
\mathcal{P}$, $\left\vert G\right\vert =2^{11}\cdot 3\cdot 7$ and $%
\left\vert \mathcal{P}\right\vert =2^{6}$. Thus, $G_{W_{1}}=H$. 

Clearly, $G_{W_{1}}$ preserves $\Delta \setminus \left\{ W_{1}\right\} $,
which has size $7$. Let $\varphi \in S$, $\varphi \neq 1$, be\ such that $%
\left( W_{1}+x\right) ^{\varphi }=W_{1}+x$ with $x\notin W_{1}$. Then $%
S=\left\langle \varphi \right\rangle $ fixes a non-zero vector lying in $%
W_{1}+x$ since $\left\vert W_{1}+x\right\vert =2^{5}$. Then $\left(
W_{1}+x\right) \cap W_{1}\neq \varnothing $, since the unique vectors of $V$
fixed by $S$ are those lying in $\left\langle a_{1},a_{2}\right\rangle
\subset W_{1}$ as we have at the beginning of this section that. So $%
W_{1}+x=W_{1}$, whereas $x\notin W_{1}$. Thus, $\Delta \setminus \left\{
W_{1}\right\} $ is a $G_{W_{1}}$-orbit of length $7$.

The actions of $G_{0}$ on $W_{1}^{G_{0}}$ and on $PG_{1}(8)$ are equivalent,
then there is $g\in G_{0}$ such that $W_{1}^{gC}=W_{1}^{g}$ since $C$ is
cyclic of order $3$. Let $\vartheta \in G_{W_{1}}$, such that $%
W_{1}^{g\vartheta }=W_{1}^{g}$. Then $\vartheta =\eta \varsigma \tau $ with $%
\eta \in C$, $\varsigma \in S$ and an element $\tau :V\longrightarrow
V,u\longmapsto u+w$ with $w\in W_{1}$, and hence $W_{1}^{g}=W_{1}^{g%
\varsigma \tau }$ since $W_{1}^{gC}=W_{1}^{g}$. Then $W_{1}^{g\varsigma
}+w=W_{1}^{g}$ and hence $w\in W_{1}^{g}$ and $W_{1}^{g\varsigma }=W_{1}^{g}$
since $0\in W_{1}^{g\varsigma }$. So $\varsigma =1$ since $\varsigma \in S$
and any non-trivial element of $S$ fixes a unique element of $W_{1}^{G_{0}}$%
, namely $W_{1}$, and $W_{1}^{g}\neq W_{1}$. Then $\tau \in
T_{W_{1}^{g}}\cap T_{W_{1}}=T_{W_{1}^{g}\cap W_{1}}$, and so $%
G_{W_{1},W_{1}^{g}}\leq T_{W_{1}^{g}\cap W_{1}}C$. Conversely, it is obvious to see that
that $T_{W_{1}^{g}\cap W_{1}}C\leq G_{W_{1},W_{1}^{g}}$. Thus, $%
G_{W_{1},W_{1}^{g}}=T_{W_{1}^{g}\cap W_{1}}C$ and hence $\left\vert
W_{1}^{gG_{W_{1}}}\right\vert =56$ since $\left\vert W_{1}^{g}\cap
W_{1}\right\vert =4$ by (i), being $W_{1}^{g}\neq W_{1}$. Thus, we have
proven that $G$ is point-imprimitive rank $3$ on $\mathcal{P}_{i}$ with
subdegrees $1,7$ and $56$. This proves (2).

Let $g$ be defined as in (2), then $S^{g}$ is the other Sylow $7$-subgroup
of $G_{0}$ normalized by $C$. Then $W_{1}^{gK}=W_{1}^{gT_{W_{2}}}$ since $%
K=T_{W_{2}^{g}}:(S^{g}:C)$ and  $S^{g}:C$ preserves $W_{1}^{g}$. Now, $%
W_{1}^{g\tau }=W_{1}^{g}$ for some $\tau \in T_{W_{2}^{g}}$ if and only if $%
\tau \in T_{W_{1}^{g}}\cap T_{W_{2}^{g}}=T_{W_{1}^{g}\cap W_{2}^{g}}$ and
hence $\left\vert W_{1}^{gK}\right\vert =8$ since\ $W_{1}^{g}\cap
W_{2}^{g}=\left\langle a_{1},a_{2}\right\rangle ^{g}$ or, more simply, by (ii).
Thus, $W_{1}^{gK}=V/W_{1}^{g}$ has length $8$.

Let $\psi \in K$ such that $W_{1}^{\psi }=W_{1}$. Hence, $\psi =\delta
\sigma \tau ^{\prime }$ with $\delta \in C$, $\sigma \in S^{g}$ and $\tau
^{\prime }\in T_{W_{2}^{g}}$. Therefore $W_{1}^{\sigma \tau }=W_{1}$, and
hence $W_{1}^{\sigma }+w^{\prime }=W_{1}$ for some $w^{\prime }\in W_{2}^{g}$%
. Then $w^{\prime }\in W_{1}$, and hence $W_{1}^{\sigma }=W_{1}$. So $\sigma =1$ since 
$\sigma \in S^{g}$, and hence $\tau \in T_{W_{1}}\cap
T_{W_{2}^{g}}=T_{W_{1}\cap W_{2}^{g}}$ and $\psi \in T_{W_{1}\cap W_{2}^{g}}C
$. Thus, $K_{W_{1}} \leq T_{W_{1}\cap W_{2}^{g}}C$. On the
other hand, $T_{W_{1}\cap W_{2}^{g}}C\leq K_{W_{1}}$. Hence,  $K_{W_{1}} = T_{W_{1}\cap W_{2}^{g}}C$ and $\left\vert
W_{1}^{K}\right\vert =28$.

Now, $\mathcal{P}_{i}\setminus (W_{1}^{K}\cup W_{1}^{gK})$ has size $28$ and
is a union of $K$-orbits and, in particular, there is at least an element of 
$\mathcal{P}_{i}\setminus (W_{1}^{K}\cup W_{1}^{gK})$ which is fixed by $C$.
Such an element is of the form $W_{1}+x_{0}$ for some $x_{0}\in V$ since $%
W_{1}^{gK}=V/W_{1}^{g}$. Let $\xi =\delta _{1}\sigma _{1}\tau _{1}$ with $%
\delta _{1}\in C$, $\sigma _{1}\in S^{g}$ and $\tau _{1}\in T_{W_{2}^{g}}$
such that $\left( W_{1}+x_{0}\right) ^{\xi }=W_{1}+x_{0}$. Then 
\[
W_{1}+x_{0}=\left( W_{1}+x_{0}\right) ^{\xi }=\left( W_{1}+x_{0}\right)
^{\sigma _{1}\tau _{1}}=W_{1}^{\sigma _{1}}+x_{0}+y_{0}
\]%
for some $y_{0}\in W_{2}^{g}$, being $x_{0}$ fixed by $C$. Then $%
W_{1}=W_{1}^{\sigma _{1}}+y_{0}$. Since $0\in W_{1}^{\sigma }$, it follows
that $y_{0}\in W_{1}$ and so $W_{1}^{\sigma }=W_{1}$. Then $\sigma _{1}=1$
since $\sigma _{1}\in S^{g}$, and so $y_{0}\in W_{1}$. Then $\tau ^{\prime
}\in T_{W_{1}}\cap T_{W_{2}^{g}}=T_{W_{1}\cap W_{2}^{g}}$. So $\left\vert
\left( W_{1}^{g}+x_{0}\right) ^{K}\right\vert =28$ since $%
K_{W_{1}+x_{0}}=T_{W_{1}\cap W_{2}^{g}}:C$ with $\left\vert W_{1}\cap
W_{2}^{g}\right\vert =8$ by (ii). Therefore, $%
W_{1}^{gK},W_{1}^{K},(W_{1}+x_{0})^{K}$ are $K$-orbits of length $8,28,28$, respectively, and form a partition of $\mathcal{P}$, which is (3).

Set $B_{1}=W_{1}^{K}$ and $B_{2}=(W_{1}+x_{0})^{K}$. The group $G_{B_{h}}$, $h=1,2$, contains $K$ and hence permutes transitively
the $7$ elements of $\Sigma $ intersecting $B_{h}$ in a non-empty. This
forces $G_{B_{h}}/(G_{B_{h}}\cap T)\cong S^{g}:C$ since%
\[
S:C\cong K/T_{W_{2}^{g}}=K/(K\cap T)\cong KT/T\leq G_{B_{h}}T/T\cong
G_{B_{h}}/(G_{B_{h}}\cap T)\leq PSL_{2}(7)
\]%
and a Frobenius group of order $21$ is maximal in $PSL_{2}(7)$. Hence, $G_{B_{h}}=\left( G_{B_{h}}\cap T\right) :(S^{g}:C)$ by \cite[%
Theorem 6.2.1(i)]{Go}. Moreover, $G_{B_{h}}$ preserve $W_{1}^{gK}=V/W_{1}^{g}
$, which is the remaining element of $\Sigma $ since $\left\vert \Sigma
\right\vert =8$. Therefore $G_{B_{1}}=G_{B_{2}}=G_{V/W_{1}^{g}}$. Since $%
T_{W_{2}}<G_{B_{h}}\cap T$, the actions of $S^{g}$ on $\left(
G_{B_{h}}\cap T\right) /T_{W_{2}}$, $V/W_{2}$ and on $V_{2}$ are all
equivalent. Also, the action of $S^{g}$ on $V_{2}$ is irreducible. Thus,
either $G_{B_{h}}=T:(S^{g}:C)$ or $G_{B_{h}}=K$. The former implies that $%
B_{h}$ contains each element of $\Sigma $ intersecting $B_{h}$ in a
non-empty set since $T$ acts transitively on each element of $\Sigma $. So $%
\left\vert B_{h}\right\vert =56$, which is a contradiction. Therefore, $%
G_{B_{h}}=K$ and so $\left\vert B_{h}^{G}\right\vert =64$, which is (4).
\end{proof}

\bigskip 

\begin{example}
\label{Ex1}The incidence structures $\mathcal{D}^{(h)}=(\mathcal{P}%
,B_{h}^{G})$, $h=1,2$, are two non-isomorphic symmetric $2$-$(64,28,12)$
designs admitting $G=2^{8}:PSL_{2}(7)$ as flag-transitive, point-imprimitive
automorphism group. In particular, $G$ is the full automorphism of $\mathcal{%
D}^{(h)}$, $h=1,2$.
\end{example}

\begin{proof}
Both the incidence structures $\mathcal{D}^{(h)}$ admit $G$ as a
flag-transitive, point-imprimitive automorphism group by their definition
and by Lemma \ref{Proex}(1). Moreover, they are symmetric $1$-design with
parameters $(v,k)=(64,28)$ by \cite[1.2.6]{Demb} and Lemma \ref{Proex}%
(1)(4). Let $x$ be any point on $B_{h}$, and let $\mathcal{O}_{1}$ and $%
\mathcal{O}_{2}$ be the $G_{x}$-orbits of length $7$ and $56$, respectively
(see Lemma \ref{Proex}(2)). Now, let $n_{1h}=\left\vert B_{h}\cap \mathcal{O}%
_{1}\right\vert $ and $\lambda _{1h}$ the number of blocks lying in $%
B_{h}^{G_{x}}$ and containing any fixed element of $\mathcal{O}_{1}$. The
integer $n_{2h}$ and $\lambda _{2h}$ are defined similarly. Clearly, $%
n_{1h}+n_{2h}=\left\vert B_{h}\setminus \left\{ x\right\} \right\vert =27$
with $1\leq n_{1h}\leq \left\vert \mathcal{O}_{1}\right\vert =7$ and $1\leq
n_{1h}\leq \left\vert \mathcal{O}_{1}\right\vert =56$.

Let $C$ be a cyclic subgroup of $G_{B_{h}}$. It follows from \cite{At} that, 
$C$ fixes exactly $4$ vectors in $V$ including $0$. Moreover, we deduce from
Lemma \ref{Proex}(3), that $C$ preserves both $W_{1}$ and $W_{1}+x_{0}$ and
any of these has size $2^{5}$. Thus $C$ fixes exactly $2$ vectors in $W_{1}$
and $2$ in $W_{1}+x_{0}$. Consequently, $W_{1}$ and $W_{1}+x_{0}$ are the
unique elements of $\mathcal{P}$ fixed by $C$. Now, $C<G_{B_{h}}$, and the
previous argument implies that, $C$ fixes a point on $B_{h}$ and acts
semiregularly on the remaining $27$ points of $B_{h}$. It follows that $n_{1h}\equiv 0%
\pmod{3}$ and $n_{2h}\equiv 0\pmod{3}$. Hence, $n_{1h}=3$ or $6$ since 
$1\leq n_{1h}\leq 7$.

Since $\left( \mathcal{O}_{1},B_{h}^{G}\right) $ and $\left( \mathcal{O}%
_{2},B_{h}^{G}\right) $ are two $1$-designs with by \cite[1.2.6]{Demb}, it
follows that $7\lambda _{1h}=28n_{1h}$ and $56\lambda _{2h}=28n_{2h}$ and so 
$n_{2h}=2\lambda _{2h}$. Then $n_{2h}$ is even, and hence $n_{1h}$ is odd
since $n_{1h}+n_{2h}=27$. Consequently, $n_{1h}=3$. Then $n_{2h}=24$ and so $%
\lambda _{1h}=4n_{1h}=12$ and $\lambda _{2h}=n_{2h}/2=12$. Thus $\mathcal{D}%
^{(h)}$ is flag-transitive, point-imprimitive symmetric $2$-$(64,28,12)$
design.

Let $X_{h}=Aut(\mathcal{D}^{(h)})$. Suppose that $X_{h}$ acts
point-primitively on $\mathcal{D}^{(h)}$, then $X_{h}$ acts point-$2$%
-homogeneously on $\mathcal{D}^{(h)}$ by \cite[Theorem 1.2]{TCZ} since $%
G\leq X_{h}$ and $G$ has order divisible by $7$. Actually, $X_{h}$ acts
point-$2$-transitively on $\mathcal{D}^{(h)}$ by \cite[Theorem 1]{Ka72}, and hence $%
\mathcal{D}^{(h)}\cong \mathcal{S}^{-}(3)$ and $X_{h}\cong 2^{6}:Sp_{6}(2)$
by \cite[Theorem]{Ka85} and \cite[Theorem 1]{Ka75}. Then $%
G_{B_{h}}\leq \left( X_{h}\right) _{B_{h}}\cong GO_{6}^{-}(2)$ by \cite[%
Corollary 3]{Ka75}, and we reach a contradiction since the order of $%
G_{B_{h}}$ is divisible by $7$, whereas the order of $GO_{6}^{-}(2)$ is not.
Thus $X_{h}$ acts point-imprimitively on $\mathcal{D}^{(h)}$, and therefore $%
X_{h}=G$ by Lemma \ref{ftinoporo}(3) since $%
PSL_{2}(7)<G\leq X_{h}$. 

Finally, assume that $\varphi $ is an isomorphism between $\mathcal{D}^{(1)}$
and $\mathcal{D}^{(2)}$. Then $\varphi $ normalizes $G$ since $G=Aut(%
\mathcal{D}^{(1)})=Aut(\mathcal{D}^{(2)})$. Then $\varphi $ normalizes $%
T=Soc(G)$, and hence $G\left\langle \varphi \right\rangle \cong
2^{6}:PGL_{2}(7)$ by \cite{At}. Moreover $B_{1}^{G\left\langle \varphi
\right\rangle }=B_{1}^{G}\cup B_{2}^{G}$, and so $G\left\langle \varphi
\right\rangle $ is a flag-transitive automorphism group of the $2$-$%
(64,56,24)$ design $(\mathcal{P}_{i},B_{1}^{G\left\langle \varphi
\right\rangle })$. Then $G\left\langle \varphi \right\rangle $ acts
point-primitively on by Theorem \ref{TheList}. Moreover, $G\left\langle \varphi
\right\rangle $ is of rank $3$ with subdegrees $1$, $7$ and $56$ since $G$
is rank $3$ with subdegrees $1$, $7$ and $56$ and $\left\vert G\left\langle
\varphi \right\rangle :G\right\vert =2$, but this contradicts \cite[Theorem]{Lieb}. Thus, $\mathcal{D}^{(1)}$ and $\mathcal{D}^{(2)}$ are
not isomorphic. This completes the proof.
\end{proof}

\bigskip 
For reader's convenience a system of generators of $G$ and a base block for the $2$-designs $\mathcal{D}^{(1)}$ and $\mathcal{D}^{(2)}$ is provided in Table \ref{smeagol} by using the package $\texttt{DESIGN}$ \cite{DifSet} of $\texttt{GAP}$ \cite{GAP}. The table also contains another way to construct such a $2$-designs as a development of McFarland $(64,28,12)$ difference sets of fourteen $2$-subgroups of $G$ (of order $64$) acting regularly on $\mathcal{P}$ (e.g. see \cite[Definition VI.1.5 and Theorem VI.1.6]{BLJ}). The computation of which groups of order $2^{6}$ admit a $(64,28,12)$ difference set whose development is $\mathcal{D}^{(1)}$ or $\mathcal{D}^{(2)}$ is carried out by using the package $\texttt{DIFSET}$ \cite{DifSet} of $\texttt{GAP}$ \cite{GAP}. These groups are denoted in Table \ref{smeagol} as in the \texttt{GAP} library \cite{GAP}.

 \bigskip

\begin{sidewaystable} 
\caption{system of generators of $G$, a base block for the the $2$-designs $\mathcal{D}^{(1)}$ and $\mathcal{D}^{(2)}$, and difference sets $2$-subgroups of $G$ (of order $64$) acting regularly on $\mathcal{P}$ for which $\mathcal{D}^{(1)}$ and $\mathcal{D}^{(2)}$ are developments.}
\label{frodo}
\tiny
{\tiny\setlength{\tabcolsep}{2pt}
\begin{tabular}{|rl|}
\hline
\tiny
& \textbf{System of generators of $G$}\\
\hline
$g_{1}=$&$(1,4)(2,3)(5,8)(6,7)(9,13)(10,14)(11,15)(12,16)(17,22)(18,21)(19,24)(20,23)(25,31)(26,32)(27,29)(28,30)(33,40)(34,39)(35,38)(36,37)(57,61)(58,62)(59,63)(60,64)$;\\
$g_{2}=$&$(1,3)(2,4)(5,7)(6,8)(17,23)(18,24)(19,21)(20,22)(25,31)(26,32)(27,29)(28,30)(33,40)(34,39)(35,38)(36,37)(41,44)(42,43)(45,48)(46,47)(49,50)(51,52)(53,54)(55,56)$;\\
$g_{3}=$&$(1,9,33)(2,16,37)(3,12,40)(4,13,36)(5,11,38)(6,14,34)(7,10,35)(8,15,39)(19,21,23)(20,22,24)(27,31,29)(28,32,30)(41,49,57)(42,55,62)(43,56,63)(44,50,60)$\\
&$(45,53,58)(46,51,61)(47,52,64)(48,54,59)$;\\
$g_{4}=$&$(1,57,2,61)(3,59,4,63)(5,64,6,60)(7,62,8,58)(9,51)(10,53)(11,50)(12,56)(13,54)(14,52)(15,55)(16,49)(17,31,19,29)(18,32,20,30)(21,27,23,25)(22,28,24,26)(33,46,37,41)$\\
&$(34,47,38,44)(35,45,39,42)(36,48,40,43)$\\
$g_{5}=$&$(1,4)(2,3)(5,8)(6,7)(17,24)(18,23)(19,22)(20,21)(33,37)(34,38)(35,39)(36,40)(41,44)(42,43)(45,48)(46,47)(49,51)(50,52)(53,55)(54,56)(57,60)(58,59)(61,64)(62,63)$;\\
$g_{6}=$&$(2,4,3)(5,6,7)(9,25,41)(10,30,48)(11,32,42)(12,27,47)(13,31,44)(14,28,45)(15,26,43)(16,29,46)(18,24,21)(20,22,23)(33,49,57)(34,55,59)(35,53,62)(36,51,64)(37,52,60)$\\
&$(38,54,58)(39,56,63)(40,50,61)$;\\
$g_{7}=$&$(17,20)(18,19)(21,24)(22,23)(25,29)(26,30)(27,31)(28,32)(33,40)(34,39)(35,38)(36,37)(41,46)(42,45)(43,48)(44,47)(49,52)(50,51)(53,56)(54,55)(57,64)(58,63)(59,62)(60,61)$;\\
$g_{8}=$&$(1,4)(2,3)(5,8)(6,7)(17,20)(18,19)(21,24)(22,23)(25,32)(26,31)(27,30)(28,29)(41,42)(43,44)(45,46)(47,48)(49,54)(50,53)(51,56)(52,55)(57,58)(59,60)(61,62)(63,64)$;\\
$g_{9}=$&$(1,4)(2,3)(5,8)(6,7)(9,12)(10,11)(13,16)(14,15)(25,31)(26,32)(27,29)(28,30)(33,37)(34,38)(35,39)(36,40)(41,46)(42,45)(43,48)(44,47)(49,50)(51,52)$\\
& $(53,54)(55,56)(57,64)(58,63)(59,62)(60,61)$.\\
\hline
\hline
& \textbf{Base blocks of $\mathcal{D}^{(1)}$ and $\mathcal{D}^{(2)}$ }\\
\hline
 $B_{1}=$&$\{ 9, 11, 13, 15, 17, 20, 22, 23, 25, 26, 31, 32, 33, 35, 38, 40, 41, 42, 43, 44, 49, 50, 53, 54, 57, 58, 61, 62\}$\\
 $B{2}=$&$\{10, 12, 14, 16, 18, 19, 21, 24, 27, 28, 29, 30, 34, 36, 37, 39, 45, 46, 47, 48, 51, 52, 55, 56, 59, 60, 63, 64\}$\\
 \hline
 \hline
 & \textbf{\texttt{SmallGroups(64,i)} having $B_{1}$ and $B_{2}$ as difference sets and $\mathcal{D}^{(1)}$ and $\mathcal{D}^{(2)}$ as their developments}\\
 \hline
 $i=$ &  $18, 23, 32, 33, 34, 35, 90, 102, 134, 136, 138, 139, 215, 256$\\
 \hline
\end{tabular}}
\end{sidewaystable} 
\normalsize
\section{Preliminaries}\label{Sec3}
In this section, we provide some useful results in both design theory and group theory. First, we give the following theorems which allow us to reduce the analysis of the flag-transitive automorphism group $G$ of a $2$-$(v,k,\lambda)$ design $\mathcal{D}$.

\bigskip

\begin{lemma}\label{L0}
 The parameters $v$, $b$, $k$, $r$, $\lambda$ of $\mathcal{D}$ satisfy $vr=bk$, $\lambda (v-1)=r(k-1)$, and $k \leq r$.   
\end{lemma}

\bigskip

For a proof see \cite[1.3.8 and 2.1.5]{Demb}.

\bigskip

\begin{theorem}[Camina-Zieschang]\label{CamZie}
Let $\mathcal{D}=(\mathcal{P}, \mathcal{B})$ be a $2$-$(v,k,\lambda)$ design admitting a
flag-transitive, point-imprimitive automorphism group $G$ preserving a nontrivial 
partition $\Sigma $ of $\mathcal{P}$ with $v_1$ classes of size $v_0$. Then $v=v_{0}v_{1}$ and the following hold:
\begin{enumerate}
    \item[(1)] There is a constant $k_0 \geq 2$ such that $\left\vert B\cap \Delta \right\vert =0$ or $%
k_0 $ for each $B\in \mathcal{B}$ and $\Delta \in \Sigma $. The parameter $k_0 $ divides $k$. Moreover,
\begin{equation}\label{rel1}
\frac{v-1}{k-1}=\frac{v_{0}-1}{k_{0}-1}\text{,}    
\end{equation}
and the following hold:
\begin{enumerate}
\item[(i)] For each $\Delta \in \Sigma $, let $\mathcal{B}_{\Delta }=\left\{ B\cap \Delta
\neq \varnothing :B\in \mathcal{B}\right\} $. Then the set $$\mathcal{R}_{\Delta}=\{(B,C) \in \mathcal{B}_{\Delta } \times \mathcal{B}_{\Delta }:  B\cap \Delta =\C\cap \Delta \}$$ is an equivalence relation on $\mathcal{B}_{\Delta }$ with each equivalence class of size $\mu$.
\item[(ii)] The incidence structure $\mathcal{D}_{\Delta }=(\Delta ,%
\mathcal{B}_{\Delta })$ is either a symmetric $1$-design with $k_{0}=v_{0}-1$, or a $2$-$(v_0,k_0 ,\lambda
_{0})$ design with $\lambda_{0}=\frac{\lambda}{\theta}$.
\item[(iii)] The group $G_{\Delta}^{\Delta}$ acts flag-transitively on $\mathcal{D}_{\Delta }$.
\end{enumerate}
\item[(2)] For each block $B$ of $\mathcal{D}$ the set $B(\Sigma)=\{\Delta \in \Sigma: B \cap \Delta \neq \varnothing\}$ has a constant size $k_{1}=\frac{k}{k_{0}}$. Moreover, 
\begin{equation}\label{rel2}
\frac{v_{1}-1}{k_{1}-1}=\frac{k_{0}(v_{0}-1)}{v_{0}(k_{0}-1)}    
\end{equation}
and the following hold:
\begin{enumerate}
\item[(i)] The set $$\mathcal{R}=\{(C,C^{\prime}) \in \mathcal{B} \times \mathcal{B}:C(\Sigma)=C^{\prime}(\Sigma)\}$$ is an equivalence relation on $\mathcal{B}$ with each class of size $\mu$;
\item[(ii)] Let $\mathcal{B}^{\Sigma}$ be the quotient set defined by $\mathcal{R}$, and for any block $C$ of $\mathcal{D}$ denote by $C^{\Sigma}$ the $\mathcal{R}$-equivalence class containing $C$. Then the incidence structure $\mathcal{D}^{\Sigma}=\left(\Sigma, \mathcal{B}^{\Sigma}, \mathcal{I} \right)$ with $\mathcal{I}=\{(\Delta,C^{\Sigma}) \in \Sigma \times \mathcal{B}^{\Sigma}: \Delta \in C(\Sigma) \}$ is either a symmetric $1$-design with $k_{1}=v_{1}-1$, or a $2$-$(v_{1},k_{1} ,\lambda_{1})$ design with $\lambda_{1}=\frac{v_{0}^{2}\lambda}{k_{0}^{2}\mu}$.
\item[(iii)] The group $G^{\Sigma}$ acts flag-transitively on $\mathcal{D}^{\Sigma}$.
\end{enumerate}
\end{enumerate}
\end{theorem}

\bigskip
For a proof see \cite[Propositions 2.1 and 2.3]{CZ}.

\bigskip

We refer to $\mathcal{D}^{\Sigma}$ simply as $\mathcal{D}_{1}$. Moreover, as mentioned in the introduction, the designs corresponding to distinct classes $\Delta,\Delta'\in\Sigma$ are isomorphic under elements of $G$ mapping $\Delta$ to $\Delta'$, we refer to $\mathcal{D}_{\Delta }$ as $\mathcal{D}_{0}$. The parameters of $\mathcal{D}_{0}$ and of $\mathcal{D}_{1}$ will be indexed by $0$ and $1$, respectively. Hence, the conclusions of Lemma \ref{L0} hold for $\mathcal{D}_{i}$ when this one is a $2$-$(v_{i},k_{i},\lambda_{i})$ design. That is, $v_{i}r_{i}=b_{i}k_{i}$, $\lambda (v_{i}-1)=r_{i}(k-1)$, and $k_{i} \leq r_{i}$, where $r_{i}$ and $b_{i}$ are the replication number and the number of blocks of $\mathcal{D}_{i}$, respectively.

\bigskip

\begin{lemma}\label{bounds}
$v_{0}>k_{0}\geq 2$. Moreover, one of the following holds:

\begin{enumerate}
\item $k_{0}=2$, $v=(v_{0}-1)(2k_{1}-1)+1$ and $G_{\Delta }^{\Delta }$ is
point-$2$-transitive on $\mathcal{D}_{0}$;

\item $3\leq k_{0}\leq v_{0}-2$ and $\mathcal{D}_{0}$ is a $2$-design;

\item $3\leq k_{0}=v_{0}-1$, $\mathcal{D}_{0}$ is a $1$-design, $%
k=t(v_{0}-2)+1$ and $v=t(v_{0}-1)+1$ for some $t\geq 2$.
\end{enumerate}
\end{lemma}

\begin{proof}
Let $\Delta \in \Sigma $. By Theorem \ref{CamZie}(1), $k_0\geq 2$. Suppose $v_0=k_0$ and let $x_0,y_0\in\Delta$ be two distinct points, and $z$ be any point of $\mathcal{D}$ not in $\Delta$. Let $B\in \mathcal{B}(x_0,y_0)$ and $B^\prime\in \mathcal{B}(x_0,z)$. Since $|B^\prime\cap\Delta|=|B\cap\Delta|=k_0=v_0$, we have that $y_0\in B^\prime$. Hence, $\mathcal{B}(x_0,z)=\mathcal{B}(x_0,y_0)$ for any point of $z$ of $\mathcal{D}$ not in $\Delta$. Thus, $r=|\mathcal{B}(x_0,y_0)|=\lambda$, which is a contradiction. Therefore, $v_0>k_0$. 



Now, (1) and (2) immediately follow from (\ref{rel1}) and Theorem \ref{CamZie}(1.ii). Finally, assume that $3\leq k_{0}=v_{0}-1$. Then $%
(v-1)(v_{0}-2)=(v_{0}-1)(k-1)$ again from from (\ref{rel1}), and hence $%
k=t(v_{0}-2)+1$ and $v=t(v_{0}-1)+1$ for some $t\geq 1$. Actually, $t\geq 2$
since $v_{0}<v$. This proves (3).
\end{proof}

\bigskip

\begin{proposition}\label{k2orTrivial}
If either $k_{0}=2$ or $3\leq k_{0}=v_{0}-1$, then parameters $\mathcal{D}%
_{0},\mathcal{D}%
_{1}$ and $\mathcal{D}$ and the possibilities for $G_{\Delta }^{\Delta }$ and $G^{\Sigma}$ are as in Table \ref{smaug}.
\begin{sidewaystable} 
\caption{Admissible parameters for $\mathcal{D}%
_{0}, \mathcal{D}%
_{1}, \mathcal{D}%
$ and $G_{\Delta }^{\Delta }$ and $G^{\Sigma
}$ for either $k_{0}=2$ or $3\leq k_{0}=v_{0}-1$.}
\label{smaug}
\tiny
{\tiny\setlength{\tabcolsep}{2pt}
	\begin{tabular}{|c|ccccc|c|c|ccccc|c|ccccc|c|c|}
		\hline
		Line & $v_{0}$ & $k_{0}$ & $\lambda _{0}$ & $r_{0}$ & $b_{0}$ & $\theta$ &  $G_{\Delta
		}^{\Delta }$ & $v_{1}$ & $k_{1}$ & $\lambda _{1}$ & $r_{1}$ & $b_{1}$ & $%
		G^{\Sigma }$ & $v$ & $k$ & $\lambda $ & $r$ & $b$ & Conditions on $\mu $ & $%
		\mu _{S}$ \\
		\hline
		\hline
		1 & $3$ & $2$ & $1$ & $2$ & $3$ & $4\mu /3$ & $S_{3}$ & $5$ & $4$ & $3$ & $4$ & $5$ & $%
		AGL_{1}(5),A_{5},S_{5}$ & $15$ & $8$ & $4\mu /3$ & $8\mu /3$ & $5\mu $ & $\mu \equiv 0%
		\pmod{3}$ & $3$ \\ 
		2 & $3$ & $2$ & $1$ & $2$ & $3$ & $28\mu /3$ &  & $9$ & $7$ & $21$ & $28$ & $36$ & $%
		PSL_{2}(8),P\Gamma L_{2}(8),A_{9},S_{9}$ & $27$ & $14$ & $28\mu /3$ & $%
		56\mu /3$ & $36\mu $ & $\mu \equiv 0\pmod{3}$ & - \\ 
		3 & $3$ & $2$ & $1$ & $2$ & $3$ & $220\mu /3$ &  & $13$ & $10$ & $165$ & $220$ & $286$ & $%
		A_{13},S_{13}$ & $39$ & $20$ & $220\mu /3$ & $440\mu /3$ & $286\mu $ & $\mu
		\equiv 0\pmod{3}$ & - \\ 
		4 & $3$ & $2$ & $1$ & $2$ & $3$ & $1820\mu /3$ &  & $17$ & $13$ & $1365$ & $1820$ & $2380$
		& $A_{17},S_{17}$ & $51$ & $26$ & $1820\mu /3$ & $3640\mu /3$ & $2380\mu $ & 
		$\mu \equiv 0\pmod{3}$ & - \\ 
		5 & $3$ & $2$ & $1$ & $2$ & $3$ & $16\mu /3$ &  & $21$ & $16$ & $12$ & $16$ & $21$ & $%
		PSL_{3}(4):\epsilon $ with $\epsilon \mid 6$ & $63$ & $32$ & $16\mu /3$ & $%
		32\mu/3 $ & $21\mu$ & $\mu \equiv 0\pmod{3}$ & $3$ \\ 
		6 & $3$ & $2$ & $1$ & $2$ & $3$ & $5168\mu $ &  & $21$ & $16$ & $11628$ & $15504$ & $20349
		$ & $A_{21},S_{21}$ & $63$ & $32
		$ & $5168\mu $ & $10336\mu $ & $20349\mu $ &  & - \\ 
		7 & $3$ & $2$ & $1$ & $2$ & $3$ & $134596\mu /3$ &  & $25$ & $19$ & $100947$ & $134596$ & $%
		177100$ & $A_{25},S_{25}$ & $75$ & $38$ & $134596\mu /3$ & $269192\mu /3$ & $%
		177100\mu $ & $\mu \equiv 0\pmod{3}$ & - \\ 
		8 & $3$ & $2$ & $1$ & $2$ & $3$ & $394680\mu $ &  & $29$ & $22$ & $888030$ & $1184040$ & $%
		1560780$ & $A_{29},S_{29}$ & $87$ & $44$ & $394680\mu $ & $789360\mu $ & $%
		1560780\mu $ &  & - \\ 
		9 & $3$ & $2$ & $1$ & $2$ & $3$ & $3506100\mu $ &  & $33$ & $25$ & $7888725$ & $10518300$ & $%
		13884156$ & $A_{33},S_{33}$ & $99$ & $50$ & $3506100\mu $ & $7012200\mu $ & $%
		13884156\mu $ &  & - \\ 
		10 & $4$ & $2$ & $1$ & $3$ & $6$ & $\mu /2$ & $A_{4},S_{4}$ & $4$ & $3$ & $2$ & $3$ & $4
		$ & $A_{4},S_{4}$ & $16$ & $6$ & $\mu /2$ & $3\mu /2$ & $4\mu $ & $\mu
		\equiv 0\pmod{2}$ & $4$ \\ 
		11 & $4$ & $2$ & $1$ & $3$ & $6$ & $5\mu /2$ &  & $7$ & $5$ & $10$ & $15$ & $21$ & $%
		A_{7},S_{7}$ & $28$ & $10$ & $5\mu /2$ & $15\mu /2$ & $21\mu $ & $\mu \equiv
		0\pmod{2}$ & - \\ 
		12 & $4$ & $2$ & $1$ & $3$ & $6$ & $14\mu $ &  & $10$ & $7$ & $56$ & $84$ & $120$ & $%
		A_{10},S_{10}$ & $40$ & $14$ & $14\mu $ & $42\mu $ & $120\mu $ &  & - \\ 
		13 & $4$ & $2$ & $1$ & $3$ & $6$ & $3\mu /2$ &  & $13$ & $9$ & $6$ & $9$ & $13$
		& $PSL_{3}(3)$ & $52$ & $18$ & $3\mu /2$ & $9\mu /2$ & $%
		13\mu $ & $\mu \equiv 0\pmod{2}$ & $4$ \\ 
		14 & $4$ & $2$ & $1$ & $3$ & $6$ & $165\mu /2$ &  & $13$ & $9$ & $330$ & $495$ & $715$
		& $A_{13},S_{13}$ & $52$ & $18$ & $165\mu /2$ & $495\mu /2$ & $%
		715\mu $ & $\mu \equiv 0\pmod{2}$ & - \\ 
		15 & $4$ & $2$ & $1$ & $3$ & $6$ & $1001\mu /2$ &  & $16$ & $11$ & $2002$ & $3003$ & $4368$
		& $A_{16},S_{16}$ & $64$ & $22$ & $1001\mu /2$ & $3003\mu /2$ & $4368\mu $ & 
		$\mu \equiv 0\pmod{2}$ & - \\ 
		16 & $4$ & $2$ & $1$ & $3$ & $6$ & $3094\mu $ &  & $19$ & $13$ & $12376$ & $18564$ & $%
		27132$ & $A_{19},S_{19}$ & $76$ & $26$ & $3094\mu $ & $9282\mu $ & $27132\mu 
		$ &  & - \\ 
		17 & $4$ & $2$ & $1$ & $3$ & $6$ & $19380\mu $ &  & $22$ & $15$ & $77520$ & $116280$ & $%
		170544$ & $A_{22},S_{22}$ & $88$ & $30$ & $19380\mu $ & $58140\mu $ & 
		$170544\mu $ &  & - \\
		18 & $4$ & $2$ & $1$ & $3$ & $6$ & $20\mu$  &  & $22$ & $15$ & $80$ & $120$ & $%
		176$ & $M_{22}$ & $88$ & $30$ & $20\mu$  & $60\mu$ & 
		$176 \mu$ &  & - \\
		19 & $4$ & $2$ & $1$ & $3$ & $6$ & $40\mu$ &  & $22$ & $15$ & $160$ & $240$ & $%
		352$ & $M_{22}:2$ & $88$ & $30$ & $40\mu$ & $120\mu$ & 
		$352\mu$ &  & - \\
		20 & $4$ & $2$ & $1$ & $3$ & $6$ & $140\mu$ &  & $22$ & $15$ & $560$ & $840$ & $%
		1232$ & $M_{22}, M_{22}:2, A_{22},S_{22}$ & $88$ & $30$ & $140\mu$ & $420\mu$ & 
		$1232\mu$ &  & - \\
		21 & $5$ & $2$ & $1$ & $4$ & $10$ & $4\mu /5$ & $AGL_{1}(5),A_{5},S_{5}$ & $9$ & $6$ & $%
		5$ & $8$ & $12$ & $G^{\Sigma} \leq AGL_{2}(3)$ & $45$ & $12$
		& $4\mu /5$ & $16\mu /5$ & $12\mu $ & $\mu \equiv 0\pmod{5}$ & - \\
        22 & $5$ & $2$ & $1$ & $4$ & $10$ & $24\mu /5$ & $AGL_{1}(5),A_{5},S_{5}$ & $9$ & $6$ & $%
		30$ & $48$ & $72$ &$ASL_2(3), AGL_2(3)$ & $45$ & $12$
		& $24\mu /5$ & $96\mu /5$ & $72\mu $ & $\mu \equiv 0\pmod{5}$ & - \\
		23 & $5$ & $2$ & $1$ & $4$ & $10$ & $28\mu /5$ & $AGL_{1}(5),A_{5},S_{5}$ & $9$ & $6$ & $%
		35$ & $56$ & $84$ & $PSL_{2}(8),P\Gamma L_{2}(8)$,$A_{9},S_{9}$ & $45$ & $12
		$ & $28\mu /5$ & $112\mu /5$ & $84\mu $ & $\mu \equiv 0\pmod{5}$ & - \\
		24 & $5$ & $2$ & $1$ & $4$ & $10$ & $4004\mu /5$ &  & $17$ & $11$ & $5005$ & $8008$ & $12376
		$ & $A_{17},S_{17}$ & $85$ & $22$ & $4004\mu /5$ & $16016\mu /5$ & $12376\mu$ & 
		$\mu \equiv 0\pmod{5}$ & - \\ 
		25 & $6$ & $2$ & $1$ & $5$ & $15$ & $2\mu /3$ & $A_{5},S_{5},A_{6},S_{6}$ & $6$ & $4$ & $%
		6$ & $10$ & $15$ & $S_{5},A_{6},S_{6},$ & $36$ & $8$ & $2\mu /3$ & $%
		10\mu /3$ & $15\mu $ & $\mu \equiv 0\pmod{3}$ & - \\ 
		26 & $6$ & $2$ & $1$ & $5$ & $15$ & $14\mu $ &  & $11$ & $7$ & $126$ & $%
		210$ & $330$ & $A_{11},S_{11}$ & $66$ & $14$ & $14\mu $ & $70\mu $ & $330\mu 
		$ &  & - \\ 
		27 & $6$ & $2$ & $1$ & $5$ & $15$ & $1001\mu /3$ &  & $16$ & $10$ & $3003$ & $5005$ & $8008$
		& $A_{16},S_{16}$ & $96$ & $20$ & $1001\mu /3$ & $5005\mu /3$ & $8008\mu $ & 
		$\mu \equiv 0\pmod{3}$ & - \\ 
		28 & $7$ & $2$ & $1$ & $6$ & $21$ & $24\mu/7$ & $AGL_{1}(7),PSL_{2}(7),A_{7},S_{7}$ & $13
		$ & $8$ & $42$ & $72$ & $117$ & $PSL_{3}(3)$ & $91$ & $16$ & $24\mu/7$ & 
		$144\mu/7$ & $117\mu$ & $\mu \equiv 0 \pmod{7}$ & - \\ 
		29 & $7$ & $2$ & $1$ & $6$ & $21$ & $264\mu/7 $& $AGL_{1}(7),PSL_{2}(7),A_{7},S_{7}$ & $13
		$ & $8$ & $462$ & $792$ & $1287$ & $A_{13},S_{13}$ & $91$ & $16$ & $264\mu/7 $& $1584\mu/7$ & $1287\mu$ & $\mu \equiv 0 \pmod{7}$ & - \\ 
		30 & $8$ & $2$ & $1$ & $7$ & $28$ & $5\mu /4$ & $AGL_{1}(8),A\Gamma
		L_{1}(8),AGL_{3}(2)$ & $8$ & $5$ & $20$ & $35$ & $56$ & $%
		A_{8},S_{8}$ & $64$ & $10$ & $5\mu /4$ & $35\mu /4$ & $56\mu $ & $\mu \equiv
		0\pmod{4}$ & - \\ 
		& & & & & & & $PSL_{2}(7),PGL_{2}(7),A_{8},S_{8}$ & & & & & & & & & & & & & \\
		31 & $4$ & $3$ & $2$ & $3$ & $4$ & $9\mu /4$ & $A_{4},S_{4}$ & $10$ & $9$ & $8$ & $9$ & $%
		10$ & $PSL_{2}(9),PGL_{2}(9),P \Sigma L_{2}(9)$ & $40$ & $27$ & $9\mu /2$ & $27\mu /4$ & $%
		10\mu $ & $\mu \equiv 0\pmod{4}$ & $4$ \\ 
		& & & & & & & & & & & & & $M_{10},P\Gamma L_{2}(9),A_{10},S_{10}$ & & & & & & & \\
		32 & $4$ & $3$ & $2$ & $3$ & $4$ & $153\mu /4$ & $A_{4},S_{4}$ & $19$ & $17$ & $136$ & $153
		$ & $171$ & $A_{19},S_{19}$ & $76$ & $51$ & $153\mu /2$ & $459\mu /4$ & $%
		171\mu $ & $\mu \equiv 0\pmod{4}$ & - \\ 
		33 & $5$ & $4$ & $3$ & $4$ & $5$ & $48\mu /15$ & $AGL_{1}(5),A_{5},S_{5}$ & $17$ & $16$ & $15$ & $16$
		& $17$ & $AGL_{1}(17),A_{17},S_{17}$ & $85$ & $64$ & $48\mu /5$ & $64\mu /5$ & $17\mu $ & $\mu
		\equiv 0\pmod{15}$ & -\\
		& & & & & & & & & & & & & $PSL_{2}(2^{4}):2^{\varepsilon },0\leq \varepsilon \leq
		2$ & & & & & & & \\
		\hline 
\end{tabular}}
\end{sidewaystable}
\end{proposition}

\begin{proof} 
If $k_{0}=2$, then $v=(v_{0}-1)(2k_{1}-1)+1=(2k_{1}-1)v_{0}-2\left(
k_{1}-1\right) $. As $v_{0}$ divides $v$, we obtain $av_{0}=2\left(
k_{1}-1\right) $ for some $a\geq 1$. Therefore,%
\begin{equation}
v_{1}=a\left( v_{0}-1\right) +1\text{, }k_{1}=\frac{av_{0}}{2}+1\text{, }%
v=  v_{0}\left( a\left( v_{0}-1\right) +1\right) \text{ and }%
k=av_{0}+2\text{.}  \label{eq1}
\end{equation}%
Since $v_0^2\leq v<100$, it follows that $3\leq v_{0}\leq 9$. From this, since $v=v_0v_1$, we have that $v_1\leq 33$, and hence $1\leq a\leq 16$.
Substituting these values in (\ref{eq1}), and bearing in mind that $%
G_{\Delta }^{\Delta }$ is point-$2$-transitive on $\mathcal{D}_{0}$, one
obtains the admissible vales for $v_{0},k_{0},v_{1},k_{1},v,k$, and $%
G_{\Delta }^{\Delta }$ by \cite[Table B.4]{DM}. Now, exploiting that $k_{0}=2$, we see that either $(\lambda_{0},r_{0},b_{0})=(1,v_{0}-1,v_{0}(v_{0}-1)/2)$. Therefore, we obtain Columns 2--6, 8, 15--16 in Table \ref{smaug}. Now, for each admissible, previously computed, pair $(v_{1},k_{1})$, we determined the all the corresponding flag-transitive $2$-$(v_{1},k_{1},\lambda_{1})$-designs in the Appendix, and we use these information to compute all the admissible pairs $(\mathcal{D}_{1},G^{\Sigma})$. This allows us to obtain $\lambda_{1}$, $r_{1}$, $b_{1}$ and $G^{\Sigma}$, and hence Columns 11--14 in Table \ref{smaug}. At this point,
\begin{equation}
\lambda =\frac{k_{0}^{2}\cdot \mu}{v_{0}^{2}\cdot \lambda_{1}}\text{, } \quad %
r=\frac{k_{0}\cdot \mu}{v_{0}\cdot r_{1}}\text{, } \quad %
b=\frac{v \cdot r}{k}\quad \text{ and } \quad
\theta=\frac{\lambda}{\lambda_{0}}\text{.} \label{equaz1}
\end{equation}%
are obtained by Theorem \ref{CamZie}(2).(i)--(ii). Therefore, we obtain Columns 7, 17--21 in Table \ref{smaug}. Note that $\mu_{S}$ is the value of $\mu$ for which $\mathcal{D}$ is symmetric.

If $3\leq k_{0}=v_{0}-1$, $\mathcal{D}_{0}$ is a $1$-design, $k=t(v_{0}-2)+1$
and $v=t(v_{0}-1)+1$ for some $t\geq 2$ by (\ref{rel1}). Since $r=\frac{\left(
v_{0}-1\right) \lambda }{v_{0}-2}$, it follows that $\lambda =m(v_{0}-2)$, hence $r=m(v_{0}-1)$, for some $m\geq 1$. Since $k=t(v_{0}-2)+1$ and $%
v=t(v_{0}-1)+1$ for some $t\geq 2$, and $k_{0}\mid k$, with $k_{0}=v_{0}-1$
and $v_{0}\mid v$ one has $v_{0}-1\mid t-1$ and $v_{0}\mid t-1$. Therefore, $%
t=\ell v_{0}(v_{0}-1)\,+1$\ for some $\ell \geq 1$. Now, $kb=vr$ implies%
\begin{equation*}    
b\left[ t(v_{0}-2)+1\right] =\left[ t(v_{0}-1)+1\right] \left(
v_{0}-1\right) m=\left[ t(v_{0}-2)+1\right] \left( v_{0}-1\right) m+t\left(
v_{0}-1\right) m 
\end{equation*}
and hence $t(v_{0}-2)+1\mid t\left( v_{0}-1\right) m$, and so $\left[ \ell
v_{0}(v_{0}-1)\,+1\right] (v_{0}-2)+1\mid \left( v_{0}-1\right) m$ since $%
t=\ell v_{0}(v_{0}-1)\,+1$\ for some $\ell \geq 1$. Thus $m=\ell ^{\prime }%
\left[ \ell v_{0}\,(v_{0}-2)+1\right] \,$\ for $\ell ^{\prime }\geq 1$.%
\begin{eqnarray*}
r &=&\left( v_{0}-1\right) \ell ^{\prime }\left[ \ell v_{0}\,(v_{0}-2)+1%
\right] \\
\lambda &=&\left( v_{0}-2\right) \ell ^{\prime }\left[ \ell
v_{0}\,(v_{0}-2)+1\right] \\
v &=&\ell v_{0}(v_{0}-1)^{2}+v_{0} \\
k &=&\ell v_{0}(v_{0}-1)(v_{0}-2)+v_{0}-1
\end{eqnarray*}%
Now, $v<100$ and $v_{0}>k_{0}=3$ lead to $(v_{0},\ell )=(4,1),(4,2),(5,1)$.
Therefore, $(v_{0},k_{0})=(4,3),(5,4)$. Then $\mathcal{D}_{0}$ is a complete $2$-design and hence $\lambda _{0}=2,3$ respectively.
Furthermore, $v_{1}=\ell (v_{0}-1)^{2}+1$ and $k_{1}=\ell v_{0}(v_{0}-2)+1$
imply $(v_{1},k_{1})=(10,9),(19,17)$ or $(17,16)$. Then $\mathcal{D}_{1}$ is a symmetric complete $2$-design and $G^{\Sigma }$ is known by Lemma \ref{LA53}. Finally, we use (\ref{equaz1}) to determine the remaining parameters of $\mathcal{D}$.
\end{proof}

\bigskip

\begin{proposition}\label{3mink0minv0-2}
If $3\leq k_{0}\leq v_{0}-2$, then the parameters of $\mathcal{D}%
_{0},\mathcal{D}%
_{1}$ and $\mathcal{D}$ and the possibilities for $G_{\Delta }^{\Delta }$ and $G^{\Sigma}$ are as in Table \ref{smeagol} or $\ref{sedici}$ according to whether $(v_{0},k_{0})$ is not or is $(16,4)$, respectively.
\begin{sidewaystable} 
\caption{Admissible parameters for $\mathcal{D}%
_{0}, \mathcal{D}%
_{1}, \mathcal{D}%
$ and $G_{\Delta }^{\Delta }$ and $G^{\Sigma
}$ for $3\leq k_{0}\leq
v_{0}-2$.}
\label{smeagol}
\tiny
{\tiny\setlength{\tabcolsep}{2pt}
\begin{tabular}{|c|ccccc|c|c|ccccc|c|ccccc|c|c|}
\hline
Line & $v_{0}$ & $k_{0}$ & $\lambda _{0}$ & $r_{0}$ & $b_{0}$ & $\theta $ & $G_{\Delta
}^{\Delta }$ & $v_{1}$ & $k_{1}$ & $\lambda _{1}$ & $r_{1}$ & $b_{1}$ & $%
G^{\Sigma }$ & $v$ & $k$ & $\lambda $ & $r$ & $b$ & Conditions on $\mu $ & $%
\mu _{S}$ \\ 
\hline
\hline
1 & $5$ & $3$ & $3$ & $6$ & $10$ & $3\mu /5$ & $A_{5},S_{5}$ & $7$ & $6$ & $5$ & $6$ & $7
$ & $AGL_{1}(7),PSL_{2}(7),A_{7},S_{7}$ & $35$ & $18$ & $9\mu /5$ & $18\mu /5$ & $%
 7\mu $ & $\mu \equiv 0\pmod{5}$ & $5$ \\ 
2 & $5$ & $3$ & $3$ & $6$ & $10$ & $33\mu /5$ & $A_{5},S_{5}$ & $13$ & $11$ & $55$ & $66$
& $78$ & $A_{13},S_{13}$ & $65$ & $33$ & $99\mu /5$ & $  198\mu /5$
& $ 78\mu $ & $\mu \equiv 0\pmod{5}$ & - \\ 
3& $5$ & $3$ & $3$ & $6$ & $10$ & $408\mu/5$ & $A_{5},S_{5}$ & $19$ & $16$ & $680$ & $816$
& $969$ & $A_{19},S_{19}$ & $95$ & $48$ & $1224\mu/5$ & $2448\mu/5$
& $969\mu$ & $\mu \equiv 0\pmod{5}$ & - \\ 
4& $6$ & $3$ & $2$ & $5$ & $10$ & $9\mu/2$ & $A_{5}$ & $11$ & $9$ & $36$ & $45$ & $55$
& $M_{11},A_{11},S_{11}$ & $66$ & $27$ & $9\mu $ & $45\mu /2$ & $ 
55\mu $ & $\mu \equiv 0\pmod{2}$ & - \\ 
5& $6$ & $3$ & $4$ & $10$ & $20$ & $9\mu /4$ & $S_{5},A_{6},S_{6}$ & $11$ & $9$ & $36$
& $45$ & $55$ & $M_{11},A_{11},S_{11}$ & $66$ & $27$ & $9\mu $ & $%
45\mu /2$ & $ 55\mu $ & $\mu \equiv 0\pmod{4}$ & -
\\ 
6& $6$ & $3$ & $2$ & $5$ & $10$ & $%
	91\mu /2$ & $A_{5}$ & $16$ & $13$ & $  364$
& $ 455$ & $  560$ & $A_{16},S_{16}$ & $96$ & $39$ & $%
91\mu $ & $  455\mu /2$ & $    560\mu $ & $\mu
\equiv 0\pmod{2}$ & - \\ 
7& $6$ & $3$ & $4$ & $10$ & $20$ & $91\mu/4 $ & $S_{5},A_{6},S_{6}$ & $16$ & $13$ & $%
  364$ & $  455$ & $  560$ & $A_{16},S_{16}$ & 
$96$ & $39$ & $91\mu $ & $455\mu /2$ & $    560\mu $ & $%
\mu \equiv 0\pmod{4}$ & - \\ 
8& $6$ & $4$ & $6$ & $10$ & $15$ & $2\mu/3 $ & $S_{5},A_{6},S_{6}$ & $11$ & $10$
& $9$ & $10$ & $11$ & $AGL_{1}(11),PSL_{2}(11),M_{11}$ & $66$
& $40$ & $4\mu $ & $20\mu /3$ & $  11\mu $ & $\mu \equiv 0\pmod{3}$ & $6$ \\
& & & & & & & & & & & & & $A_{11},S_{11}$ & & & & & & & \\
9& $7$ & $3$ & $1$ & $3$ & $7$ & $36\mu /7$ & $7:3,PSL_{2}(7)$ & $10$ & $8$ & $  28
$ & $36$ & $45$ & $PGL_{2}(9),M_{10},P\Gamma
L_{2}(9),$ & $70$ & $24$ & $36\mu /7$ & $  108\mu /7$
& $  45\mu $ & $\mu \equiv 0\pmod{7}$ & - \\
& & & & & & & & & & & & & $A_{10},S_{10}$ & & & & & & & \\
10& $7$ & $3$ & $2$ & $6$ & $14$ & $18\mu /7$ & $AGL_{1}(7)$ & $10$ & $8$ & $  28
$ & $36$ & $45$ & $PGL_{2}(9),M_{10},P\Gamma
L_{2}(9),$ & $70$ & $24$ & $36\mu /7$ & $  108\mu /7$
& $  45\mu $ & $\mu \equiv 0\pmod{7}$ & - \\
& & & & & & & & & & & & & $A_{10},S_{10}$ & & & & & & & \\
11& $7$ & $3$ & $4$ & $12$ & $28$ & $9\mu/7$ & $PSL_{2}(7)$ & $10$ & $8$ & $  28
$ & $36$ & $45$ & $PGL_{2}(9),M_{10},P\Gamma
L_{2}(9),$ & $70$ & $24$ & $36\mu/7$ & $108\mu/7$
& $45\mu$ & $\mu \equiv 0\pmod{7}$ & - \\
& & & & & & & & & & & & & $A_{10},S_{10}$ & & & & & & & \\
12& $7$ & $3$ & $5$ & $15$ & $35$ & $36\mu /35$ & $A_{7},S_{7}$ & $10$ & $8$ & $%
  28$ & $36$ & $45$ & $PGL_{2}(9),M_{10},P\Gamma
L_{2}(9),$ & $70$ & $24$ & $36\mu /7$ & $%
108\mu /7$ & $    45\mu $ & $\mu \equiv 0\pmod{35}$ & 
- \\
& & & & & & & & & & & & & $A_{10},S_{10}$ & & & & & & & \\
13& $7$ & $4$ & $2$ & $4$ & $7$ & $8\mu /7$ & $PSL_{2}(7)$ & $9$ & $8$ & $7$ & $8$ & $9$
& $PSL_{2}(8),P\Gamma
L_{2}(8)$,$A_{9},S_{9}$ & $63$ & $32$ & $16\mu /7$ & $  32\mu /7$
& $  9\mu $ & $\mu \equiv 0\pmod{7}$ & $7$ \\
& & & & & & & & & & & & & $G^{\Sigma}\leq AGL_{2}(3) $ & & & & & & & \\
14& $7$ & $4$ & $10$ & $20$ & $35$ & $8\mu /35$ & $A_{7},S_{7}$ & $9$ & $8$ & $7$ & $8$
& $9$ & $PSL_{2}(8),P\Gamma L_{2}(8)$,$A_{9},S_{9}$ & $63$ & $32$ & $16\mu /7$ & $ 
32\mu /7 $ & $  9\mu$ & $\mu \equiv 0\pmod{35}$ & - \\
& & & & & & & & & & & & & $G^{\Sigma}\leq AGL_{2}(3) $ & & & & & & & \\
15& $9$ & $5$ & $35$ & $70$ & $126$ & $5\mu /63$ & $A_{9},S_{9}$ & $11$ & $10$ & $9$ & $10$ & $11$ & $%
AGL_{1}(11),PSL_{2}(11),M_{11}$ & $99$ & $50$ & $25\mu /9$ & $%
  50\mu /9$ & $    11\mu $ & $\mu \equiv 0%
\pmod{63}$ & - \\
& & & & & & & & & & & & & $A_{11},S_{11}$ & & & & & & & \\
16& $6$ & $3$ & $2$ & $5$ & $10$ & $\mu/2 $ & $A_{5}$ & $6$ & $5$ & $4$ & $5$ & $6$ & $%
A_{5},S_{5},A_{6},S_{6}$ & $36$ & $15$ & $\mu $ & $5\mu /2$ & $ 
  6\mu $ & $\mu \equiv 0\pmod{2}$ & $6$ \\ 
17& $6$ & $3$ & $4$ & $10$ & $20$ & $\mu/4 $ & $S_{5},A_{6},S_{6}$ & $6$ & $5$ & $4$ & 
$5$ & $6$ & $A_{5},S_{5},A_{6},S_{6}$ & $36$ & $15$ & $\mu $ & $5\mu /2$ & $%
  6\mu $ & $\mu \equiv 0\pmod{4}$ & - \\ 
18 & $8$ & $4$ & $3$ & $7$ & $14$ & $\mu /2$ & $AGL_{1}(8),A\Gamma
L_{1}(8),AGL_{3}(2),PSL_{2}(7)$ & $8$ & $7$ & $6$ & $7$ & $8$ & $%
AGL_{1}(8),A\Gamma L_{1}(8),AGL_{2}(3)$ & $64$ & $28$
& $3\mu /2$ & $7\mu /2$ & $  8\mu $ & $\mu \equiv 0\pmod{2}$ & $%
8$ \\
& & & & & & & & & & & & & $PSL_{2}(7),PGL_{2}(7),A_{8},S_{8}$ & & & & & & & \\
19 & $8$ & $4$ & $6$ & $14$ & $28$ & $\mu /4$ & $PGL_{2}(7)$ & $8$ & $7$ & $6$ & $7$ & $%
8$ & $AGL_{1}(8),A\Gamma L_{1}(8),AGL_{2}(3)$ & $64$
& $28$ & $3\mu /2$ & $7\mu /2$ & $8\mu $ & $\mu \equiv 0\pmod{4}$ & $8$
\\ 
& & & & & & & & & & & & & $PSL_{2}(7),PGL_{2}(7),A_{8},S_{8}$ & & & & & & & \\
20 & $8$ & $4$ & $9$ & $21$ & $42$ & $\mu /6$ & $PSL_{2}(7),PGL_{2}(7)$ & $8$ & $7$ & $6
$ & $7$ & $8$ & $AGL_{1}(8),A\Gamma
L_{1}(8),AGL_{2}(3)$ & $64$ & $28$ & $3\mu /2$ & $%
  7\mu /2$ & $    8\mu $ & $\mu \equiv 0\pmod{%
6}$ & - \\
& & & & & & & & & & & & & $PSL_{2}(7),PGL_{2}(7),A_{8},S_{8}$ & & & & & & & \\
21 & $8$ & $4$ & $12$ & $28$ & $56$ & $\mu /8$ & $AGL_{3}(2)$ & $8$ & $7$ & $6
$ & $7$ & $8$ & $AGL_{1}(8),A\Gamma
L_{1}(8),AGL_{2}(3)$ & $64$ & $28$ & $3\mu /2$ & $%
  7\mu /2$ & $    8\mu $ & $\mu \equiv 0\pmod{%
8}$ & $8$ \\
& & & & & & & & & & & & & $PSL_{2}(7),PGL_{2}(7),A_{8},S_{8}$ & & & & & & & \\
22 & $8$ & $4$ & $15$ & $35$ & $70$ & $\mu /10$ & $A_{8},S_{8}$ & $8$ & $7$ & 
$6$ & $7$ & $8$ & $AGL_{1}(8),A\Gamma
L_{1}(8),AGL_{2}(3)$ & $64$ & $28$ & $3\mu /2$ & $%
7\mu /2$ & $  8\mu $ & $\mu \equiv 0\pmod{10}$ & - \\
& & & & & & & & & & & & & $PSL_{2}(7),PGL_{2}(7),A_{8},S_{8}$ & & & & & & & \\
23 & $9$ & $3$ & $1$ & $4$ & $12$ & $\mu /3$ & $G^{\Delta}_{\Delta} \leq AGL_{2}(3)$ & $5$ & $4$ & $3$ & $4$ & $5$ & $AGL_{1}(5),A_{5},S_{5}$
& $45$ & $12$ & $\mu /3$ & $4\mu /3$ & $  5\mu $ & $\mu \equiv 0%
\pmod{3}$ & $9$ \\ 
24 & $9$ & $3$ & $6$ & $24$ & $72$ & $\mu /18$ & $ASL_{2}(3), AGL_{2}(3)$ & $5$ & $4$ & $3$ & $4$ & $5$ & $AGL_{1}(5),A_{5},S_{5}$
& $45$ & $12$ & $\mu /3$ & $4\mu /3$ & $  5\mu $ & $\mu \equiv 0%
\pmod{18}$ & - \\
25 & $9$ & $3$ & $7$ & $28$ & $84$ & $\mu /21$ & $PSL_{2}(8),P\Gamma L_{2}(8),A_{9},S_{9}
$ & $5$ & $4$ & $3$ & $4$ & $5$ & $AGL_{1}(5),A_{5},S_{5}$ & $45$ & $12$ & $\mu /3$ & $%
4\mu /3$ & $    5\mu $ & $\mu \equiv 0\pmod{21}$ & -
\\ 
26 & $9$ & $3$ & $1$ & $4$ & $12$ & $7\mu /3$ & $G^{\Delta}_{\Delta} \leq AGL_{2}(3)$ & $9$ & $7$ & $21$ & $28$ & $36$ & $%
PSL_{2}(8),P\Gamma L_{2}(8),A_{9},S_{9}$ & $81$ & $21$ & $7\mu /3$ & $28\mu
/3$ & $  36\mu $ & $\mu \equiv 0\pmod{3}$ & $\ $- \\ 
27 & $9$ & $3$ & $6$ & $24$ & $72$ & $7\mu /18$ & $ASL_{2}(3), AGL_{2}(3)$ & $9$ & $7$ & $21$ & $28$ & $36$ & $%
PSL_{2}(8),P\Gamma L_{2}(8),A_{9},S_{9}$ & $81$ & $21$ & $7\mu /3$ & $28\mu
/3$ & $  36\mu $ & $\mu \equiv 0\pmod{18}$ & $\ $- \\ 
28 & $9$ & $3$ & $7$ & $28$ & $84$ & $\mu /3$ & $PSL_{2}(8),P\Gamma L_{2}(8),A_{9},S_{9}
$ & $9$ & $7$ & $21$ & $28$ & $36$ & $PSL_{2}(8),P\Gamma L_{2}(8),A_{9},S_{9}
$ & $81$ & $21$ & $7\mu /3$ & $  28\mu /3$ & $  36\mu $
& $\mu \equiv 0\pmod{3}$ & - \\ 
29 & $10$ & $4$ & $2$ & $6$ & $15$ & $%
	2\mu /5$ & $S_{5},A_{6},S_{6}$ & $7$ & $6
$ & $5$ & $6$ & $7$ & $AGL_{1}(7),PSL_{2}(7),A_{7},S_{7}$ & $70$ & $24$ & $%
4\mu /5$ & $  12\mu /5$ & $  7\mu $ & $\mu \equiv 0\pmod{5}$ & $10$ \\ 
30 & $10$ & $4$ & $4$ & $12$ & $30$ & $\mu /5$ & $M_{10},PGL_{2}(9),P\Gamma
L_{2}(9)$ & $7$ & $6$ & $5$ & $6$ & $7$ & $AGL_{1}(7),PSL_{2}(7),A_{7},S_{7}$
& $70$ & $24$ & $4\mu /5$ & $  12\mu /5$ & $ 7\mu $ & $\mu \equiv 0\pmod{5}$ & $10$ \\ 
31 & $10$ & $4$ & $24$ & $72$ & $180$ & $\mu /30$ & $M_{10},PGL_{2}(9),P\Gamma
L_{2}(9)$ & $7$ & $6$ & $5$ & $6$ & $7$ & $AGL_{1}(7),PSL_{2}(7),A_{7},S_{7}$
& $70$ & $24$ & $4\mu /5$ & $  12\mu /5$ & $ 7\mu $ & $\mu \equiv 0\pmod{30}$ & - \\ 
32 & $10$ & $4$ & $28$ & $84$ & $210$ & $\mu /35$ & $A_{10},S_{10}$ & $7$ & $6$ & $5$ & $%
6$ & $7$ & $AGL_{1}(7),PSL_{2}(7),A_{7},S_{7}$ & $70$
& $24$ & $4\mu /5$ & $12\mu /5$ & $  7\mu $ & $\mu \equiv 0\pmod{%
35}$ & -\\
\hline 
\end{tabular}}
\end{sidewaystable}
\end{proposition}
\begin{proof}
\begin{sidewaystable}
\caption{Admissible parameters for $\mathcal{D}%
_{0}, \mathcal{D}%
_{1}, \mathcal{D}$ and $G_{\Delta }^{\Delta }$ and $G^{\Sigma
}$ for $(v_{0},k_{0})=(16,4)$.}\label{sedici}
{\small\setlength{\tabcolsep}{2pt}
\begin{tabular}{|c|ccccc|c|ccccc|c|ccccc|c|c|}
\hline
Line & $v_{0}$ & $k_{0}$ & $\lambda _{0}$ & $r_{0}$ & $b_{0}$ & $G_{\Delta
}^{\Delta }$ & $v_{1}$ & $k_{1}$ & $\lambda _{1}$ & $r_{1}$ & $b_{1}$ & $%
G^{\Sigma }$ & $v$ & $k$ & $\lambda $ & $r$ & $b$ & Conditions on $\mu $ & $%
\mu _{S}$ \\
\hline
\hline
1 & $16$ & $4$ & $1$ & $5$ & $20$ & $2^{4}:5\leq G_{\Delta }^{\Delta }\leq
A\Gamma L_{2}(4)$ & $6$ & $5$ & $4$ & $5$ & $6$ & $A_{5},S_{5},A_{6},S_{6}$
& $96$ & $20$ & $\mu /4$ & $5\mu /4$ & $6\mu $ & $\mu \equiv 0\pmod{4}$ & 
$16$ \\ 
2 & $16$ & $4$ & $2$ & $10$ & $40$ & $2^{4}:(5:4),ASL_{2}(4),A\Sigma L_{2}(4)
$ & $6$ & $5$ & $4$ & $5$ & $6$ & $A_{5},S_{5},A_{6},S_{6}$ & $96$ & $20$ & $%
\mu /4$ & $5\mu /4$ & $6\mu $ & $\mu \equiv 0\pmod{8}$ & $16$ \\ 
3 & $16$ & $4$ & $3$ & $15$ & $60$ & $AGL_{1}(16),AGL_{1}(16):2$ & $6$ & $5$
& $4$ & $5$ & $6$ & $A_{5},S_{5},A_{6},S_{6}$ & $96$ & $20$ & $\mu /4$ & $%
5\mu /4$ & $6\mu $ & $\mu \equiv 0\pmod{12}$ & $-$ \\ 
4 & $16$ & $4$ & $3$ & $15$ & $60$ & $ASL_{2}(4),A\Sigma
L_{2}(4),ASp_{4}(2),A\Gamma Sp_{4}(2)$ & $6$ & $5$ & $4$ & $5$ & $6$ & $%
A_{5},S_{5},A_{6},S_{6}$ & $96$ & $20$ & $\mu /4$ & $5\mu /4$ & $6\mu $ & $%
\mu \equiv 0\pmod{12}$ & $-$ \\ 
5 & $16$ & $4$ & $4$ & $20$ & $80$ & $ASp_{4}(2)$ & $6$ & $5$ & $4$ & $5$ & $%
6$ & $A_{5},S_{5},A_{6},S_{6}$ & $96$ & $20$ & $\mu /4$ & $5\mu /4$ & $6\mu $
& $\mu \equiv 0\pmod{16}$ & $-$ \\ 
6 & $16$ & $4$ & $6$ & $30$ & $120$ & $2^{4}:(15:4),AGL_{2}(4),A\Gamma
L_{2}(4)$ & $6$ & $5$ & $4$ & $5$ & $6$ & $A_{5},S_{5},A_{6},S_{6}$ & $96$ & 
$20$ & $\mu /4$ & $5\mu /4$ & $6\mu $ & $\mu \equiv 0\pmod{24}$ & $-$ \\ 
7 & $16$ & $4$ & $7$ & $35$ & $140$ & $2^{4}:A_{7},AGL_{4}(2)$ & $%
6$ & $5$ & $4$ & $5$ & $6$ & $A_{5},S_{5},A_{6},S_{6}$ & $96$ & $20$ & $\mu
/4$ & $5\mu /4$ & $6\mu $ & $\mu \equiv 0\pmod{28}$ & $-$ \\ 
8 & $16$ & $4$ & $12$ & $60$ & $240$ & $2^{4}:(15:4)$ & $6$ & $5$ & $4$ & $5$
& $6$ & $A_{5},S_{5},A_{6},S_{6}$ & $96$ & $20$ & $\mu /4$ & $5\mu /4$ & $%
6\mu $ & $\mu \equiv 0\pmod{48}$ & $-$ \\ 
9 & $16$ & $4$ & $12$ & $60$ & $240$ & $A\Sigma L_{2}(4),ASp_{4}(2)$ & $6$ & 
$5$ & $4$ & $5$ & $6$ & $A_{5},S_{5},A_{6},S_{6}$ & $96$ & $20$ & $\mu /4$ & 
$5\mu /4$ & $6\mu $ & $\mu \equiv 0\pmod{48}$ & $-$ \\ 
10 & $16$ & $4$ & $36$ & $180$ & $720$ & $A\Gamma L_{2}(4)$ & $6$ & $5$ & $4$
& $5$ & $6$ & $A_{5},S_{5},A_{6},S_{6}$ & $96$ & $20$ & $\mu /4$ & $5\mu /4$
& $6\mu $ & $\mu \equiv 0\pmod{144}$ & $-$ \\ 
11 & $16$ & $4$ & $84$ & $420$ & $1680$ & $2^{4}:A_{7},AGL_{4}(2)$ & $6$ & $5
$ & $4$ & $5$ & $6$ & $A_{5},S_{5},A_{6},S_{6}$ & $96$ & $20$ & $\mu /4$ & $%
5\mu /4$ & $6\mu $ & $\mu \equiv 0\pmod{336}$ & $-$ \\ 
12 & $16$ & $4$ & $91$ & $455$ & $1820$ & $A_{16},S_{16}$ & $6$ & $5$ & $4$
& $5$ & $6$ & $A_{5},S_{5},A_{6},S_{6}$ & $96$ & $20$ & $\mu /4$ & $5\mu /4$
& $6\mu $ & $\mu \equiv 0\pmod{364}$ & $-$\\
\hline
\end{tabular}}
\end{sidewaystable}
If $\gcd(v_{0},k_{0})=1$, then $v_{0}<v_{1}$ by \cite[Lemma 2.2(iv)]{CZ}, and so 
$v=v_{1}v_{0}>v_{0}^{2}$ implies $v_{0}\leq 9$ and so $v_{0}=5,6,7,8$ or $9$%
. Hence, 
\[
(v_{0},k_{0})=(5,3),(7,3),(7,4),(7,5),(8,3),(8,5),(9,4),(9,5),(9,7)
\]%
since $\gcd(v_{0},k_{0})=1$ and $\mathcal{D}_{0}$ is $2$-design with $3\leq
k_{0}\leq v_{0}-2$. For each of the values $v_{0}<v_{1}<100/v_{0}$ we compute%
\[
k_{1}=\frac{-k_{0}+v_{0}-v_{0}v_{1}+k_{0}v_{0}v_{1}}{k_{0}(v_{0}-1)},
\]%
and we have
\[
(v_{0},k_{0},v_{1},k_{1})=(5,3,7,6),(5,3,13,11),(7,3,10,8),(7,4,9,8),(9,5,11,10).%
\text{ }
\]%
If $(v_{0},k_{0},v_{1},k_{1})=(5,3,7,6)$, or $(5,3,13,11)$, then both $\mathcal{D}_{0}$ and $\mathcal{D}_{1}$ are complete $2$-designs, $A_{5}\trianglelefteq
G_{\Delta }^{\Delta }\leq S_{5}$ and $A_{v_{1}}\trianglelefteq
G^{\Sigma }\leq S_{v_{1}}$ by Lemma \ref{LA53}.

If $(v_{0},k_{0},v_{1},k_{1})=(7,3,10,8)$, then either $\mathcal{D}_{0}\cong PG _{2}(2)$ and $G_{\Delta}^{\Delta} \cong 7:3,PSL_{2}(7)$; or $\mathcal{D}_{0}$ is a $2$-$(7,3,2)$ design, union of two copies of $PG _{2}(2)$, and $G_{\Delta}^{\Delta} \cong A GL_{1}(7)$; or $\mathcal{D}_{0}$ is a $2$-$(7,3,4)$ design and $G_{\Delta}^{\Delta} \cong PSL_{2}(7)$; or $\mathcal{D}_{0}$ is the complete $2$-$(7,3,5)$ design and $G_{\Delta}^{\Delta} \cong
A_{7},S_{7}$ by Lemma \ref{LA4}. Moreover, $\mathcal{D}_{1}$ is the complete $2$-$(10,8,28)$ designs, $G^{\Sigma} \cong
PGL_2(9),M_{10},P\Gamma L_2(9),A_{10},S_{10}$ by Lemma \ref{LA53}. 

If $(v_{0},k_{0},v_{1},k_{1})=(7,4,9,8)$, then either $\mathcal{D}_{0}\cong \overline{PG_{2}(2)}$ and $G_{\Delta}^{\Delta} \cong PSL_{2}(7)$, or $\mathcal{D}_{0}$ is the complete $2$-$(7,4,10)$ design and $G_{\Delta}^{\Delta} \cong A_{7},S_{7}$ by Lemma \ref{LA4}. Moreover, $\mathcal{D}_{1}$ is the complete $2$-design $(7,3,5)$ design and $G^{\Sigma} \cong
A_{7},S_{7}$ by Lemma \ref{LA4}. 

Finally, if $(v_{0},k_{0},v_{1},k_{1})=(9,5,11,10)$, then both $\mathcal{D}_{0}$ and $\mathcal{D}_{1}$ are complete $2$-designs and hence $G_{\Delta}^{\Delta} \cong A_{9},S_{9}$ or $G^{\Sigma} \cong AGL_{1}(11), PSL_2(11),M_{11},A_{11},S_{11}$ by Lemma \ref{LA4}.

If $\gcd(v_{0},k_{0})>1$ and $v_{0}<v_{1}$, we obtain $v_{0}\leq 9$ and so $%
v_{0}=5,6,7,8$ or $9$ again and hence $%
(v_{0},k_{0})=(6,3),(6,4),(8,4),(8,6),(9,3),(9,6)$, and hence 
\[
(v_{0},k_{0},\lambda
_{0},r_{0},v_{1},k_{1})=(6,3,i,5i/2,11,9),(6,3,i,5i/2,16,13),(6,4,6,10,11,10).
\]

Assume that $(v_{0},k_{0})=(6,3)$. Then either $\mathcal{D}_{0}$ is a $2$-$(6,3,2)$ design and $G_{\Delta}^{\Delta} \cong A_{5}$, or $\mathcal{D}_{0}$ is the complete $2$-$(6,3,4)$ design and $G_{\Delta}^{\Delta} \cong S_{5},A_{6},S_{6}$ by Lemma \ref{LA1}. Moreover, $\mathcal{D}_{1}$ is the complete $2$-design and $G^{\Sigma}\cong A_{v_{1}},S_{v_{1}}$, or additionally $G^{\Sigma}\cong AGL_{1}(11),PSL_{2}(11),M_{11}$ for $(v_{1},k_{1})=(11,10)$ by Lemma \ref{LA53}.

Assume that $(v_{0},k_{0},\lambda_{0},r_{0},v_{1},k_{1})=(6,4,6,10,11,10)$. Then both $\mathcal{D}_{0}$ and $\mathcal{D}_{1}$ are complete $2$-designs and hence $G_{\Delta}^{\Delta} \cong A_{5},A_{6},S_{6}$ or $G^{\Sigma}$ is isomorphic to one of the groups $AGL_{1}(11), PSL_2(11),M_{11},A_{11},S_{11}$ by Lemma \ref{LA53}.

If $\gcd(v_{0},k_{0})>1$ and $v_{0}\geq v_{1}$, it follows that $%
v_{1}^{2}<v$ and so $3\leq k_{0}<k_{1}<v_{1}\leq 9$, and in particular $%
5\leq v_{1}\leq 9$. Now, since 
\[
k_{1}=\frac{-k_{0}+v_{0}-v_{0}v_{1}+k_{0}v_{0}v_{1}}{k_{0}(v_{0}-1)},
\]%
and we have
\[
(v_{0},k_{0},v_{1},k_{1})=(6,3,6,5),\left( 8,4,8,7\right)
,(9,3,5,4),(9,3,9,7),(10,4,7,6),\left( 16,4,6,5\right).
\]%
We now proceed as above and we determine the pair $(\mathcal{D}_{0},G_{\Delta}^{\Delta})$ by Lemmas \ref{LA1},\ref{LA2}, \ref{LA5}, \ref{LA6}, \ref{LA7} and \ref{LA8}, respectively, whereas $(\mathcal{D}_{1},G^{\Sigma})$ follows from Lemma \ref{LA53}. This completes the proof.
\end{proof}

\section {Proof of the main result}\label{Sec 4}
 In this section, we focus on the case where $\mathcal{D}$ is symmetric. We will use al the results contained in the previous sections as well as those contained in the Appendix, to prove Theorem \ref{main} and Corollary \ref{cor1}.

 \bigskip

\begin{proposition}\label{sinithos}
Let $\mathcal{D}$ be a symmetric $2$-$(v,k,\lambda )$ design admitting $G$
as a flag-transitive point-imprimitive automorphism group. Then then parameters $\mathcal{D}%
_{0},\mathcal{D}_{1}$ and $\mathcal{D}$ and the possibilities for $G_{\Delta }^{\Delta }$ and $G^{\Sigma}$ are as in Table \ref{symm}. In particular, $G$ acts $2$-transitively on $\Sigma$.
\begin{sidewaystable}
\caption{Admissible parameters for $\mathcal{D}%
_{0}, \mathcal{D}%
_{1}, \mathcal{D}%
_{0}$ and $G_{\Delta }^{\Delta }$ and $G^{\Sigma
}$ when $\mathcal{D}$ is symmetric.}\label{symm}
{\small\setlength{\tabcolsep}{2pt}
\begin{tabular}{|c|ccccc|c|c|ccccc|c|c|ccc|}
\hline
Line & $v_{0}$ & $k_{0}$ & $\lambda _{0}$ & $r_{0}$ & $b_{0}$ & $\theta $ & $G_{\Delta
}^{\Delta }$ & $v_{1}$ & $k_{1}$ & $\lambda _{1}$ & $r_{1}$ & $b_{1}$ & $\mu$ & $%
G^{\Sigma }$ & $v$ & $k$ & $\lambda $ \\
\hline
\hline
1 & $3$ & $2$ & $1$ & $2$ & $3$ & $4$ & $S_{3}$ & $5$ & $4$ & $3$ & $4$ & $5$ & $3$ & $%
AGL_1(5),A_{5},S_{5}$ & $15$ & $8$ & $4$ \\ 
2 & $4$ & $2$ & $1$ & $3$ & $6$ & $2$ & $A_{4},S_{4}$ & $4$ & $3$ & $2$ & $3$ & $4$& $4$
& $A_{4},S_{4}$ & $16$ & $6$ & $2$ \\ 
3 & $4$ & $2$ & $1$ & $3$ & $6$ & $6$ & $A_{4},S_{4}$ & $13$ & $9$ & $6$ & $9$ & $13$ & $4$
& $PSL_{3}(3)$ & $52$ & $18$ & $6$ \\ 
4 & $5$ & $3$ & $3$ & $6$ & $10$ & $3$ & $A_{5},S_{5}$ & $7$ & $6$ & $5$ & $6$ & $7
$ & $5$ & $AGL_1(7),PSL_{2}(7),A_{7},S_{7}$ & $35$ & $18$ & $9$ \\ 
5 & $6$ & $4$ & $6$ & $10$ & $15$ & $4$ & $S_{5},A_{6},S_{6}$ & $11$ & $10$
& $9$ & $10$ & $11$ & $6$ & $AGL_{1}(11),PSL_{2}(11),M_{11},A_{11},S_{11}$ & $66$
& $40$ & $24$ \\ 
6 & $7$ & $4$ & $2$ & $4$ & $7$ & $8$ & $PSL_{2}(7)$ & $9$ & $8$ & $7$ & $8$ & $9$ & $3$ & $PSL_{2}(8),P\Gamma L_{2}(8)$,$A_{9},S_{9}$ & $63$ & $32$ & $16$ \\ 
& & & & & & & & & & & & & & $G^{\Sigma}\leq AGL_{3}(2) $ & & &\\
7 & $3$ & $2$ & $1$ & $2$ & $3$ & $16$ & $S_{3}$ & $21$ & $16$ & $12$ & $16$ & $21$ & $7$ & $PSL_{3}(4):\epsilon$ with $\epsilon \mid 6$ & $63$ & $32$ & $16$ \\
8 & $6$ & $3$ & $2$ & $5$ & $10$ & $3$ & $A_{5}$ & $6$ & $5$ & $4$ & $5$ & $6$ & $6$ & $%
A_{5},S_{5},A_{6},S_{6}$ & $36$ & $15$ & $6$ \\ 
9 & $8$ & $4$ & $3$ & $7$ & $14$ & $4$ & $AGL_{1}(8),A\Gamma
L_{1}(8),AGL_{2}(3),PSL_{2}(7)$ & $8$ & $7$ & $6$ & $7$ & $8$ & $8$ & $AGL_{1}(8),A\Gamma
L_{1}(8),AGL_{3}(2)$ & $64$ & $28$
& $12$ \\
& & & & & & & & & & & & & & $PSL_{2}(7),PGL_{2}(7),A_{8},S_{8}$ & & &\\
10 & $8$ & $4$ & $6$ & $14$ & $28$ & $2$ & $PGL_{2}(7)$ & $8$ & $7$ & $6$ & $7$ & $8
$ & $8$ & $AGL_{1}(8),A\Gamma
L_{1}(8),AGL_{3}(2)$ & $64$ & 
$28$ & $12$ \\ 
& & & & & & & & & & & & & & $PSL_{2}(7),PGL_{2}(7),A_{8},S_{8}$ & & &\\
11 & $8$ & $4$ & $12$ & $28$ & $56$ & $1$ & $AGL_{3}(2)$ & $8$ & $7$ & $6
$ & $7$ & $8$ & $8$ & $AGL_{1}(8),A\Gamma
L_{1}(8),AGL_{3}(2)$ & $64$ & $28$ & $12$ \\
& & & & & & & & & & & & & & $PSL_{2}(7),PGL_{2}(7),A_{8},S_{8}$ & & &\\
12 & $9$ & $3$ & $1$ & $4$ & $12$ & $3$ & $G^{\Delta}_{\Delta} \leq AGL_{3}(2)$ & $5$ & $4$ & $3$ & $4$ & $5$ & $9$ & $A_{5},S_{5}$
& $45$ & $12$ & $3$ \\ 
13 & $10$ & $4$ & $2$ & $6$ & $15$ & $4$
	&  $S_{5},A_{6},S_{6}$ & $7$ & $6
$ & $5$ & $6$ & $7$ & $10$ & $AGL_{1}(7),PSL_{2}(7),A_{7},S_{7}$ & $70$ & $24$ & $8$
\\ 
14 & $10$ & $4$ & $4$ & $12$ & $30$ & $2$ & $PGL_{2}(9),M_{10},P\Gamma
L_{2}(9)$ & $7$ & $6$ & $5$ & $6$ & $7$ & $10$ & $AGL_{1}(7),PSL_{2}(7),A_{7},S_{7}$
& $70$ & $24$ & $8$ \\
15 & $16$ & $4$ & $1$ & $5$ & $20$ & $4$ & $2^{4}:5\leq G_{\Delta }^{\Delta }\leq
A\Gamma L_{2}(4)$ & $6$ & $5$ & $4$ & $5$ & $6$ & $16$ & $A_{5},S_{5},A_{6},S_{6}$
& $96$ & $20$ & $4$ \\ 
16 & $16$ & $4$ & $2$ & $10$ & $40$ & $2$& $2^{4}:(5:4),ASL_{2}(4),A\Sigma
L_{2}(4)$ & $6$ & $5$ & $4$ & $5$ & $6$ & $16$ & $A_{5},S_{5},A_{6},S_{6}$ & $96$ & 
$20$ & $4$\\
\hline
\end{tabular}}
\end{sidewaystable}
\end{proposition}
\begin{proof}
Table \ref{symm} arises follows from Tables \ref{smeagol}, \ref{smaug} and \ref{sedici} for $\mu=\mu_{S}$. In all cases of Table \ref{symm}, the group $G$ acts $2$-transitively on $\Sigma$.    
\end{proof}

\bigskip

Let $\Delta \in \Sigma $ and $x \in \Delta $. Since $G_{(\Sigma)} \trianglelefteq G_{\Delta}$ and $G_{(\Delta)} \trianglelefteq G_{x}$, it is immediate to verify that $(G^{\Sigma})_{\Delta}=(G_{\Delta})^{\Sigma}$ and that $\left( G_{\Delta}^{\Delta} \right)_{x}=(G_{x})^{\Delta}$. Hence, in the sequel $(G^{\Sigma})_{\Delta}$ and $\left( G_{\Delta}^{\Delta} \right)_{x}$ will simply be denoted by $G^{\Sigma}_{\Delta}$ and $G_{x}^{\Delta}$, respectively. Moreover, the following holds:
\begin{equation}\label{toulaxiston}
\frac{G^{\Sigma}_{\Delta}}{G_{(\Delta)}^{\Sigma}} \cong \frac{G_{\Delta}}{G_{(\Delta)}G_{(\Sigma)}} \cong \frac{G^{\Delta}_{\Delta}}{G_{(\Sigma)}^{\Delta}}.
\end{equation}

\bigskip

\begin{lemma}\label{soc} One of the following holds:
\begin{enumerate}
    \item $G_{(\Sigma)}=1$, $\mathcal{D} \cong \overline{PG_{3}(2)}$ and $G\cong \Sigma L_{2}(4)$;
    \item $G_{(\Sigma)}\neq 1$ and $Soc(G_{\Delta }^{\Delta })\trianglelefteq G_{(\Sigma )}^{\Delta }$ for any $\Delta \in \Sigma $.
\end{enumerate}
\end{lemma}

\begin{proof}
Assume that $G_{(\Sigma )}=1$. Then $G^{\Sigma }=G$, and hence $G_{\Delta}^{\Sigma }=G_{\Delta}$. Furthermore, $G_{\Delta }^{\Delta }$ is a quotient group of $G_{\Delta}^{\Sigma }$ by (\ref{toulaxiston}). Now, looking at Table \ref{symm}, we see that only the following cases as in Lines 1, 3, 7, 8, 9 and 13 fulfill the previous property. More precisely, we have the following admissible cases:
\begin{enumerate}
    \item[(i).] $\mathcal{D}$ is a $2$-$(15,8,4)$ design, $G\cong S_{5}$, $G_{\Delta}\cong S_{4}$ and $G_{\Delta}^{\Delta}\cong S_{3}$;
    \item[(ii).]$\mathcal{D}$ is a $2$-$(52,18,6)$ design, $G\cong PSL_{3}(3)$, $G_{\Delta}\cong 3^{2}:GL_{2}(3)$ and $G_{\Delta}^{\Delta}\cong S_{4}$;
    \item[(iii).] $\mathcal{D}$ is a $2$-$(63,32,16)$ design, $G\cong P\Gamma L_{3}(4)$, $G_{\Delta}\cong 2^{4}:(3 \times A_{5})\cdot2$ and $G_{\Delta}^{\Delta}\cong S_{3}$;
    \item[(iv).] $\mathcal{D}$ is a $2$-$(36,15,6)$ design, $G\cong A_{6}$, $G_{\Delta}\cong A_{5}$ and $G_{\Delta}^{\Delta}\cong A_{5}$;
     \item[(v).] $\mathcal{D}$ is a $2$-$(64,28,12)$ design, $G\cong AGL_{3}(2)$, $G_{\Delta}\cong PSL_{2}(7)$ and $G_{\Delta}^{\Delta}\cong PSL_{2}(7)$;
     \item[(vi).] $\mathcal{D}$ is a $2$-$(70,24,8)$ design, $G\cong A_{7}:2^{\varepsilon}$, $\varepsilon=0,1$, $G_{\Delta}\cong A_{6}:2^{\varepsilon}$ and $G_{\Delta}^{\Delta}\cong A_{6}:2^{\varepsilon}$.
\end{enumerate}
Cases (ii), (iv) and (vi) cannot occur by \cite[Theorem 1]{MS}. In (v), one has $\left\vert G\right\vert
=2^{6}\cdot 3\cdot 7$, and hence $\left\vert
G_{B}\right\vert = 21$ since $b=v=2^{6}$, where $B$ is any block of $\mathcal{D}$. However, this contradicts $G_{B}$ acting transitively
on $B$ and $k=28$. So, (v) is ruled  out. In (iii), one has $G_{(\Delta)}\cong  2^{4}: A_{5}$, and $A_{5}$ acts irreducibly on the normal subgroup of order $2^{4}$. Thus, if we denote $G_{(\Delta)}$ with $K$,  $K^{\prime}\cong  2^{4}: A_{5}$ fixes each of the $\theta=4$ blocks of $\mathcal{D}$ intersecting $\Delta$ in the same subset. In particular, $K^{\prime}$ fixes each of the $12$ blocks of $\mathcal{D}$ intersecting $\Delta$ in a non-empty set. Then there are two distinct blocks $B_{1}$ and $B_{2}$ of $\mathcal{D}$ such that the blocks $B_{1}^{\Sigma}$ and $B_{2}^{\Sigma}$ are also distinct since $\mu_{S}=3$, and these are also preserved by $K^{\prime}$. Now, both $B_{1}^{\Sigma}$ and $B_{2}^{\Sigma}$ are the complementary sets of two (distinct) lines of $PG_{2}(4)$ since $\mathcal{D}_{1}\cong\overline{PG_{2}(4)}$ by Lemma \ref{LA9}(3). Then $K^{\prime}$ preserves the lines $\ell_{1}$ and $\ell_{2}$ of $PG_{2}(4)$ complementary to $B_{1}^{\Sigma}$ and $B_{2}^{\Sigma}$, respectively, and their intersection point, say $x$. So, $2^{4}: A_{5}\cong K^{\prime}\leq G_{x,\ell_{1},\ell_{2}}$, and we reach a contradiction since the order of $G_{x,\ell_{1},\ell_{2}}$ is not divisible by $5$ being $G\cong P\Gamma L_{3}(4)$ flag-transitive on $PG_{2}(4)$ and $5$ the number of lines incident with any fixed point of $PG_{2}(4)$. Finally, (i) implies (1) by \cite[Lemma 4.3]{PZ}.

Assume that $G_{(\Sigma )} \neq 1$. If there is $\Delta _{0}\in \Sigma $ such that $G_{(\Sigma )}^{\Delta_{0}
}=1$. Then $G_{(\Sigma )}\leq G_{(\Delta )}$ for each $\Delta \in \Sigma$ since $G_{(\Sigma)}\unlhd G$ and $G$ acts transitively on $\Sigma$, and hence $G_{(\Sigma )}=1$, which contrary to our assumption. Thus, $G_{(\Sigma )}^{\Delta} \neq 1$ for each $\Delta \in \Sigma$. Now, recall that $G_{\Delta }^{\Delta }$ acts point primitively on $\mathcal{D}_{0}$ since it acts point-$2$-transitively on $\mathcal{D}_{0}$ when $G_{\Delta }^{\Delta }$ is as in Lines 1--11 and 13--14, and by Lemmas \ref{LA6} and \ref{LA8} when $G_{\Delta }^{\Delta }$ is as in Lines 12 and 15--16, respectively. Hence, $Soc(G_{\Delta }^{\Delta })\trianglelefteq G_{(\Sigma )}^{\Delta }$ for each $\Delta \in \Sigma$ by \cite[Theorem 4.3B]{DM} since $1 \neq G_{(\Sigma )}^{\Delta } \unlhd G_{\Delta}^{\Delta }$, which is (2).
\end{proof}

\bigskip

\begin{theorem}\label{TheList}
Let $\mathcal{D}$ be a symmetric $2$-$(v,k,\lambda )$ design admitting a
flag-transitive point-imprimitive automorphism group $G$. If $G_{(\Sigma)}\neq 1$, then one of the
following holds:

\begin{enumerate}
\item $\mathcal{D}$ is isomorphic to one of the two $2$-$(16,6,2)$ designs as in \cite{ORR}, and $G$ is as in Lines 1 or 2 of Table \ref{telionos}, respectively;

\item $\mathcal{D}$ is isomorphic to the $2$-$(45,12,3)$ design as in \cite{P}, and $G$ is one of the groups $G$ is as in Line 3 of Table \ref{telionos};

\item $\mathcal{D}$ is isomorphic to the $2$-$(15,8,4)$ design as in \cite{CK} or \cite{PZ}, and $\Sigma L_{2}(4) \unlhd G \leq \Gamma L_{2}(4)$;

\item $\mathcal{D}$ is isomorphic to one of the four $2$-$(96,20,4)$ designs as in \cite{LPR}, and $G$ is as in Lines 5--8 of Table \ref{telionos}, respectively;

\item $\mathcal{D}$ is a $2$-$(63,32,16)$ design and the following hold:
\begin{enumerate}
\item $\mathcal{D}_{0}$ is the complete $2$-$(3,2,1)$ design and $%
G_{\Delta }^{\Delta }\cong S_{3}$
\item $\mathcal{D}_{1}\cong \overline{PG_{2}(4)}$ and $PSL_{3}(4) \unlhd G^{\Sigma
} \leq P\Gamma L_{3}(4)$.
\end{enumerate}
\item $\mathcal{D}$ is a $2$-$(64,28,12)$ design and the following hold:
\begin{enumerate}
\item either $\mathcal{D}_{0}\cong AG_{3}(2)$ with all planes as blocks, and $%
G_{\Delta }^{\Delta }$ isomorphic to one of the groups $AGL_{1}(8),A\Gamma
L_{1}(8),AGL_{3}(2)$, or $\mathcal{D}_{0}$ is a $2$-$(8,4,12)$ design and $G_{\Delta }^{\Delta }\cong AGL_{3}(2)$;
\item $\mathcal{D}_{1}$ is the complete $2$-$(8,7,6)$ design, and $G^{\Sigma
}$ isomorphic to one of the group $AGL_{1}(8),A\Gamma
L_{1}(8),PSL_{2}(7),PGL_{2}(7),A_{8},S_{8}$.
\end{enumerate}
\end{enumerate}
\end{theorem}
\begin{proof}
Assertions (1)--(4) follow from \cite[Theorem 1]{MS} and \cite[Theorem 1.2]%
{Mo} for $\lambda \leq 10$. In these cases, $G$ is detemined by using \cite{CK,ORR,LPR,P,PZ} together with aid of \texttt{GAP} \cite{GAP}. Hence, assume that $\lambda >10$. Hence, only
the cases as in lines 5--7 and 9--11 need to be analyzed.

Suppose that $G_{\Delta }^{\Delta }$ is almost
simple. Since $G$ is permutationally isomorphic to a subgroup $G_{\Delta
}^{\Delta }\wr G^{\Sigma }$ by \cite{SP}, we may identify the point set of $%
\mathcal{P}$ with $\Delta \times \Sigma $ and hence $\mathcal{P=}$ $%
\bigcup\limits_{i=1}^{v_{1}}\Delta _{i}$, where $\Delta _{i}=\left( \Delta
,i\right) $. For each $i=1,...,v_{1}$ set $T_{i}=Soc(G_{\Delta _{i}}^{\Delta
_{i}})$ and $T=\prod\limits_{i=1}^{v_{1}}T_{i}$. Then $T\trianglelefteq
G_{\Delta }^{\Delta }\wr G^{\Sigma }$, and $G$ normalizes $T$. Hence, $%
L=G_{(\Sigma )}\cap T$ is a normal subgroup of $G$ since $G_{(\Sigma
)}\vartriangleleft G$. Moreover, $L^{\Delta _{i}}\leq T^{\Delta _{i}}=T_{i}$
for each $i=1,...,v_{1}$. 

Suppose that there is $1 \leq i_{0} \leq v_{1}$
such that $L^{\Delta _{i_{0}}}=1$, then $L\leq G_{(\Delta _{i_{0}})}$ and
hence $L=1$ since $L\vartriangleleft G$ and $G$ acts transitively on $\Sigma 
$. Then $G_{(\Sigma )}\cap T=1$ and so $G_{(\Sigma )}$ is solvable since $%
G_{(\Sigma )}\leq S$ with $S=\prod\limits_{i=1}^{v_{1}}G_{\Delta
_{i}}^{\Delta _{i}}$ and $S/T$ is solvable. Therefore $G_{(\Sigma )}^{\Delta
}$ is solvable and hence $G_{(\Sigma )}^{\Delta }=1$ since $G_{\Delta
}^{\Delta }$ almost simple, but this contradicts Lemma \ref{soc}(2). Thus $L^{\Delta _{i}} \neq 1$ for each $i=1,...,v_{1}$, and hence 
$L^{\Delta_{i}}=T_{i}$ for each $i=1,...,v_{1}$  by \cite[Theorem 4.3B]{DM} since $G$ acts
point-primitively on $\mathcal{D}$, $L^{\Delta _{i}} \unlhd G^{\Delta _{i}}$ and $L^{\Delta _{i}}\leq T_{i}$
for each $i=1,...,v_{1}$.

Let $L_{i}=L\cap K_{i}$, where $K_{i}=T_{1}\times \cdots \times
T_{i-1}\times \left\{ 1\right\} \times T_{i+1}\times \cdots \times T_{v_{1}}$
for each $i=1,...,v_{1}$. Then $L_{i}\trianglelefteq L$ with $%
L_{i}=L_{(\Delta _{i})}$, and $\bigcap\limits_{i=1}^{v_{1}}L_{i}=1$.
Moreover $L_{i}\neq L_{j}$ for $i\neq j$\ and 
\[
L/L_{i}=L/L_{(\Delta _{i})}=L/\left( L\cap G_{(\Delta _{i})}\right) \cong
LG_{(\Delta _{i})}/G_{(\Delta _{i})}=L^{\Delta _{i}}=T_{i}
\]
for each $i=1,...,v_{1}$. Thus, $L=T$ by \cite[Lemma 4.3.A]{DM}. Then any
non-trivial element of $\left\{ 1\right\} \times \cdots \times \left\{
1\right\} \times T_{i}\times \left\{ 1\right\} \times \cdots \times \left\{
1\right\} $  fix $\mathcal{P}\setminus \Delta _{i}$ pointwise, which has size 
$\left( v_{1}-1\right) v_{0}>v_{0}v_{1}/2=v/2$ since $v_{1}>2$, but this is contrary to 
\cite[Corollary 3.7]{Land}. Thus, $G_{\Delta }^{\Delta }$ cannot be almost
simple. Therefore, only the cases  as in Lines 7, 9 or 11 of Table \ref{symm} are admissible, and
we obtain (5) and (6), respectively.
\end{proof}

\begin{table}[h!]
	\centering
	\caption{Symmetric $2$-designs with $\lambda \leq 4$ and related flag-transitive point-imprimitive automorphism groups.}\label{telionos}
\begin{tabular}{lll}
\hline
Line & $(v,k,\lambda)$ & flag-transitive point-imprimitive automorphism group $G
$ \\
\hline
1 & $(16,6,2)$ & $(2^{4}:3):2$, $(2^{4}:2):3$, $\left( (2^{4}:2):2\right) :3$%
, $\left( (2^{4}:2):3\right) :2$, \\ 
 &  & $((4\times 4):3):2$, $((4\times 4):2):3$ (4 classes), $\left(
((4\times 4):3):2\right) :2$ \\ 
 &  & $\left( \left( \left( (2^{4}:2):2\right) :2\right) :2\right) :3$, $%
\left( \left( \left( (2^{4}:2):2\right) :2\right) :3\right) :2$ (4 classes), 
\\ 
 &  & $\left( \left( \left( 2\times (2^{4}:2)\right) :2\right) :3\right) :2$
\\ 
2 & $(16,6,2)$ & $((2\times \left( \left( 4\times 2\right) :2\right) :4):3$, 
$\left( ((2\times \left( \left( 4\times 2\right) :2\right) :4):3\right) :2$
\\ 
3 & $(45,12,3)$ & $(3^{4}:5):8$, $3^{4}:2.A_{5}$,  $3^{4}:2.S_{5}$ \\
4 & $(15,8,4)$ & $\Sigma L_{2}(4)$, $\Gamma L_{2}(4)$\\
5 & $(96,20,4)$ & $2^{8}.(3\times A_{6}).2$, $2^{8}.(3\times A_{6})$, $%
2^{8}.S_{6}$, $2^{8}.\Gamma L_{2}(4)$, $2^{8}.A_{6}$, $2^{8}.GL_{2}(4)$, \\ 
  &             &   $2^{8}.A_{5},2^{8}.S_{5}$ \\
6 & $(96,20,4)$ & $2^{8}.S_{6}$, $2^{8}.A_{6}$, $2^{8}.S_{5}$, $2^{8}.A_{5}$%
, $2^{4}.S_{6}$, $2^{8}.A_{6}$ (4 classes), $2^{4}.S_{5}$ \\ 
7 & $(96,20,4)$ & $2^{6}.(3.A_{6}).2$, $2^{6}.(3.A_{6})$, $2^{6}.\Gamma
L_{2}(4)$, $2^{6}.GL_{2}(4)$, $2^{6}.S_{5}$, $2^{6}.A_{5}$ \\
   &            & $2^{5}.S_{6}$, $%
2^{5}.A_{6}$ (2 classes), $2^{5}.S_{5}$ \\
8 & $(96,20,4)$ & $2^{6}.S_{5},2^{6}.A_{5}$, $2^{5}.S_{5}$, $2^{5}.A_{5}$, $%
2^{4}.S_{5}$ \\
\hline
\end{tabular}
\end{table}

\bigskip

\begin{lemma}\label{SixtyThree}
Let $\mathcal{D}$ be a $2$-$(63,32,16)$ design as in Theorem \ref{TheList}(5). Then $\mathcal{D}\cong \overline{PG_{5}(2)}$ and $%
\Sigma L_{3}(4)\trianglelefteq G\leq \Gamma L_{3}(4)$.
\end{lemma}

\begin{proof}
Let $\mathcal{D}$ be a $2$-$(63,32,16)$ design as Theorem \ref{TheList}(5). Then $%
G_{\Delta }^{\Delta }\cong S_{3}$, and hence either $G_{(\Sigma
)}^{\Delta }$ is cyclic of order $3$, or $G_{(\Sigma )}^{\Delta }\cong S_{3}$ since $G_{(\Sigma )}^{\Delta }\unlhd G_{\Delta }^{\Delta }$. In particular, the order of $G_{(\Sigma)}$ is divisible by $3$.

Suppose that $G_{(\Sigma )}$ contains a subgroup of $K$ of order $9$. Then 
$K_{x}\neq 1$. Then any element $\psi \in K_{x}$, $\psi \neq 1$, the $\psi$ fixes at least a
block of $\mathcal{D}$, say $B$, by \cite[Theorem 3.1]{Land}. Now, $B$ intersects precisely $16
$ elements of $\Sigma $ each of them in $k_{0}=2$ points, and $\psi $ intersects
each such intersections pointwise since $\psi \in K\leq G_{(\Sigma )}$ and $%
\psi $ is a $3$-element. Therefore $\psi $ fixes $B$ pointwise, and hence $%
\psi $ fixes more than $v/2$ points of $\mathcal{D}$ since $k=32$ and $v=63$%
, but this contradicts \cite[Theorem 3.5]{Land}. Thus, $G_{(\Sigma )}$
contains a unique Sylow $3$-subgroup, which is clearly normal in $G_{(\Sigma
)}$ and acts transitively on each element of $\Sigma $ since $G_{(\Sigma )}\leq
\prod_{\Delta \in \Sigma }G_{\Delta }^{\Delta }$. Therefore, no involution in $G_{(\Sigma)}$ centralizes the Sylow $3$-subgroup of $G_{(\Sigma)}$, if there are any. Thus, $G_{(\Sigma )}$
is either cyclic of order $3$ or $S_{3}$. 

Assume that $G_{(\Sigma
)}\cong S_{3}$, Then each of the $3$ involutions of $G_{(\Sigma )}$ fixes
exactly $21$ points. Moreover, no distinct involutions fix the same point
since the normal cyclic subgroup of $G_{(\Sigma )}$ of order $3$ acts point-semiregularly
on $\mathcal{D}$. Therefore, $\left\{ Fix(\sigma _{1}),Fix(\sigma
_{2}),Fix(\sigma _{3})\right\} $, where $\sigma _{1},\sigma_{2},\sigma _{3}$ are
the $3$ involutions of $G_{(\Sigma )}$, is a $G$-invariant partition of $%
\mathcal{D}$ in classes each of size $21$ since $G_{(\Sigma )}\cong S_{3}$ is normal in $G$. However,
this is contrary to Theorem \ref{TheList}. Thus $G_{(\Sigma )}$ is cyclic of order $3$,
and hence $G_{(\Sigma )}$ is central in the preimage $L$ in $G$ of the group 
$PSL_{3}(4)$. Then either $L\cong 3\times PSL_{3}(4)$ or $L\cong SL_{3}(4)$.
In the former case, $PSL_{3}(4)\trianglelefteq G$ and hence the $PSL_{3}(4)$%
-orbits form $G$-invariant partition of $\mathcal{D}$ in $3$ classes each of
size $21$, again contrary to Theorem \ref{TheList}. Thus $L\cong SL_{3}(4)$ and hence $%
SL_{3}(4)\trianglelefteq G\leq \Gamma L_{3}(4)$.
It can be deduced from \cite{At} that the group $G$ has exactly two permutation permutation representations of degree $63$ and, regarding $G$ as a subgroup of $SL_{6}(2)$, these are equivalent
to the action of $G$ on the set of the points and of the set of the hyperplanes of 
of $PG_{5}(2)$, respectively, by \cite[Table 8.3]{BHRD}. These are in turn
equivalent via the inverse-transpose automorphism of $G$, therefore we may
identify the point set of $\mathcal{D}$ with that of $PG_{5}(2)$. Again by \cite[Table 8.3]{BHRD}, the group $G$ lies in a maximal $\mathcal{C}_{3}$-subgroup of $SL_{6}(2)$, and hence it preserves a regular spread of $PG_{5}(2)$, namely $\Sigma$, on which induces $G^{\Sigma}$ in its $2$-transitive permutation representation of degree $21$. Thus, the point-$G_{x}$-orbits on $\mathcal{D}$ have length either $1$, $1$, $1$ and $60$, or $1$, $2$ and $60$ according as $G$ is isomorphic to $SL_{3}(4), GL_{3}(4)$ or $\Sigma L_{3}(4), \Gamma L_{3}(4)$, respectively. On the other hand, we know that $\frac{k}{\gcd(k,\lambda)}=2$ must divide the length of each non-trivial point-$G_{x}$-orbit distinct from $\{x\}$ since $(y^{G_{x}},C^{G_{x}})$, where $y$ is any point of $\mathcal{D}$ distinct from $x$, and $(x,C)$ is flag of $\mathcal{D}$, is a $1$-design by \cite[(1.2.6)]{Demb}. Therefore, $\Sigma L_{3}(4)\trianglelefteq G\leq \Gamma L_{3}(4)$. Moreover, $G_{B}$, where $B$ is any block of $\mathcal{D}$, is not conjugate in $G$ to $G_{x}$. Therefore, $G_{B}$ is the stabilizer a hyperplane since we have seen that the actions of $G$ on the set of the points and the set of hyperplanes of $PG_{5}(2)$ provide the unique permutation representations of $G$ of degree $63$. Since an hyperplane of of $PG_{5}(2)$ and its complementary set consist of $31$ and $32$ points, respectively, and bearing in mind that $G_{B}$ acts transitively on $B$ and $k=32$, it follows that $B$ is the complementary set of a hyperplane of $PG_{5}(2)$. Thus $\mathcal{D}$ is the complementary design of $PG_{5}(2)$. This completes the proof. 
\end{proof}

\bigskip

\begin{lemma}\label{ftinoporo} Let $\mathcal{D}$ be a $2$-$(64,28,12)$ design as in Theorem \ref{TheList}(6). Then the following hold:
\begin{enumerate}
    \item $G_{(\Sigma )}$ is an elementary abelian $2$-group of order $2^{e}$, $e\geq 3$, acting transitively on each $\Delta \in \Sigma$. Moreover, $G_{(\Delta)}\leq G_{(\Sigma)}$ for each $\Delta \in \Sigma$;
    \item  $\mathcal{D}_{0}\cong AG_{3}(2)$ with all planes as blocks;
    \item The following table holds: 
    \begin{table}[h!]
    \centering
    \caption{Admissible $G_{\Delta}^{\Delta}$, $G^{\Sigma}$, $G_{(\Sigma)}$, $G$ and some related quotient groups.}\label{perifanos}
\begin{tabular}{cccccccc}
               \hline
			Line & $G_{\Delta}^{\Delta}$ & $G_{(\Sigma)}^{\Delta}$ & $G^{\Sigma}$ & $G_{\Delta}^{\Sigma}$ & $G_{(\Delta)}^{\Sigma}$ & $G_{(\Sigma)}$ & $G$ \\
		        \hline
		      1& $AGL_{1}(8)$ & $2^{3}$ & $AGL_{1}(8)$ & $7$ & $1$ & $2^{e}$ & $(2^{e}.2^{3}):7$ \\
			2& $A\Gamma L_{1}(8)$ & $2^{3}$ & $A\Gamma L_{1}(8)$ & $7:3$ & $1$ & $2^{e}$ & $(2^{e}.2^{3}):(7:3)$  \\
		    3&    $A\Gamma L_{1}(8)$ & $2^{3}$ & $PSL_{2}(7)$ & $7:3$ & $1$ & $2^{e}$ & $2^{e}.PSL_{2}(7)$\\
		      \hline
		        \end{tabular}
\end{table}
\end{enumerate}
\end{lemma}

\begin{proof}
Let $\mathcal{D}$ be a $2$-$(64,28,12)$ design as in Theorem \ref{TheList}(6). Then $G_{\Delta }^{\Delta }\leq ASL_{3}(2)$, and hence $%
G_{(\Sigma )}$ is a $\left\{ 2,3,7\right\} $-group since $G_{(\Sigma )}\leq \prod_{\Delta \in \Sigma }G_{\Delta }^{\Delta }$.

Assume that $G_{(\Sigma
)}$ contains a non-trivial $7$-element, say $\gamma $. Then $\gamma $ fixes
at least a point on each element of $\Sigma $ since $v_{0}=8$. Let $x\in
\Delta _{1}\in \Sigma $ and $y\in \Delta _{2}\in \Sigma \setminus \left\{
\Delta _{1}\right\} $. Then $\gamma $ fixes at least one of the $\lambda =12$
blocks containing $x$ and $y$, say $B$. Then $\gamma $ fixes pointwise $%
B\cap \Delta $ for any $\Delta \in \Sigma $ such that $B\cap \Delta \neq
\varnothing $ since $k_{0}=4$, and hence $\gamma \in G_{(\Delta )}$ since $G_{\Delta
}^{\Delta }\leq AGL_{3}(2)$ and any Sylow $7$-subgroup of $AG
L_{3}(2)$ fixes a unique point of $\Delta $. Thus, $\gamma $ fixes pointwise
each element of $\Sigma $, a contradiction. Thus, $G_{(\Sigma )}$ is a $%
\left\{ 2,3\right\} $-group. If there is $\Delta _{0}\in \Sigma $ such that $%
3$ divides the order of $G_{(\Sigma )}^{\Delta _{0}}$, then $G_{(\Sigma
)}^{\Delta _{0}}=G_{\Delta _{0}}^{\Delta _{0}}$ since $Soc(G_{\Delta _{0}}^{\Delta _{0}}) \unlhd G_{(\Sigma )}^{\Delta
_{0}}\trianglelefteq G_{\Delta _{0}}^{\Delta _{0}}$ by Lemma \ref{soc} and $G_{\Delta _{0}}^{\Delta _{0}}$ is isomorphic to one of the groups $AGL_{1}(8)$, $A\Gamma L_{1}(8)$ or $AGL_{3}(2)$ by Theorem \ref{TheList}(6.a). So, the
order of $G_{(\Sigma )}^{\Delta _{0}}$, and hence that of $G_{(\Sigma )}$
is divisible by $7$, a contradiction. Thus $7$ does not divide the order of $%
G_{(\Sigma )}^{\Delta }$ for any $\Delta \in \Sigma $. Now, if $3$ divides
the order of $G_{(\Sigma )}$, then the previous argument implies that $%
G_{(\Sigma )}\cap G_{(\Delta )}$ contains all Sylow $3$-subgroups of $%
G_{(\Sigma )}$ for any $\Delta \in \Sigma 
$ since $G_{(\Sigma )}\cap G_{(\Delta )}\trianglelefteq G_{(\Sigma )}$. Hence, any Sylow $3$-subgroup of $G_{(\Sigma)}$ fixes pointwise each element of $\Sigma $, a contradiction. Thus, $%
G_{(\Sigma )}$ is a (possibly trivial) $2$-group.

The group $G^{\Delta}_{(\Sigma )}$ is a $2$-group being a quotient group of the $2$-group $G_{(\Sigma )}$. On the other hand, $Soc(G_{\Delta}^{\Delta}) \unlhd G^{\Delta}_{(\Sigma )} \unlhd G^{\Delta}_{\Delta}$ by Lemma \ref{soc}. Therefore, $G^{\Delta}_{\Delta}/G_{(\Sigma)}^{\Delta}$ is isomorphic to $7$, $7:3$ or $PSL_{2}(7)$ according to whether $G^{\Delta}_{\Delta}$ is isomorphic to $AGL_{1}(8)$, $A\Gamma L_{1}(8)$, or $AGL_{3}(2)$, respectively. Then $G^{\Sigma}_{\Delta}/G_{(\Delta)}^{\Sigma}$ is isomorphic to $7$, $7:3$ or $PSL_{2}(7)$ by (\ref{toulaxiston}). On the other hand, by Theorem \ref{TheList}(6.b), the group $G^{\Sigma}_{\Delta}$ is isomorphic to $7$, $7:3$, $7:3$, $7:6$, $A_{7}$ or $S_{7}$, according to whether $G^{\Sigma}_{\Delta}$ is isomorphic to $AGL_{1}(8)$, $A\Gamma L_{1}(8)$, $PSL_{2}(7)$, $PGL_{2}(7)$, $A_{8}$, $S_{8}$, respectively. Matching the previous information, we see that the unique possibilities are those as in columns 2--6 of Table \ref{perifanos}. Then $G_{(\Delta)}\leq G_{(\Sigma)}$ for each $\Delta \in \Sigma$. Moreover, $\mathcal{D}_{0}\cong AG_{3}(2)$ by Theorem \ref{TheList}(6.a) since $G_{\Delta}^{\Delta}$ is isomorphic to $ A \Gamma L_{1}(8)$ $ A \Gamma L_{1}(8)$. Thus, we obtain (3) and the last part of (1).

Since $G_{(\Sigma )}$ is a $2$-group, $G_{(\Sigma )}\leq \prod_{\Delta \in \Sigma }G_{\Delta }^{\Delta }$ and, $G_{\Delta}^{\Delta}$ is isomorphic to $ A \Gamma L_{1}(8)$ $ A \Gamma L_{1}(8)$, it  follows that $%
G_{(\Sigma )}$ is an elementary abelian $2$-group of order $2^{e}$ with $e \geq 0$. Moreover, $G_{(\Sigma)}$ acts transitively on each $\Delta \in \Sigma$ by Lemma \ref{soc} and Theorem \ref{TheList}(6.a), and hence $e\geq 3$. This completes the proof of (1). Now, $G_{\Sigma}$ and $G$ are as in Columns 7--8 of Table \ref{perifanos}. In particular, $G \cong (2^{e}.2^{3}):7$ or $(2^{e}.2^{3}):(7:3)$ as Lines 1--2 or Table \ref{perifanos} by \cite[Theorem 6.2.1(i)]{Go}. This completes the proof.   
\end{proof}

\bigskip

Throughout the remainder of this section, $G_{(\Sigma)}$ is simply by $V$. Hence, $V$ is an elementary abelian $2$-group of order $2^{e}$, $e\geq 3$, acting transitively on each $\Delta \in \Sigma$. Moreover, $G_{(\Delta)}\leq G_{(\Sigma)}$ for each $\Delta \in \Sigma$.

Finally, $\mathcal{D}_{0}^{(i)}$ will denote the isomorphic copy of $\mathcal{D}_{0}$ having $\Delta_{i}$ as a point set, where $\Delta_{i} \in \Sigma$ for $i=1,...,8$.

\bigskip

\begin{lemma}\label{vdropi}
The group $G$ acts $2$-transitively on the set $\{V_{(\Delta_{i})}:i=1,...,8\}$.
\end{lemma}
\begin{proof}
Set $\mathcal{V}=\{V_{(\Delta_{i})}:i=1,...,8\}$, then $\left\vert \mathcal{V}\right\vert \leq 8$. If $\left\vert \mathcal{V}\right\vert < 8$, then there are $1\leq i_{0},j_{0} \leq 8$ with $i_{0}\neq j_{0}$ such that $V_{(\Delta_{i_{0}})}=V_{(\Delta_{j_{0}})}$. Since $V_{(\Delta_{i_{0}})} \unlhd G_{\Delta_{i_{0}}}$, and for each $g\in G$ one has $V_{(\Delta_{i_{0}})}^{g}=V_{(\Delta_{i_{0}}^{g})}$, and $G_{\Delta_{i_{0}}}$ acts transitively on $\Sigma \setminus \{\Delta_{i_{0}}\}$, it follows that $V_{(\Delta_{i_{0}})}$ fixes pointwise each element of $\Sigma$. So $V_{(\Delta_{i_{0}})}$ fixes $\mathcal{D}$ pointwise, and hence $V_{(\Delta_{i_{0}})}=1$. Therefore, $\left\vert G\right\vert
=2^{6}\cdot c\cdot 7$ with $c=1$ or $3$ by Lemma \ref{ftinoporo}(3) since $V=G_{(\Sigma)}$ and $\left\vert V\right\vert =2^{e}$ and $\left\vert V_{(\Delta_{i_{0}})}\right\vert =2^{e-3}$. Then $\left\vert
G_{B}\right\vert = c \cdot 7$, where $B$ is any block of $\mathcal{D}$ since $b=v=2^{6}$, and we reach a contradiction since $G_{B}$ acts transitively
on $B$ and $k=28$ because $G$ acts flag-transitively on $\mathcal{D}$. Thus $\left\vert \mathcal{V}\right\vert = 8$. Moreover, since $G$ acts $2$-transitively on $\Sigma$ and for each $g\in G$ one has $V_{(\Delta_{j})}^{g}=V_{(\Delta_{j}^{g})}$ with $j=1,...,8$, it follows that $G$ acts $2$-transitively on $\mathcal{V}$.      
\end{proof}

\bigskip

\begin{lemma}
\label{Anixi}The following hold:

\begin{enumerate}
\item For each $1\leq i\leq 8$, let $\mathcal{B}_{i}$ be the set blocks of $%
\mathcal{D}$ which disjoint from $\Delta _{i}$. Then $\mathcal{B}_{i}$ is $V$%
-orbit of length $2^{3}$.

\item The $\mathcal{B}_{1},...,\mathcal{B}_{8}$ are all the $V$-orbits on
the block set of $\mathcal{D}$.

\item Let $\mathcal{B}_{i}$, then for each $1\leq j\leq 8$ and $j\neq i$
there is a unique parallel class $\mathcal{C}_{ij}=\left\{ \pi _{ij},\pi
_{ij}^{\prime }\right\} $ of $\mathcal{D}_{0}^{(j)}\cong AG_{3}(2)$
consisting of planes such that that for $2^{2}$ elements of $\mathcal{B}_{i}$
the intersection with $\Delta _{j}$ is $\pi _{ij}$, and for the remaining $%
2^{2}$ ones is $\pi _{ij}^{\prime }$.

\item For each $1\leq i\leq 8$, let $\mathcal{C}_{i}=\bigcup_{j=1,j\neq
i}^{8}\mathcal{C}_{ij}$. Then $\mathcal{C}_{s}\cap \mathcal{C}%
_{t}=\varnothing $ for each  $1\leq s,t\leq 8$ with $s\neq t$.

\item Let $B\in \mathcal{B}_{s}$ and $B^{\prime }\in \mathcal{B}_{t}$ with $%
s\neq t$. Then there are precisely $6$ elements $\Delta $ of $\Sigma $ such
that $\left\vert B\cap B^{\prime }\cap \Delta \right\vert =2$.
\end{enumerate}
\end{lemma}

\begin{proof}
Let $\Delta _{i}\in \Sigma $ and let $\mathcal{B}_{i}$ be the set blocks of $%
\mathcal{D}$ which disjoint from $\Delta _{i}$. Then $\mathcal{B}_{i}\neq
\varnothing $ since $k_{1}=7$ and $G$ acts transitively on $\Sigma $. Let $%
B\in \mathcal{B}_{i}$, then $B$ intersects each $\Delta _{j}$ in $\Sigma
\setminus \left\{ \Delta _{i}\right\} $ in a non empty set since $k_{1}=7$
and $k_{0}=4$. Thus, $\left\vert \mathcal{B}_{i}\right\vert =8$ since $\mu =8
$. Moreover, $\mathcal{B}_{i}$ is a union $V$-orbits since $V$ preserves $%
\Delta _{i}$, and $\mathcal{B}_{i}\cap \mathcal{B}_{j}=\varnothing $
for $i\neq j$ since $\Delta _{i}\neq \Delta _{j}$. On the other hand, $V$
has exactly $8$ block-orbits on $\mathcal{D}$ by \cite[Theorem 1,46]{HP}
since $V$ has exactly $8$ point-orbits on $\mathcal{D}$, namely the elements
of $\Sigma $. Thus, each $\mathcal{B}_{i}$ is $V$-orbit of length $8$ and $%
\mathcal{B}_{1},...,\mathcal{B}_{8}$ are all the $V$-orbits on the block set
of $\mathcal{D}$, which are the assertion (1) and (2).

Clearly, the group $V$ induces the (full) translation group on each element of $%
\Sigma $. Moreover, each $\mathcal{D}_{0}^{(j)}\cong AG_{3}(2)$ admits a
plane parallelism and each parallel class has size $2$. Thus, if $B\in \mathcal{B%
}_{i}$, then $V$ preserves in $\mathcal{D}_{0}^{(j)}$ the parallel class
determined by the plane $B\cap \Delta _{j}$. Since $(B\cap \Delta _{j})^{V}$
and $\mathcal{B}_{i}$ are both $V$ orbits of length $2$ and $8$,
respectively, it follows from \cite[1.2.6]{Demb} that, exactly $4$ of the blocks in $\mathcal{B}_{i}$
intersect $\Delta _{j}$ in $B\cap \Delta _{j}$, and the remaining $4$
intersect $\Delta _{j}$ in the plane in $\mathcal{D}_{0}^{(j)}$ parallel to $%
B\cap \Delta _{j}$. Therefore, each $V$-orbit $\mathcal{B}_{i}$ determines a
unique parallel class $\mathcal{C}_{ij}=\left\{ \pi _{ij},\pi _{ij}^{\prime
}\right\} $ in $\mathcal{D}_{0}^{(j)}$ for $1\leq j\leq 8$ and $j\neq i$
such that the intersection set of half elements of $\mathcal{B}_{i}$ with $\Delta _{j}$ is the
plane $\pi _{ij}$, and the intersection set of the remaining half ones is $\pi _{ij}^{\prime }$. This
proves (3).

Assume that there are $1\leq s,t\leq 8$ with $s\neq t$ such that $\mathcal{C}%
_{s}\cap \mathcal{C}_{t}\neq \varnothing $. Then there is $1\leq m,n\leq 8$
with $m\neq s$ and $n\neq t$ such that $\pi \in \mathcal{C}_{sm}\cap 
\mathcal{C}_{tn}$. Then $m=n$ by the definition of $\mathcal{C}_{ij}$.
Hence, $\pi \in \mathcal{C}_{sm}\cap \mathcal{C}_{tm}$. Then, by (3) there
are $4$ blocks in $\mathcal{B}_{s}$ and other $4$ in $\mathcal{B}_{t}$ whose
intersection with $\Delta _{m}$ is $\pi $. Then there are at least $8$
blocks of $\mathcal{D}$ such that the intersection of any of them with $%
\Delta _{m}$ is $\pi $, which is not the case since $\theta =4$. This proves
(4).

Finally, assume that $B\in \mathcal{B}_{s}$ and $B^{\prime }\in \mathcal{B}%
_{t}$ with $s\neq t$. Since $\mathcal{C}_{s}\cap \mathcal{C}_{t}=\varnothing 
$ by (4), for each $1\leq j\leq 8$ with $j\neq s,t$, the set $B\cap \Delta
_{j}$ and $B^{\prime }\cap \Delta _{j}$ are planes of $\mathcal{D}%
_{0}^{(j)}\cong AG_{3}(2)$ lying in different parallel classes. Thus $B\cap
B^{\prime }\cap \Delta _{j}$ is a $1$-space of $\mathcal{D}_{0}^{(j)}$ again
since this one is $AG_{3}(2)$. This proves (5). 
\end{proof}

\bigskip

\begin{lemma}
\label{kalokairi}$\left\vert V\right\vert=2^{e}$ with $5\leq e \leq 11$.
\end{lemma}

\begin{proof}
Let $\Delta _{1}, \Delta _{2},...,\Delta_{h}$ be distinct elements of $\Sigma$, where $1\leq h \leq 8$, then
\begin{equation*}
\left\vert V:V_{(\Delta _{1}\cup \Delta _{2} \cup...\cup\Delta
_{j})}\right\vert \leq 2^{3h}\text{.}
\end{equation*}
If $V_{(\Delta _{1}\cup \Delta _{2} \cup...\cup\Delta
_{h})}\neq 1$, then $2^{3}\cdot h\leq 2^{5}$ by \cite[Corollary 3.7]{Land} since 
\begin{equation*}
\left\vert\mathrm{Fix}(V_{(\Delta _{1}\cup \Delta _{2} \cup...\cup\Delta
_{h})}) \right\vert \geq 2^{3}\cdot h\text{.}
\end{equation*}
Therefore $h \leq 4$, and hence $\left\vert V \right\vert = 2^{e}$ with $e \leq 15$.

Let $B\in \mathcal{B}_{s}$ and $B^{\prime }\in \mathcal{B}_{t}$ with $s\neq t
$. Then there are precisely $6$ elements $\Delta $ of $\Sigma $ such that $%
\left\vert B\cap B^{\prime }\cap \Delta \right\vert =2$ by Lemma \ref{Anixi}%
(5). Moreover, $\left\vert V:V_{B,B^{^{\prime }}}\right\vert \leq 2^{6}$
since $\left\vert \mathcal{B}_{s}\right\vert =\left\vert \mathcal{B}%
_{t}\right\vert =2^{3}$ Lemma \ref{Anixi}(1)--(2). Then $\left\vert
V:V_{B,B^{^{\prime }},(\Delta _{1}\cup ...\cup \Delta _{5})}\right\vert \leq
2^{11}$ by choosing distinct $\Delta _{1},...,\Delta _{5}$ among the $6$
elements $\Delta $ of $\Sigma $ such that $\left\vert B\cap B^{\prime }\cap
\Delta \right\vert =2$. Then $\left\vert V\right\vert \leq 2^{11}$ since $%
V_{B,B^{^{\prime }},(\Delta _{1}\cup ...\cup \Delta _{5})}\leq V_{(\Delta
_{1}\cup ...\cup \Delta _{5})}=1$. Thus, $e \leq 11$.

Finally, $\left\vert G\right\vert
=2^{e+3}\cdot c\cdot 7$ with $c=1$ or $3$ by Lemma \ref{ftinoporo}(3) since $V=G_{(\Sigma)}$ and $\left\vert V\right\vert =2^{e}$. Then $\left\vert
G_{B}\right\vert =2^{e-3}\cdot c \cdot 7$ since $b=v=2^{6}$, and hence $e-3 \geq 2$ since $G_{B}$ acts transitively
on $B$ and $k=28$. Thus $5 \leq e \leq 11$, which is the assertion.
\end{proof}

\bigskip

\begin{proposition}\label{Notos}
Let $\mathcal{D}$ be a symmetric $2$-$(64,28,12)$ design admitting a
flag-transitive point-imprimitive automorphism group $G$ preserving a point-partition $\Sigma$ of $\mathcal{D}$. If $G^{\Sigma}\cong PSL_{2}(7)$, then one of the following holds:
\begin{enumerate}
    \item $\mathcal{D}$ is one of the two $2$-designs constructed in Section \ref{Sec2}, and $G=Aut(\mathcal{D})\cong 2^{8}:PSL_{2}(7)$;
    \item $\mathcal{D}$ is the $2$-design $\mathcal{S}^{-}(3)$ constructed in \cite{Ka75}, $G$ is one of the groups $2^{3}:(2^{3}:PSL_{2}(7))$, $2^{6}:PSL_{2}(7)$, $2^{6}.PSL_{2}(7)$, and $Aut(\mathcal{D})\cong 2^{6}:Sp_{6}(2)$.   
\end{enumerate}
\end{proposition}

\begin{proof}
We are going to prove the assertion in a series of steps.

\smallskip
\noindent \textbf{(I).} $C_{G}(V) = V$, $G/V=G^{\Sigma}\cong PSL_{2}(7)$ is isomorphic to a subgroup of $GL(V)$.

Let $C=C_{G}(V)$ and suppose that $V<C$. Then $C^{\Sigma
}\neq 1$, and hence $C/V=C^{\Sigma } \cong PSL_{2}(7)$ since $C \unlhd G$ and $G^{\Sigma }\cong PSL_{2}(7)$ by Lemma \ref{ftinoporo}(2). Thus $V=Z(C)$, and hence $C$ is a perfect covering of $PSL_{2}(7)$. Then $\left\vert V\right \vert \leq 2$ by \cite{At}, whereas $\left\vert V\right \vert \geq 2^{3}$ by Lemma \ref{ftinoporo}(1). Thus $C=V$, and hence $G^{\Sigma}=G/V$ is isomorphic to a subgroup of $GL(V)$ since $V=G_{(\Sigma )}$.

\smallskip
\noindent \textbf{(II).} $G$ is a perfect group and $Z(G)=1$.

It follows immediately from (I) that $G^{\prime}=G$, and hence $G$ is perfect. Assume that $C_{V}(G)\neq 1$. Let $\xi \in C_{V}(G)$, $\xi \neq 1$, and let $%
\gamma $ be a $7$-element of $G$. Clearly, $\gamma $ fixes exactly one
element $x$ of $\mathcal{D}$ which is necessarily fixed by $\xi $. Then $\xi 
$ fixes $\Delta $ pointwise since $\xi \in V$ and $V$ ia abelian and acts transitively on $\Delta$. Now, let $g\in G$ such that $\Delta^{g}\neq \Delta$. Then $\xi$ fixes $\Delta^{g}$ pointwise since $\xi$ and $g$ commute. Then $\xi$ fixes pointwise each element of $\Sigma$ since it commutes with $\left\langle \xi \right\rangle$ since this one is transitive on $\Sigma\setminus \{\Delta\}$, a contradiction. Thus $C_{V}(G)= 1$ and, in particular, $Z(G)=1$. 

\smallskip
\noindent \textbf{(III).} Determination of $G$ and $\mathcal{D}$. 

Since $G$ is a perfect non-central extension of $PSL_{2}(7)$ by an elementary abelian $2$-group of order $2^{e}$ with $5 \leq e \leq 11$, by \cite[Section 5.8]{Holple}, $G$ is a perfect group as in Table \ref{Finsymm}. In particular, $e=6,8,9$ or $11$, and in $G$ is one of the perfect groups contained in the library $\texttt{PERFECT GROUPS}$ (based on \cite{Holple}) of $\texttt{GAP}$ \cite{GAP}. 

\begin{sidewaystable}[p!]
\caption{flag-transitive point-imprimitive symmetric $2$-$(64,28,12)$ designs.}\label{Finsymm}
{\small\setlength{\tabcolsep}{2pt}
\begin{tabular}{|c|c|c|c|c|c|c|}
\hline
Line &$e$ & Type & Structure Description of $G$ & $\mathcal{D}$ & Isomorphism
classes of $\mathcal{D}$ & $Aut(\mathcal{D})$ \\
\hline
\hline
1 & $6$ & $(10752,1)$ & $2^{3}:(2^{3}:PSL_{2}(7))$ & No &  &    \\ 
2 &      & $(10752,2)$ & $2^{3}:(2^{3}.PSL_{2}(7))$ & No &  &    \\ 
3 &      & $(10752,5)$ & $2^{3}:(2^{3}:PSL_{2}(7))$ & $\mathcal{S}^{-}(3)$ & $1$ & $2^{6}:Sp_{6}(2)$   \\ 
4 &      & $(10752,6)$ & $2^{3}:(2^{3}.PSL_{2}(7))$ & No &  &    \\ 
5 &      & $(10752,7)$ & $2^{6}.PSL_{2}(7)$ & No &  &    \\ 
6 &      &$(10752,8)$ & $2^{6}:PSL_{2}(7)$ & $\mathcal{S}^{-}(3)$ & $1$ & $2^{6}:Sp_{6}(2)$   \\ 
7 &      &$(10752,9)$ & $2^{6}.PSL_{2}(7)$ & $\mathcal{S}^{-}(3)$ & $1$ & $2^{6}:Sp_{6}(2)$   \\
\hline
8 &  $8$ & $(43008,25)$ & $2^{8}:PSL_{2}(7)$ & See Section \ref{Sec2}& $2$ & $2^{8}:PSL_{2}(7)$   \\
\hline
9  & $9$ & $(86016,1)$ & $2^{3}:(2^{3}:(2^{3}:PSL_{2}(7)))$ & No &  &    \\ 
10 &     &$(86016,2)$ & $2^{3}:(2^{3}:(2^{3}.PSL_{2}(7)))$ & No &  &    \\ 
11 &     &$(86016,4)$ & $2^{3}:(2^{3}:(2^{3}:PSL_{2}(7)))$ & No &  &    \\ 
12 &     &$(86016,5)$ & $2^{3}:(2^{6}:PSL_{2}(7))$ & No &  &    \\ 
13 &     &$(86016,6)$ & $2^{3}:(2^{3}:(2^{3}.PSL_{2}(7)))$ & No &  &    \\ 
14 &     &$(86016,11)$ & $2^{3}:(2^{3}:(2^{3}.PSL_{2}(7)))$ & No &  &    \\ 
15 &     &$(86016,12)$ & $2^{6}:(2^{3}.PSL_{2}(7))$ & No &  &    \\ 
16 &     &$(86016,13)$ & $2^{3}:(2^{6}:PSL_{2}(7))$ & No &  &    \\ 
17 &     &$(86016,14)$ & $2^{3}:(2^{6}.PSL_{2}(7))$ & No &  &    \\ 
18 &     & $(86016,29)$ & $2^{3}:(2^{6}:PSL_{2}(7))$ & No &  &    \\ 
19 &     & $(86016,31)$ & $2^{3}:(2^{6}.PSL_{2}(7))$ & No &  &    \\ 
20 &     & $(86016,33)$ & $2^{6}:(2^{3}.PSL_{2}(7))$ & No &  &    \\ 
21 &     & $(86016,35)$ & $2^{9}.PSL_{2}(7)$ & No &  &    \\ 
22 &     & $(86016,38)$ & $2^{9}:PSL_{2}(7)$ & No &  &    \\
\hline
23 & $11$ & $(344064,35)$ & $(2^{3}\times 2^{8}):PSL_{2}(7)$ & No &  &    \\ 
24 &    & $(344064,36)$ & $2^{8}:(2^{3}.PSL_{2}(7))$ & No &  &   \\
\hline
\end{tabular}}
\end{sidewaystable}
If $e \neq 11$, we constructed an algorithm that works as follows
\begin{itemize}
    \item We select, up to conjugation, all the subgroups of $G$ of index $2^{6}$ and for each of these we evaluate the action on the right cosets of such subgroups. In this way we obtain all the inequivalent transitive permutation representations of $G$ of degree $2^{6}$;
    \item In any of the previous representation, up to conjugation we determine all possible subgroups having at least one orbit of length $28$, and we record each such orbit.
    \item Now, having the transitive permutation representations of $G$ of degree $2^{6}$, and for each of these all orbits of length $28$ of possibly some subgroups of $G$, we use the package $\texttt{DESIGN}$ \cite{Design} of $\texttt{GAP}$ \cite{GAP} to determine, up to isomorphism, all $2$-$(64,28,12)$ design admitting $G$ as a flag-transitive transitive automorphism, and evaluate the corresponding full automorphism group.  
\end{itemize}
The output is that $G$, $\mathcal{D}$ and isomorphism classes of this one are as in the statement of the proposition.

If $e=11$, then $G$ is the perfect group of type $(344064,35)$ or $(344064,36)$ by \cite[Section 5.8]{Holple}. Then $G$ is $(W \times U):PSL_{2}(7)$ or $W:(U.PSL_{2}(7))$, where $W$ and $U$ are elementary abelian of order $2^8$ or $2^3$, respectively. Moreover, in both cases $V=UV$ is abelian of order $2^{11}$, $W$ is $G$-invariant and the group induced on $W$ by $G$ is $PSL_{2}(7)$ in its (linear) absolutely irreducible $8$-dimensional representation. Finally, denote by $M$ the group $U:PSL_{2}(7)$ or $U.PSL_{2}(7)$ according to whether $G$ is $(W \times U):PSL_{2}(7)$ or $W:(U.PSL_{2}(7))$, respectively. 

To simplify the computation we operate the following reduction to identify the candidate subgroups of $G$ to be the stabilizers in $G$ of a point or a block of $\mathcal{D}$.

Let $(x,B)$ be a any flag of $\mathcal{D}$, and let $\Delta_{i} \in \Sigma$ such that $x \in \Delta_{i}$, and $\mathcal{B}_{j}$ be one of the the block-$V$-orbits, defined in Lemma \ref{Anixi}(1),(2), such that $B \in \mathcal{B}_{j}$. Then $G_{x}=V_{(\Delta_{i_{1}})}:(7:3)$ and $G_{B}=V_{(\mathcal{B}_{j})}:(7:3)$. Furthermore, $G_{x} \cap G_{B} \cong 2^{5}:3$ by \cite[Lemmas 1 and 2]{HM}.

Note that, $\{V_{(\Delta_{i})}:1\leq i \leq 8\}$ is a set of at most eight $8$-subspaces of $V$, and $G$ acts on $\Sigma$ inducing $G/V \cong PSL_{2}(7)$ in its $2$-transitive representation of degree $8$ by Lemma \ref{vdropi}. This forces $W \neq V_{(\Delta_{i})}$ for each $i=1,...,8$. Then the well-known Grassmann identity implies that $\left\vert W \cap V_{(\Delta_{i})}\right\vert \geq 2^{5}$. Now, $W^{\Delta} \unlhd G_{\Delta}^{\Delta}$ since $W$ is a normal subgroup of $G$ contained in $V$ for any $\Delta \in \Sigma$. Also, $W^{\Delta}\neq 1$ since, we have seen that, $W \neq V_{(\Delta)}$. Then $2^{3} \cong Soc(G_{\Delta}^{\Delta}) \unlhd W^{\Delta}$ since $G_{\Delta}^{\Delta} \cong A\Gamma L_{1}(8)$ acts primitively on $\Delta_{i}$  by Theorem \ref{TheList}(6.a) and Lemma \ref{ftinoporo}(3). Thus $\left\vert W \cap V_{(\Delta_{i})}\right\vert = 2^{5}$. Then $\{W \cap V_{(\Delta_{i})}:1\leq i \leq 8\}$ is a set of exactly $8$ distinct $5$-subspaces of $W$ on which $M$ acts inducing $PSL_{2}(7)$ in its $2$-transitive representation of degree $8$.

A similar reasoning shows that $\{V_{(\mathcal{B}_{j})}:1\leq j \leq 8\}$ is a set of exactly $8$ distinct $G$-invariant $8$-subspaces of $V$, on which $G$ acts inducing $G/V \cong PSL_{2}(7)$ in its $2$-transitive representation of degree $8$. Moreover, $\{W \cap V_{(\mathcal{B}_{j})}:1\leq j \leq 8\}$ is a set of exactly $8$ distinct $5$-subspaces of $W$ on which $M$ acts inducing $PSL_{2}(7)$ in its $2$-transitive representation of degree $8$.

We have seen in Section \ref{Sec2} that $PSL_{2}(7)$ has exactly two orbits of $5$-subspaces of $W$. Let $X_{1}$ and $X_{2}$ be two $5$-subspaces of $W$, representatives of the aforementioned $2$ orbits, respectively. Then $X_{1}=W \cap V_{(\Delta_{i_{0}})}$ and $X_{1}=W \cap V_{(\mathcal{B}_{j_{0}})}$ for some $1 \leq i_{0},j_{0} \leq 8$. Starting from this, we use $\texttt{GAP}$ \cite{GAP} to proceed:

\begin{itemize}
    \item We search for all $8$-subspaces of $V$ that are distinct from $W$ and contain $X_{1}$ or $X_{2}$, and we found $1395$ candidates containing $X_{1}$ and $1395$ containing $X_{2}$. One of these $8$-subspaces of $V$ is eligible to be $ V_{(\Delta_{i_{0}})}$ and the other $V_{(\mathcal{B}_{j_{0}})}$ for suitable $1 \leq i_{0},j_{0}\leq 8$;   
    \item We filter the $8$-subspaces of $V$ obtained in the previous step with respect to the property that their conjugacy class in $G$ has length $8$, and we found exactly $3$ such conjugacy classes, say $Y_{h}^{G}$ with $h=1,2,3$. Now, from the normalizer in $G$ of $Y_{h}$, we determine the subgroup $Y_{h}:(7:3)$. These groups are the only eligible ones to be the stabilizer of a point, or a block, or possibly both. However, that action of $G$ on the set of the right cosets of $Y_{h}:(7:3)$ is not faithful for two of them say those for $h=1$ or $2$, and the reason is that the normal subgroups $U$ of $G$ is contained in each member of $Y_{h}^{G}$. Hence, the group $G$ has only one admissible transitive permutation representation of degree $2^{6}$, namely the one on the set of the right cosets of $Y_{3}:(7:3)$, and its subdegrees are $1,7,56$. 
\end{itemize}    
From the previous computation we deduce that, if $x$ and $B$ are any point and block of $\mathcal{D}$, then $G_{x}$ and $G_{B}$ to lie in the same conjugacy $G$-class since the actions of $G$ on the point set and the block set of $\mathcal{D}$ are both transitive of degree $2^{6}$. So, $G_{x}$ is also the stabilizer in $G$ of some block of $\mathcal{D}$. This is a contradiction since the subdegrees of $G$ in its on the set of the right cosets of $Y_{3}:(7:3)$ are $1,7,56$. This completes the proof.
\end{proof}

\bigskip

\begin{theorem}\label{Himonas}
Let $\mathcal{D}$ be a symmetric $2$-$(64,28,12)$ design admitting a
flag-transitive point-imprimitive automorphism group $G$. Then one of the following holds:
\begin{enumerate}
    \item $\mathcal{D}$ is one of the two $2$-designs constructed in Section \ref{Sec2}, and $G=Aut(\mathcal{D})\cong 2^{8}:PSL_{2}(7)$;
    \item $\mathcal{D}$ is the $2$-design $\mathcal{S}^{-}(3)$ constructed in\cite{Ka75}, $G$ is one of the groups $2^{3}:(2^{3}:PSL_{2}(7))$, $2^{6}:PSL_{2}(7)$, $2^{6}.PSL_{2}(7)$, $(2^{6}.2^{3}):7$, or $(2^{6}.2^{3}):(7:3)$, and $Aut(\mathcal{D})\cong 2^{6}:Sp_{6}(2)$.   
\end{enumerate}
\end{theorem}

\begin{proof} Let $\Sigma$ be the $G$-invariant point-partition of $\mathcal{D}$. If $G^{\Sigma}\cong PSL_{2}(7)$, the assertion immediately follows from Proposition \ref{Notos}. Hence, in order to complete the proof, we need to settle the cases (1) and (2) of Table \ref{perifanos} in Lemma \ref{ftinoporo}(3). Here, $G=S.U$, where $S$ is a Sylow $2$-subgroup of $G$, $S/V\cong 2^{3}$ and $U$ is either $7$ or $7:3$, respectively. 

Set $C=C_{G}(V)$. Then $V\leq C$ by Lemma \ref{ftinoporo}(1). If $V \neq C$ then $C^{\Sigma}\neq 1$, and hence $C^{\Sigma}\neq 1$  then $C^{\Sigma}$ contains the Sylow $2$-subgroup of $G^{\Sigma}$ by \cite[Theorem 4.3B]{DM} since $G$ acts $2$-transitively on $\Sigma$, being $AGL_{1}(8) \leq G^{\Sigma} \leq A \Gamma L_{1}(8)$ Lemma \ref{ftinoporo}(3). Then $S\leq C$, and hence $V\leq Z(S)$ since $V$ is abelian by Lemma \ref{ftinoporo}(1). Then $V_{(\Delta)}=1$ for each $\Delta\in \Sigma$ since $S$ acts point-transitively on $\mathcal{D}$. Thus $e=3$ since $\left\vert V_{(\Delta)}\right\vert=2^{e-3}$, but this contradict Lemma \ref{kalokairi}. Thus $C=V$, and hence $G^{\Sigma}=G/V\leq GL(V)$.

One has $Z(S)\leq V$ Since $Z(S)\cap V\neq 1$, being $%
V\vartriangleleft S$. Note that, $Z(S)\cap V_{(\Delta _{i})}=1$ for each $%
\Delta _{i}\in \Sigma $ since $S$ acts point-transitively on $\mathcal{D}$.
Then $1\neq Z(S)\cong Z(S)^{\Delta _{i}}\trianglelefteq G_{\Delta }^{\Delta }
$, and hence $Z(S)$ acts transitively on $\mathcal{D}_{0}^{(i)}$ by \cite[%
Theorem 4.3B]{DM} since $AGL_{1}(8)\leq G_{\Delta _{i}}^{\Delta _{i}}\leq
A\Gamma L_{1}(8)$ acts point-primitively on $\mathcal{D}_{0}^{(i)}$.
Actually, $Z(S)$ acts point-regularly on $\mathcal{D}_{0}^{(i)}$ for each $%
i=1,...,8$ since $Z(S)$ is abelian, and so $Z(S)\cong 2^{3}$. Therefore, $%
V=Z(S)V_{(\Delta _{i})}$ with $Z(S)\cap V_{(\Delta _{i})}=1$ for each $%
i=1,...,8$ since $\left\vert Z(S)\right\vert =2^{3}$, $\left\vert V_{(\Delta
_{i})}\right\vert =2^{e-3}$ and $\left\vert V\right\vert =2^{e}$.

Let $\beta \in S\setminus V$ and suppose that $Z(S)<C_{V}(\beta )$. Then is $%
w\in V\setminus Z(S)$ commuting with $\beta $. Now, $w=zy$ for some $z\in
Z(S)$ and $y\in V_{(\Delta _{i_{0}})}$, with $y\neq 1$ and  $1\leq i_{0}\leq
8$, since $V=Z(S)V_{(\Delta _{i_{0}})}$ with $Z(S)\cap V_{(\Delta _{i_{0}})}=1
$. Hence, $w^{\beta }=w=zy$ and $w^{\beta }=\left( zy\right) ^{\beta
}=z^{\beta }y^{\beta }=zy^{\beta }$ since $z\in Z(S)$. Therefore, $y^{\beta
}=y$. Then $\beta $ preserves $\Delta _{i_{0}}$ and hence $\left\vert
S:V\left\langle \beta \right\rangle \right\vert =\left\vert
S/V:V\left\langle \beta \right\rangle /V\right\vert <8$ since $\beta \in
S\setminus V$, whereas $S$ acts transitively on $\Sigma$. Therefore, $C_{V}(\beta )=Z(S)$ for each $\beta \in
S\setminus V$. 

Recall that $G/V=G^{\Sigma }$ and $AGL_{1}(8)\leq G^{\Sigma }\leq A\Gamma
L_{1}(8)$ acts $2$-transitively on $\Sigma $ and on $\left\{ V_{(\Delta
_{j})}:j=1,...,8\right\} $ via conjugation by Lemma \ref{vdropi}. Therefore, $\left\langle \beta
\right\rangle $ has order $4$ and acts$\left\{ V_{(\Delta
_{j})}:j=1,...,8\right\} $ inducing a fixed point free involution. Let $V_{(\Delta _{s})}
$ and $V_{(\Delta _{t})}=V_{(\Delta _{s})}^{\beta }$. Then $V_{(\Delta
_{t})}^{\beta }=V_{(\Delta _{s})}$, and hence $\left( V_{(\Delta _{t})}\cap
V_{(\Delta _{s})}\right) ^{\beta }=V_{(\Delta _{t})}^{\beta }\cap V_{(\Delta
_{s})}^{\beta }=V_{(\Delta _{t})}\cap V_{(\Delta _{s})}$. Now, $\left\vert
V_{(\Delta _{t})}\cap V_{(\Delta _{s})}\right\vert =2^{m}$ for some $0\leq
m\leq e$. If $m>0$, then there is $x\in V_{(\Delta _{t})}\cap V_{(\Delta
_{s})}$, $x\neq 1$, such that $x^{\beta }=x$. Then $x\in Z(S)$ since $%
C_{V}(\beta )=Z(S)$, and hence $w=1$ since $V_{(\Delta _{t})}\cap Z(S)=1$, a
contradiction. Thus $V_{(\Delta _{t})}\cap V_{(\Delta _{s})}=1$ and so $%
2(e-3)\leq e$, and hence $e\leq 6$. Actually, $e=5$ or $6$ since $e\geq 5$
by Lemma \ref{kalokairi}. By using \texttt{GAP} \cite{GAP}, we see that there are no subgroups of $GL_{2}(5)$
isomorphic to $AGL_{1}(8)$ or $A\Gamma L_{1}(8)$ stabilizing both a plane of $%
PG_{4}(2)$ and a (non maximal) line spread of size $8$. Thus $e=6$, and so $G$ is either $(2^{6}.2^{3}):7$ or $(2^{6}.2^{3}):(7:3)$.

 The set $\left\{ V_{(\Delta _{j})}:j=1,...,8\right\} \cup \left\{
Z(S)\right\} $ is a $3$-spread of $V$, and it is regular as the arising translation
plane is Desarguesian as it is order $8$ (see \cite[Theorem 1.4(a)]{Lu}). Therefore, $G/V$ is a
subgroup of $GL_{6}(2)$ stabilizing a regular $3$-spread of $V$, and so $%
G/V\leq \Gamma L_{2}(8)\cong P\Gamma L_{2}(8)\times 7$ by \cite[Theorem 1.10]{Lu}, and hence $G/V\leq
P\Gamma L_{2}(8)$ with $AGL_{1}(8)\leq G/V\leq A\Gamma L_{1}(8)$. Now,
the possible extensions of $V$ by $AGL_{1}(8)$ or $A\Gamma L_{1}(8)$,
which are determined by using \texttt{GAP} \cite{GAP}, are precisely $6$. Clearly, $G$ is one among such extensions. Nevertheless, again by using \texttt{GAP} \cite{GAP}, only one yields a flag-transitive
point-imprimitive $2$-$(64,28,12)$ design, and such $2$-design is isomorphic to $%
\mathcal{S}^{-}(3)$. This completes the proof.

\end{proof}

\bigskip

\begin{proof}[Proof of Theorem \ref{main}]
The assertion immediately follows from Lemma \ref{soc} or Theorem \ref{TheList}, Lemma \ref{SixtyThree} and Theorem \ref{Himonas}, according to whether $G_{(\Sigma)}$ is or is not trivial, respectively.   
\end{proof}

\bigskip

\begin{proof}[Proof of Corollary \ref{cor1}]
The assertion immediately follows from \cite{BGMV} or Theorem \ref{main} according to whether the flag-transitive automorphism of the $2$-design is or is not point-primitive.   
\end{proof}

\section{Appendix}\label{Sec5}
In this section, we prove a series of classification results for the pair $(\mathcal{S},\Gamma)$, where $\mathcal{S}$ is a $2$-$(v,k,\lambda)$ design with specific values of the parameters $v$ and $k$, and $\Gamma$ is flag-transitive automorphism group $\mathcal{S}$. The importance of such results is next explained: a $2$-design $\mathcal{D}$ admitting a flag-transitive automorphism group $G$ 'decomposes' into two remarkable flag-transitive (possibly trivial) $2$-designs $\mathcal{D}_{0}$ and $\mathcal{D}_{1}$ with specific parameters by Theorem \ref{CamZie}. Then $\mathcal{D}_{0}$ and $\mathcal{D}_{1}$ are completely classified when their parameters coincide with those of some $2$-design in this section. This plays a key role in recovering $\mathcal{D}$ from the knowledge of both $\mathcal{D}_{0}$ and $\mathcal{D}_{1}$.

\bigskip

\begin{lemma}
\label{LA1}Let $\mathcal{S}$ be a $2$-$(6,3,\lambda )$ design admitting $%
\Gamma $ as a flag-transitive automorphism group. Then one of the following
holds:

\begin{enumerate}
\item $\mathcal{S}$ is a $2$-$(6,3,2)$ design and $\Gamma \cong A_{5}$;

\item $\mathcal{S}$ is the complete $2$-$(6,3,4)$ design and $\Gamma \cong
S_{5},A_{6},S_{6}$.
\end{enumerate}
\end{lemma}

\begin{proof}
Let $\mathcal{S}$ be a $2$-$(6,3,\lambda )$ design admitting $\Gamma $ as a
flag-transitive automorphism group. Then $\Gamma $ acts point-primitively on $%
\mathcal{S}$ by \cite[2.3.7(e)]{Demb}, and hence either $\Gamma \leq
AGL_{1}(5)$ or $\Gamma $ is isomorphic to one of the groups $%
A_{5},S_{5},A_{6},S_{6}$ by \cite[Table B.4]{DM}. The former is ruled out
since it contradicts that the order $\Gamma $ must be divisible by $3$. Moreover, the
assertion (2) follows when $\Gamma \cong A_{6},S_{6}$ since any of these groups acts
point-$3$-transitively on $\mathcal{S}$.

Finally, assume that $\Gamma \cong A_{5}$ or $S_{5}$. Clearly, we may
identify the points of $\mathcal{S}$ with $PG _{1}(5)$. Now, $\Gamma $
has a unique conjugacy class of cyclic subgroups of order $3$. Any such
subgroup partitions $PG_{1}(5)$ into two orbits of length $3$. Hence, a
block $B$ is any of them. If $\Gamma \cong A_{5}$, then $\left( \Gamma
\right) _{B}\cong S_{3}$ and so $b=10$, $r=5$ and $\lambda =2$.
Hence, we obtain (1) in this case. If $\Gamma \cong S_{5}$, then $\left(
\Gamma \right) _{B}\cong S_{3}<2\times S_{3}$ which switches the two $S_{3}$%
-orbits of length $3$. Thus $b=20$ and hence $r=10$ and $\lambda =4$.
Therefore, we obtain (2) in this case.
\end{proof}

\bigskip

\begin{lemma}
\label{LA2}Let $\mathcal{S}$ be a $2$-$(v,k,\lambda )$ design admitting $%
\Gamma $ as a flag-transitive automorphism group. If $v=k+1$ or $k+2$, then

\begin{enumerate}
\item $\mathcal{S}$ is the complete $2$-$(v,k,\lambda )$ design, and $\Gamma $
acts point-$2$-transitively on $\mathcal{S}$;

\item $k$ is odd, $\Gamma $ is a point-primitive rank $3$ group on $\mathcal{%
S}$ with subdegrees $1,\frac{k+1}{2},\frac{k+1}{2}$. 
\end{enumerate}
\end{lemma}

\begin{proof}
If $v=k+1$, then $\gcd (v-1,k-1)=1$ and hence $\Gamma $ acts point-$2$%
-transitively on $\mathcal{S}$ by \cite[Corollary 4.6]{Ka69}; if $v=k+2$,
then $\gcd (v-1,k-1)=2$ and hence $\Gamma $ is either point-$2$%
-transitively or primitive rank $3$ on $\mathcal{S}$ with $1$,$\frac{v-1}{%
2}$,$\frac{v-1}{2}$ as subdegrees again by \cite[Corollary 4.6]{Ka69}. In the
latter case, $r=\frac{k+1}{\gcd \left( k-1,2\right) }\cdot \frac{\lambda }{%
(k-1/\gcd \left( k-1,2\right) )}$ and hence $\gcd \left( r,\lambda \right) =%
\frac{\lambda }{(k-1/\gcd \left( k-1,2\right) )}$. Since $\frac{r}{\gcd
\left( r,\lambda \right) }=\frac{k+1}{\gcd \left( k-1,2\right) }$ divides $%
\frac{v-1}{2}=\frac{k+1}{2}$ it follows that $k$ is odd. This proves (2). 

If $\Gamma $ acts point-$2$-transitively on $\mathcal{S}$, then $b=\binom{v%
}{k}$ and hence $\mathcal{S}$ is the complete $2$-$(v,k,\lambda )$ design,
which is (1).
\end{proof}

\bigskip 

Let 
\small
\begin{eqnarray*}
I_{1} &=&\left\{
(4,3),(5,3),(5,4),(6,5),(7,6),(8,7),(9,8),(10,9),(11,10),(17,16),\right\} 
\\
I_{2} &=&\left\{(6,4),(7,5), 
(9,7),(10,8),(11,9),(13,11),(19,17)\right\} 
\\
I_{3} &=&\left\{(8,5),(10,7),(13,10),(16,13)\right\}
\\
I_{4} &=&\left\{(16,11),(11,7),(17,13)\right\}
\\
I_{\geq 5} &=&\left\{(16,10),(16,11),(17,11),(19,13),(25,19),(22,15),(29,22),(33,25)\right\}
\\
I&=&I_{1} \cup I_{1} \cup I_{3} \cup I_{4} \cup I_{\geq 5}
\end{eqnarray*}

\bigskip

\normalsize
\begin{lemma}\label{LA53}
Let $\mathcal{S}$ be a $2$-$(v,k,\lambda )$ design, admitting $%
\Gamma $ as a flag-transitive automorphism group. If $\left( v,k\right) \in
I$, then $\Gamma $ acts point-$2$-transitively on $\mathcal{S}$%
. Moreover, one of the following holds:

\begin{enumerate}
\item $\mathcal{S}$ is complete, and either $\Gamma \cong A_{v},S_{v}$ or one
of the following holds:

\begin{enumerate}
\item[(i)]  $\left( v,k\right) =(6,4)$ and $\Gamma \cong $ $S_{5},A_{6},S_{6}$;

\item[(ii)] $\left( v,k\right) =(9,7)$ and $\Gamma \cong PSL_{2}(8),P\Gamma
L_{2}(8),A_{9},S_{9}$;

\item[(iii)] $\left( v,k\right) =(10,8)$ and $\Gamma \cong
PGL_2(9),M_{10},P\Gamma L_2(9),A_{10},S_{10}$;

\item[(iv)] $\left( v,k\right) =(11,10)$ and $\Gamma \cong AGL_{1}(11),
PSL_{2}(11),M_{11},A_{11},S_{11}$;

\item[iv)] $\left( v,k\right) =(17,16)$ and $\Gamma \cong AG
L_{1}(17),PSL_{2}(2^{4}):2^{\varepsilon },0\leq \varepsilon \leq
2,A_{17},S_{17}$.
\end{enumerate}

\item $\left( v,k\right) =(22,15)$ and one of the following holds:

\begin{enumerate}
\item[(i)] $\mathcal{S}$ is a $2$-$(22,15,80)$ design and $\Gamma \cong
M_{22}$;
\item[(ii)] $\mathcal{S}$ is a $2$-$(22,15,160)$ design union of two copies
of the design as in (i), and $\Gamma \cong M_{22}:{2}$;
\item[(iii)] $\mathcal{S}$ is a $2$-$(22,15,560)$ complete design and $\Gamma \cong
M_{22}$,$M_{22}:{2},A_{22},S_{22}$.
\end{enumerate}
\end{enumerate}
\end{lemma}

\begin{proof}
If $\left( v,k\right) \in I_{1} \cup I_{2}$ then $\Gamma $ acts point-$2$-transitively
on $\mathcal{S}$ possibly except for $\left( v,k\right)
=(5,3),(7,5),(9,7),(11,9),(13,11),(19,17)$ and $\Gamma $ acting as
point-primitive rank 3 group on $\mathcal{S}$. In the exceptional cases, if $\left(
v,k\right) \neq (9,7)$, then $v$ is a prime number and so either $\Gamma $ acts point-$2$-transitively on $\mathcal{S}$, or $\Gamma
\leq AGL_{1}(v)$ by \cite[p. 99]{DM}. In the latter $\left\vert \Gamma
\right\vert \mid (k+2)(k+1)$, whereas $\left\vert \Gamma \right\vert $ is
divisible\ by $k>2$ by the flag-transitivity. If $\left( v,k\right) =(9,7)$
then either $\Gamma \leq AGL_{2}(3)$ or $\Gamma \cong
PSL_{2}(8),P\Gamma L_{2}(8),A_{9},S_{9}$ by \cite[Table B.4]{DM}. However,
only the latter cases occur since $\left\vert \Gamma \right\vert $ is divisible by $k=7$, and these groups are $2$-transitive. Therefore, $\Gamma $ acts point-$2$%
-transitively on $\mathcal{S}$ for each $\left( v,k\right) \in I_{1} \cup I_{2}$.
Still, the possibilities are listed in \cite[Table B.4]{DM} which compared
with the fact that $\left\vert \Gamma \right\vert $ is divisible by $k$
implies either $\Gamma \cong A_{v},S_{v}$, or one of the cases (1.i)--(1.iv) holds.

If $\left( v,k\right) \in I_{3} \cup I_{4} \cup I_{\geq 5}$ and $\left( v,k\right) \neq (22,15),(33,25)$,
then either $v$ is a prime, or $k$ is a prime, or $\gcd (v-1,k-1)\leq 4$. In either case, $\Gamma $ acts point-primitively on $\mathcal{S}$ by \cite[p. 99]{DM}, \cite[2.3.7(e)]{Demb} or \cite[Theorems 1.2, 13 and 1.4]{ZZ}. We may apply
Lemma \ref{bounds} to $\mathcal{S}$ to get $v\geq 3$ and since $v=v_{0}^{\prime}\cdot v_{1}^{\prime
}$ and $k=k_{0}^{\prime }\cdot k_{1}^{\prime }$ with $k_{0}^{\prime }\leq v$
and $k_{1}^{\prime }\leq v_{1}^{\prime }$ by \cite[Proposition 2.1]{CZ}, we
see that $\Gamma $ acts point-primitively on $\mathcal{S}$ when $\left(
v,k\right) =(22,15),(33,25)$. Therefore, $\Gamma $ acts point-primitively on 
$\mathcal{S}$ in each case. Therefore, the possibilities are listed in \cite[%
Table B.4]{DM} and, bearing in mind that $\left\vert \Gamma \right\vert $ is
divisible \ by $k$, we obtain either $\Gamma \cong A_{v},S_{v}$ \ or $\left(
v,k\right) =(22,15)$ and $\Gamma \cong M_{22},M_{22}:{2},A_{22},S_{22}$. In the
former case, $\mathcal{S}$ is complete since both $A_{v}$ and $S_{v}$ are point-$k$%
-transitive on $\mathcal{S}$. In the latter case, one can see that the
unique flag-transitive $2$-designs are exactly those listed in
(2.i)--(2.iii) by using \cite{GAP}. 
\end{proof}

\bigskip

\begin{lemma}
\label{LA4}Let $\mathcal{S}$ be a $2$-$(7,k,\lambda )$ design, with $k=3,4$,
admitting $\Gamma $ as a flag-transitive automorphism group. Then $\Gamma $
acts point-primitively on $\mathcal{S}$ and one of the following holds:

\begin{enumerate}
\item $\mathcal{S}\cong PG _{2}(2)$ and $\Gamma \cong 7:3,PSL_{2}(7)$;

\item $\mathcal{S}$ is a $2$-$(7,3,2)$ design, union of two copies of $PG _{2}(2)$, and $\Gamma \cong A GL_{1}(7)$;

\item $\mathcal{S}$ is a $2$-$(7,3,4)$ design and $\Gamma \cong PSL_{2}(7)$;

\item $\mathcal{S}$ is the complete $2$-$(7,3,5)$ design and $\Gamma \cong
A_{7},S_{7}$;

\item $\mathcal{S}\cong \overline{PG_{2}(2)}$ and $\Gamma \cong PSL_{2}(7)$;

\item $\mathcal{S}$ is the complete $2$-$(7,4,10)$ design and $\Gamma \cong
A_{7},S_{7}$;
\end{enumerate}
\end{lemma}

\begin{proof}
If $k=3$ or $4$, then $\Gamma $ acts point-primitively on $\mathcal{S}$ by 
\cite[2.3.7(e)]{Demb} or \cite[Theorem 1.4]{ZZ}, respectively. Then either $%
\Gamma \leq AG L_{1}(7)$, or $\Gamma \cong PSL_{2}(7)$, or $A_{7}\trianglelefteq \Gamma
\leq S_{7}$ by \cite[Table B.4]{DM}. In the latter case, $\mathcal{S}$ is the complete $2$-design and we obtain (4) or (6) according to whether $k=3$ or $4$, respectively. Hence, we can assume that either $\Gamma \leq AGL_{1}(7)$ or $\Gamma \cong PSL_{2}(7)$.

Suppose that $k=3$. If either $\Gamma \cong 7:3$ or $\Gamma \cong PSL_{2}(7)$, we may identify the point-set of $\mathcal{S}$ with $PG _{2}(2)$. Since the blocks are $3$-subsets of $PG _{2}(2)$, they are either lines or triangles. In the former case, we have $%
\mathcal{S}\cong PG _{2}(2)$, which is (1). In the latter, $\Gamma \cong PSL_{2}(7)$ since $%
r=3\lambda$ (with $\lambda>1$) divides the order of $\Gamma $, and hence we obtain (3). Finally, if $\Gamma \cong AGL_{1}(7)$, then $b=14$, $r=6$ and $\lambda =2$, and we have (2). 

Now suppose that $k=4$. Then $\Gamma \cong PSL_{2}(7)$ since $k$ divides the order of $\Gamma $ and $\Gamma \ncong A_{7},S_{7}$ by our assumption. Thus, $\mathcal{S}$ is a $2$-$(7,4,2)$ complementary design of $PG_{2}(2)$ since $\Gamma $ acts point-$2$-transitively on $\mathcal{S}$, which is (5).
\end{proof}

\bigskip

\begin{lemma}
\label{LA5}Let $\mathcal{S}$ be a $2$-$(8,4,\lambda )$ design admitting $%
\Gamma $ as a flag-transitive automorphism group. Then $\Gamma $ acts
point-primitively on $\mathcal{S}$ and one of the following holds:

\begin{enumerate}
\item $\mathcal{S}\cong AG_{3}(2)$ and $\Gamma $ is one
of the groups $AGL_{1}(8),A\Gamma L_{1}(8),AGL_{3}(2),PSL_{2}(7)$;

\item $\mathcal{S}$ is a $2$-$(8,4,6)$ design, union of two copies of $%
AG _{3}(2)$, and $\Gamma \cong PG L_{2}(7)$;

\item $\mathcal{S}$ is a $2$-$(8,4,9)$ design and $\Gamma \cong
PSL_{2}(7),PGL_{2}(7)$;

\item $\mathcal{S}$ is a $2$-$(8,4,12)$ design and $\Gamma \cong AGL_{3}(2)$;

\item $\mathcal{S}$ is the complete $2$-$(8,4,15)$ design and $\Gamma \cong
A_{8},S_{8}$.
\end{enumerate}
\end{lemma}

\begin{proof}
If $(v,k)=(8,4)$, then $\Gamma $ acts point-primitively on $\mathcal{S}$ by 
\cite[2.3.7(e)]{Demb}. Moreover, by \cite[Table B.4]{DM}, and bearing in mind that $r=7\frac{\lambda }{3%
}$, either $\Gamma $
is isomorphic to one of the groups $AGL_{1}(8)$, $A\Gamma
L_{1}(8)$,$AGL_{3}(2),PSL_{2}(7),PGL_{2}(7)$, or $\Gamma $ is isomorphic
to one of the groups $A_{8},S_{8}$ and in this case $\mathcal{S}$ is the complete $2$-$%
(8,4,15)$ design. The latter is the assertion (5), whereas, in the remaining cases, $5$ does not divide the order of $G$, and so $\mathcal{S}$ is the $2$-$%
(8,4,\lambda )$ design with $\lambda =3,6,9,12$ and $r=7\frac{\lambda }{3%
}$ since $\lambda < \binom{6}{2}=15$. 

If $\Gamma$ is one of the group $AGL(3,2)$, $AGL_{1}(8)$ or $A\Gamma L_{1}(8)$, then the translation group of $G$ acts point-regularly on $\mathcal{D}$ since $\Gamma$ acts point-primitively on $\mathcal{S}$. Thus we may identify the point set of $\mathcal{S}$ with that of $AG_{3}(2)$ in a way that the actions of $\Gamma$ on the point-sets of these structures are equivalent. Now, note that $AGL(3,2)$, as well as any of the groups $AG L_{1}(8),A\Gamma L_{1}(8)$, act
point-$2$-transitively and flag-transitively on $AG_{3}(2)$. Also, $%
AGL_{3}(2)$ contains a copy of $PSL_{2}(7)$, and it acts point-$2$-transitively
and flag-transitively on $AG _{3}(2)$ (see \cite[Theorem 2.14]{KanLib}).

Let $B$ be any block of $\mathcal{S}$, then $B$ is a $4$-subset of $AG
_{3}(2)$. Since $AG_{3}(2)$ admits a plane parallelism and each
parallel class consists of exactly two planes, either there is plane $\pi 
$ of $AG_{3}(2)$ such that $B=\pi $, or $\left\vert B\cap \pi ^{\prime
}\right\vert =2$ for any plane $\pi ^{\prime }$ of $AG_{3}(2)$. In the
former case, one has $\mathcal{S}=AG _{3}(2)$, and hence (1) holds true.

If $\left\vert B\cap \pi ^{\prime }\right\vert =2$ for any plane $\pi
^{\prime }$ of $AG_{3}(2)$, then $AGL_{3}(2)_{B}\cong S_{4}$, being $AGL_{3}(2)$ the full collineation group of $AG _{3}(2)$, and hence 
$\mathcal{S}$ is a $2$-$(8,4,12)$ design. Clearly, none of the groups $%
AGL_{1}(8),A\Gamma L_{1}(8),$ or $PSL_{2}(7)$ acts transitively on $%
\mathcal{S}$ since the stabilizer of a point in any of these groups has
order not divisible by $4$. Thus, we obtain (4).

Now, consider $PSL_{2}(7)$ in its natural $2$-transitive action on $PG
_{1}(7)$. By \cite{At}, the subgroups of $PSL_{2}(7)$ having one orbit of length $4$ are
precisely those lying in the conjugacy class of the cyclic subgroups of
order $4$ or in any representative of any of the two conjugacy classes
subgroups isomorphic to $A_{4}$. Let $H_{1}$,$H_{2}$,$H_{3}$ be
representative of these classes. Then each $H_{i}$ partitions $PG
_{1}(7)$ into two orbits of length $4$, and these are switched by $%
N_{PSL_{2}(7)}(H_{i})=D_{8},S_{4},S_{4}$ respectively. Let $B_{i}$ be any of
the two $H_{i}$-orbit of length $4.$ Then $H_{i}=PSL_{2}(7)_{B_{i}}$ for $%
i=1,2,3$, and hence $(PG _{1}(7),B_{i}^{PSL_{2}(7)})$ is a $2$-$%
(8,4,\theta _{i})$ design, with $\theta _{i}=9,3,3$, respectively, since $%
PSL_{2}(7)$ acts $2$-transitively on $PG _{1}(7)$. The action of $%
PSL_{2}(7)$ on $PG _{1}(7)$ is the unique one of degree $8$, hence $%
(PG _{1}(7),B_{2}^{PSL_{2}(7)})\cong (PG_{1}(7),B_{3}^{PSL_{2}(7)})\cong AG_{3}(2)$. Now $PGL_{2}(7)$
acts on $PG_{1}(7)$, stabilizes the $H_{1}^{PSL_{2}(7)}$ and fuses the
classes $H_{2}^{PSL_{2}(7)}$ and $H_{3}^{PSL_{2}(7)}$, and this implies that 
$(PG_{1}(7),B_{1}^{PSL_{2}(7)})=(PG_{1}(7),B_{1}^{PG
L_{2}(7)})$ is a $2$-$(8,4,9)$ design and that $(PG
_{1}(7),B_{1}^{PG L_{2}(7)})=(PG_{1}(7),B_{1}^{PS
L_{2}(7)}\cup B_{3}^{PSL_{2}(7)})$ is a $2$-$(8,4,6)$ design union of two copies of $AG_{2}(3)$ and both admit $%
PGL_{2}(7)$ as a flag-transitive automorphism group. Thus, we obtain (2) and (3). This completes the proof.
\end{proof}

\bigskip

\begin{lemma}
\label{LA6}Let $\mathcal{S}$ be a $2$-$(9,k,\lambda )$ design, with $k=3,5,6$%
,$\,$ admitting $\Gamma $ as a flag-transitive automorphism group. Then $%
\Gamma $ acts point-primitively on $\mathcal{S}$ and one of the following
holds:

\begin{enumerate}
\item $\mathcal{S}\cong AG_{2}(3)$ and $\Gamma \leq AGL_{2}(3)$;

\item $\mathcal{S}$ is a $2$-$(9,3,6)$ design and $\Gamma \cong ASL_{2}(3), AGL_{2}(3)$;

\item $\mathcal{S}$ is the complete $2$-$(9,3,7)$ design and $\Gamma \cong
PSL_{2}(8),P\Gamma L_{2}(8),A_{9},S_{9}$;

\item $\mathcal{S}$ is the complete $2$-$(9,5,35)$ design and $\Gamma \cong A_{9},S_{9}$.

\item $\mathcal{S}\cong \overline{AG_{2}(3)}$ and $\Gamma \leq AGL_{2}(3)$;

\item $\mathcal{S}$ is a $2$-$(9,6,30)$ design complementary to the $2$-design in (2), and $\Gamma$ is one of the groups $ ASL_{2}(3)$, $ AGL_{2}(3)$;

\item $\mathcal{S}$ is the complete $2$-$(9,6,35)$ design and $\Gamma \cong
PSL_{2}(8),P\Gamma L_{2}(8),A_{9},S_{9}$.
\end{enumerate}
\end{lemma}

\begin{proof}
If $k=3$, then $\Gamma $ acts point-primitively on $\mathcal{S}$ by \cite[%
2.3.7(e)]{Demb} since $v=9$. Then either $3^{2}\trianglelefteq \Gamma \leq
AGL_{2}(3)$, or $\Gamma $ is isomorphic to one of the groups $%
PSL_{2}(8),P\Gamma L_{2}(8),A_{9},S_{9}$ by \cite[Table B.4]{DM}. The $7$%
-transitivity of $A_{9}$ or $S_{9}$ implies that $\mathcal{S}$ is a complete 
$2$-$(9,3,7)$ design, which is (3).

Assume that $3^{2}\trianglelefteq \Gamma \leq
AGL_{2}(3)$. Set $T=Soc(AGL_{2}(3))=3^{2}$
and let $B$ be any block of $\mathcal{S}$. Leet $B$ be any block of $\mathcal{S}$. If $T_{B} \neq 1$, then $\mathcal{S}\cong AG_{2}(3)$ and $\Gamma \leq AGL_{2}(3)$, which is (1); if $T_{B}=1$, then $B$ is triangle in $AG_{2}(3)$, $1 \leq \lambda \leq 7$, $9\mid b$, $r=3\frac{b}{v}$ and $3^{2}:3 \leq
\Gamma \leq AG L_{2}(3)$. On the other hand, $r=4\lambda$ implies $3 \mid \lambda$. Therefore, $\lambda=3$ or $6$ since $1 \leq \lambda \leq 7$. The former is ruled out by \cite[Theorem 1.1]{Mo0}, then $\lambda=6$, $r=24$ and hence $AS L_{2}(3) \leq \Gamma \leq AG L_{2}(3)$. Thus, by the point-$2$-transitivity of $\Gamma$ on $AG_{2}(3)$ we have that $\mathcal{D}$ is the $2$-$(9,3,6)$ design whose points and blocks are respectively the points and triangles of $AG_{2}(3)$, and we obtain (2).

Assume that $PSL_{2}(8)\trianglelefteq \Gamma \leq P\Gamma L_{2}(8)$. Note
that, $PSL_{2}(8)$ has unique conjugacy class of cyclic subgroups of order $3
$ and each of these splits $PG_{1}(9)$ into three orbits of length $3$
permuted transitively by a \ cyclic subgroup of $PSL_{2}(8)$ of order $9$.
The stabilizer in $PSL_{2}(8)$ of one of these, say $B$, is $S_{3}$, which
fuses the remaining two orbits of length $3$. Thus $(PG
_{1}(9),B^{PSL_{2}(8)})$ is the complete $2$-$(9,3,7)$ design, and so the
same conclusions hold for $P\Gamma L_{2}(8),A_{9}$, or $S_{9}$. Thus, we obtain (3) also in this case.

If $k=5$, then $\Gamma $ acts point-primitively on $\mathcal{S}$ by \cite[%
Theorem 1.4]{ZZ} since $v=9$. Then $\Gamma $ is isomorphic to one of the
groups $A_{9},S_{9}$ by \cite[Table B.4]{DM} since $k=5$ must divide the
order of $\Gamma $. The point-$5$-transitivity of $A_{9}$ or $S_{9}$ on $%
\mathcal{S}$ implies that $\mathcal{S}$ is a complete $2$-$(9,5,35)$ design,
which is (4).

Assume that $k=6$, and recall that either $\Gamma \leq AGL_{2}(3)$ ot $\Gamma $ is isomorphic to one of the groups $%
PSL_{2}(8),P\Gamma L_{2}(8),A_{9},S_{9}$. Moreover, $r=8\frac{\lambda}{5}$. Suppose that $\Gamma \leq AGL_{2}(3)$. Let $B$ be any block off $\mathcal{D}$. If $T_{B}\neq 1$, the $T_{B}$ is cyclic of order $3$ since $k=6$. Moreover, $B$ is the union of two distinct $T_{B}$-orbits since $G_{B}$ acts transitively on $B$, and $T_{B} \unlhd G_{B}$. Therefore, $B$ is the union of two (distinct) parallel $1$-spaces of $AG_{2}(3)$. Hence, $\mathcal{S}\cong \overline{AG_{2}(3)}$ and $\Gamma \leq AGL_{2}(3)$, which is (5).

If $T_{B}=1$, then $v=9$ divides $b$ and so $r=6\frac{b}{9}$. Then, $8\frac{\lambda}{5}=6\frac{b}{9}$ implies that $15 \mid \lambda$, $24 \mid r$ and $ASL_{2}(3) \unlhd \Gamma \leq AGL_{2}(3)$. Actually, $\lambda=15$ or $30$ since $\lambda \leq \binom{9-2}{6-2}=35$. Then $\overline{\mathcal{S}}$ is either a $2$-$(9,3,3)$ design, or a $2$-$(9,3,6)$ design, and in both cases $\Gamma$ is a point-$2$-transitive automorphism group of $\overline{\mathcal{S}}$. Actually, the above argument implies that $\overline{\mathcal{S}}$ is the $2$-$(9,3,6)$ design whose points and blocks are respectively the vectors and the triangles of $AG_{2}(3)$. Hence, $\mathcal{S}$ is a $2$-$(9,6,30)$ design complementary to the $2$-design in (2), and $ASL_{2}(3) \unlhd \Gamma \leq AGL_{2}(3)$, which is (6).  

Finally, assume that $\Gamma $ is isomorphic to one of the groups $%
PSL_{2}(8),P\Gamma L_{2}(8),A_{9},S_{9}$ by \cite[Table B.4]{DM}. Then $\overline{\mathcal{S}}$ is a $2$-$(9,3,\lambda )$ design admitting $%
\Gamma $ as a point-$2$-transitively on $\mathcal{S}$. Moreover, the $3$%
-subgroup acting semiregularly on a block of $\overline{\mathcal{S}}$ acts
regularly in its complementary set. Therefore, $\overline{\mathcal{S}}$ is a $2$-$%
(9,3,7)$ design admitting $\Gamma $ acts flag-transitive automorphism
group as in (3). Hence, $\mathcal{S}$ and $\Gamma $ are as in (7). 
\end{proof}

\bigskip

\begin{lemma}
\label{LA7}Let $\mathcal{S}$ be a $2$-$(10,4,\lambda )$ design admitting $%
\Gamma $ as a flag-transitive automorphism group. Then $\Gamma $ acts
point-primitively on $\mathcal{S}$ and one of the following holds:

\begin{enumerate}

\item $\mathcal{S}$ is the $2$-$(10,4,2)$ designs and $\Gamma \cong S_{5},A_{6}, S_{6}$;

\item $\mathcal{S}$ is a $2$-$(10,4,4)$ design, union of two copies of a $2$-design as in (1), and $\Gamma \cong PGL_{2}(9),M_{10},P\Gamma
L_{2}(9)$;

\item $\mathcal{S}$ is a $2$-$(10,4,24)$ design and $\Gamma \cong PGL_{2}(9),M_{10},P\Gamma
L_{2}(9)$;

\item $\mathcal{S}$ is the complete $2$-$(10,4,28)$ design and $\Gamma
\cong A_{10},S_{10}$.
\end{enumerate}
\end{lemma}

\begin{proof}
If $(v,k)=(10,4)$, then $\Gamma $ acts
point-primitively on $\mathcal{S}$ by \cite[Lemma 2.7]{ZZ}, hence $\Gamma $ is isomorphic to one of the groups $%
A_{5}$, $S_{5}$, $PSL_{2}(9)$, $PGL_{2}(9)$, $P\Sigma L_{2}(9)$, $M_{10}$, $P\Gamma
L_{2}(9)$, $A_{10}$, $S_{10}$ by \cite[Table B.4]{DM}. Clearly, both $A_{10}$ and $S_{10}$ lead $\mathcal{S}$ to be the
complete $2$-$(10,4,28)$ design, and we obtain (4). If $\Gamma \cong A_{5},S_{5}$, only the second one admits one subgroup with an orbit of length $4$, namely $S_{4}$, leading $\mathcal{S}$ to be a $2$-$(10,4,2)$ design. Thus, we obtain (1) in this case.

Assume now that $\Gamma $ is isomorphic to one of the
groups $PSL_{2}(9)$, $PGL_{2}(9)$, $P\Sigma
L_{2}(9) \cong S_{6}$, $M_{10}$, $P\Gamma L_{2}(9)$. By \cite{At}, each of these groups has primitive permutation representation of degree $10$, namely the one on $PG_{1}(9)$. Let $\mathcal{O}$ be an orbit of length $4$ of some subgroup of $\Gamma$. Since any Sylow $2$-subgroup of $\Gamma_{\mathcal{O}}$ is contained in a Sylow $2$-subgroup $P\Gamma L_{2}(9)$, we may assume that a Sylow $2$-subgroup of $\Gamma_{\mathcal{O}}$ is contained in the Sylow $2$-subgroup $\left\langle \alpha, \beta , \gamma \right\rangle \cong 8:2^{2}$ of $P\Gamma L_{2}(9)$, where $\alpha: x \rightarrow \omega x$ with $\omega$ a primitive element of $\mathbb{F}_{9}$, $\beta: x \rightarrow x^{-1}$ and $\gamma: x \rightarrow x^{3}$. Thus, we obtain one of the following possibilities:
\begin{enumerate}
    \item[(i)] $\mathcal{O}$ is one of the sets $\{\omega,\omega^{3},\omega^{5},\omega^{7}\}$ or $\{1,\omega^{2},\omega^{4},\omega^{6}\}$, and $\left\langle \alpha, \beta , \gamma \right\rangle_{\mathcal{O}}=\left\langle \alpha^{2}, \beta , \gamma \right\rangle \cong D_{8} \times 2$;
\item[(ii)] $\mathcal{O}$ is an orbit $\{1,\omega^{3},\omega^{4},\omega^{7}\}$ under $\left\langle \alpha, \beta,\gamma  \right\rangle_{\mathcal{O}}=\left\langle \alpha\beta, \alpha\gamma  \right\rangle \cong D_{8}$.
\end{enumerate}Assume that (i) holds. Then $\Gamma_{\mathcal{O}}\cap \left\langle \alpha, \beta , \gamma \right\rangle$ is $\left\langle \alpha^{2}, \beta \right\rangle \cong D_{8}$ or $\left\langle \alpha^{2}, \beta , \gamma \right\rangle$ according to whether $\Gamma$ is $PSL_{2}(9),PGL_{2}(9),M_{10}$, or $P \Sigma L_{2}(9), P\Gamma L_{2}(9)$, respectively. By \cite[Lemmas 6(iv) and 10(v)]{COT}, each of the sets $\{\omega,\omega^{3},\omega^{5},\omega^{7}\}$  or $\{1,\omega^{2},\omega^{4},\omega^{6}\}$ is one orbit under a unique subgroup of $PSL_{2}(9)$ isomorphic to $S_{4}$, and these $S_{4}$, both containing $\left\langle \alpha^{2}, \beta \right\rangle$, belong to two distinct conjugacy $PSL_{2}(9)$-classes, respectively. Now, by \cite{At}, the two conjugacy $PSL_{2}(9)$-classes of subgroups isomorphic to $S_{4}$ are fused in $PGL_{2}(9),M_{10},P\Gamma L_{2}(9)$ but not in $P\Sigma L_{2}(9) \cong S_{6}$. Again by \cite{At}, the stabilizer of $\{\omega,\omega^{3},\omega^{5},\omega^{7}\}$  or $\{1,\omega^{2},\omega^{4},\omega^{6}\}$ is one orbit under a subgroup of $P \Sigma L_{2}(9)$ isomorphic to $S_{4} \times 2$, and these $S_{4}\times 2$ belong to two distinct conjugacy $P \Sigma L_{2}(9)$-classes fused in $P\Gamma L_{2}(9)$. Thus, $\mathcal{S}=(PG_{1}(9),\mathcal{O}^\Gamma)$ is a $2$-$(10,4,\lambda)$ design, where $\lambda = 2$ or $4$ according to whether $\Gamma \cong A_{6}, S_{6}$ or $\Gamma \cong PGL_{2}(9), M_{10}, P\Gamma L_{2}(9)$, respectively, since $\Gamma$ acts $2$-transitively on $PG_{1}(9)$. Thus, we obtain (1) and (2).  

Assume that (ii) holds. Then $\Gamma_{\mathcal{O}}\cap \left\langle \alpha, \beta , \gamma \right\rangle$ is $\left\langle \alpha^{4} \right\rangle \cong 2$ for $\Gamma \cong PSL_{2}(9)$ or $P\Sigma L_{2}(9)$, $\left\langle \alpha^{4}, \alpha\beta \right\rangle \cong 2^{2}$ for $\Gamma \cong PGL_{2}(9)$, $\left\langle \alpha\gamma \right\rangle \cong 4$ for $\Gamma \cong M_{10}$, and $\left\langle \alpha\beta, \alpha\gamma  \right\rangle \cong D_{8}$ for  $\Gamma \cong P\Gamma L_{2}(9)$. Hence, $\Gamma$ is neither isomorphic to $PSL_{2}(9)$ nor to $P\Sigma L_{2}(9)$ since the size of $\mathcal{O}$ is $4$. Moreover, the order of $\Gamma_{\mathcal{O}}$ is not divisible by $5$ since any nontrivial $5$-element in $PSL_{2}(9)$ acts f.p.f. on $PG_{1}(9)$. Assume that the order of $\Gamma_{\mathcal{O}}$ is divisible by $3$. So it is the order of $PSL_{2}(9)_{\mathcal{O}}$. Then $PSL_{2}(9)_{\mathcal{O}} \cong S_{3}$ or $3^{2}:2$ \cite[Theorem 2]{COT} since $PSL_{2}(9) \cap \left\langle \alpha, \beta , \gamma \right\rangle=\left\langle \alpha^{4} \right\rangle \cong 2$. Then  $PSL_{2}(9)_{\mathcal{O}}$ must contain an element $\tau:x \rightarrow x+c$ or $\tau^{\prime}:x \rightarrow \frac{x}{ex+f}$ for some suitable $c,e,f \in \mathbb{F}_{9}$, but none of these fixes an element in $\mathcal{O}$. Thus $\Gamma_{\mathcal{O}} \leq \left\langle \alpha, \beta , \gamma \right\rangle$, and hence $\Gamma_{\mathcal{O}}$ is $\left\langle \alpha^{4}, \alpha\beta \right\rangle \cong 2^{2}$ for $\Gamma \cong PGL_{2}(9)$, $\left\langle \alpha\gamma \right\rangle \cong 4$ for $\Gamma \cong M_{10}$, and $\left\langle \alpha\beta, \alpha\gamma  \right\rangle \cong D_{8}$ for  $\Gamma \cong P\Gamma L_{2}(9)$. Therefore, in this case, $\mathcal{S}=(PG_{1}(9),\mathcal{O}^\Gamma)$ is a $2$-$(10,4,24)$ design by the $2$-transitivity of $\Gamma$ on $PG_{1}(9)$, and we obtain (3). This completes the proof.
\end{proof}

\bigskip

\begin{lemma}
\label{LA8}Let $\mathcal{S}$ be a $2$-$(16,4,\lambda )$ design admitting $%
\Gamma $ as a flag-transitive automorphism group. Then $\Gamma $ acts
point-primitively on $\mathcal{S}$ and one of the following holds:

\begin{enumerate}
\item $\mathcal{S}\cong AG _{2}(4)$ and $\Gamma \cong
2^{4}:5,2^{4}:(5:2),2^{4}:(5:4),AG L_{1}(16),AG
L_{1}(16):2,A\Gamma L_{1}(16),ASL_{2}(4),A\Sigma L_{2}(4),A\Gamma L_{2}(4)$;

\item $\mathcal{S}$ is a $2$-$(16,4,2)$ design and $\Gamma \cong
2^{4}:(5:4),ASL_{2}(4),A\Sigma L_{2}(4)$

\item $\mathcal{S}$ is a $2$-$(16,4,3)$ design and $\Gamma \cong AG
L_{1}(16),AGL_{1}(16):2$;

\item $\mathcal{S}$ is a $2$-$(16,4,3)$ design and $\Gamma \cong
ASL_{2}(4),A\Sigma L_{2}(4),ASp_{4}(2),A\Gamma Sp_{4}(2)$;

\item $\mathcal{S}$ is a $2$-$(16,4,4)$ design and $\Gamma \cong ASp_{4}(2)$;

\item $\mathcal{S}$ is a $2$-$(16,4,6)$ design and $\Gamma \cong
2^{4}:(15:4),AGL_{2}(4),A\Gamma L_{2}(4)$;

\item $\mathcal{S}$ is a $2$-$(16,4,7)$ design and $\Gamma
\cong 2^{4}:A_{7},AGL_{4}(2)$;

\item $\mathcal{S}$ is a $2$-$(16,4,12)$ design and $\Gamma \cong
2^{4}:(15:4)$;

\item $\mathcal{S}$ is a $2$-$(16,4,12)$ design and $\Gamma \cong A\Sigma
L_{2}(4),ASp_{4}(2);$

\item $\mathcal{S}$ is a $2$-$(16,4,36)$ design and $\Gamma \cong A\Gamma
L_{2}(4)$

\item $\mathcal{S}$ is a $2$-$(16,4,84)$ design and $\Gamma \cong
2^{4}:A_{7},AGL_{4}(2)$;

\item $\mathcal{S}$ is the complete $2$-$(16,4,91)$ design and $\Gamma \cong
A_{16},S_{16}$.
\end{enumerate}
\end{lemma}

\begin{proof}
Since $(v-1,k-1)=3$, then $\Gamma $ acts point-primitively on $\mathcal{S}$
by \cite[Theorem 1.4]{ZZ}. Then either $\Gamma \leq AGL_{4}(2)$ or $%
A_{16}\trianglelefteq \Gamma \leq S_{16}$. \ In the latter case, the point-$4
$-transitivity of $\Gamma $ on $\mathcal{S}$ implies that $\mathcal{S}$ is
the complete $2$-$(16,4,91)$ design, which is (12).

Assume that $\Gamma \leq AGL_{4}(2)$. The order of the group $\Gamma $ is
divisible by $5$ since $r_{0}=5\lambda $. If $\lambda =1$, then $\mathcal{S}%
\cong AG_{2}(4)$, and it is easy to see that $\Gamma \cong
2^{4}:5,2^{4}:(5:2),2^{4}:(5:4),AGL_{1}(16),AG
L_{1}(16):2,A\Gamma L_{1}(16),ASL_{2}(4),A\Sigma L_{2}(4),A\Gamma L_{2}(4)$
acts flag-transitively on $\mathcal{S}$. If $\lambda =2,3$ or $4$, then
(2),(3),(4) or (5) by \cite[Theorem 1.2]{TCZ} and \cite[Theorem 1.1]{Mo0}.
Finally, if $\lambda >4$, then (6)--(12) occur by \cite{GAP}.
\end{proof}

\bigskip

\begin{lemma}
\label{LA9} Let $\mathcal{S}$ be a $2$-$(v,k,\lambda )$ design with $%
(v,k)=(13,8)$, $(13,9)$ or $(21,16)$ admitting $\Gamma $ as a flag-transitive
automorphism group. Then $\Gamma $ acts point-$2$-transitively on $\mathcal{S}$ and one of the following holds:

\begin{enumerate}
\item $(v,k)=(13,8)$ and one of the following holds:
\begin{enumerate}
    \item $\mathcal{S}$ is a $2$-$(13,8,42)$ design and $\Gamma\cong PSL_{3}(3)$;
     \item $\mathcal{S}$ is the complete $2$-$(13,8,462)$ design and $\Gamma \cong A_{13},S_{13}$.
\end{enumerate}
\item $(v,k)=(13,9)$ and one of the following holds:
\begin{enumerate}
\item $\mathcal{S}\cong \overline{PG_{2}(3)}$ and $\Gamma\cong PSL_{3}(3)$;
\item $\mathcal{S}$ is the complete $2$-$(13,9,330)$ design and $\Gamma \cong A_{13},S_{13}$.
\end{enumerate}
\item $(v,k)=(21,16)$ and one of the following holds:
\begin{enumerate}
\item $\mathcal{S}\cong  \overline{PG_{2}(4)}$ and $PSL_{3}(4)\trianglelefteq \Gamma\leq P\Gamma L_{3}(4)$;
\item $\mathcal{S}$ is the complete $2$-$(21,16,11628)$ design and $\Gamma \cong A_{21},S_{21}$.
\end{enumerate}
\end{enumerate}
\end{lemma}

\begin{proof}
Suppose that $(v_{1},k_{1})=(21,16)$ and that $\Gamma$ acts point-imprimitively
on $\mathcal{S}$. Then, by 
Theorem \ref{CamZie}, we have that $v=v_{0}^{\prime}\cdot v_{1}^{\prime }$, $k=k_{0}^{\prime }\cdot k_{1}^{\prime
}$ with $2\leq k_{0}^{\prime }\leq v$, $k_{1}^{\prime }\leq v_{1}^{\prime }$ and $(v_0^\prime - 1)/(k_0^\prime-1)=(v - 1)/(k-1)=5/4$, which leads to contradiction since the only possibilities for $v_0^\prime$ are $3$ and $7$. Therefore, $\Gamma $
acts point primitively on $\mathcal{S}$. Moreover, when $(v,k)=(13,8)$ or $(13,9)$ the group $\Gamma$ acts point-primitively on $\mathcal{S}$ since $v$ is a prime. Thus, $\Gamma $ acts point-primitively on $%
\mathcal{S}$ in any case.

If $\Gamma \cong A_{v}$ or $S_{v}$, then $\Gamma$ acts point-$k$-transitively on $\mathcal{S}$ since $k \leq v-2$, and this forces $\mathcal{S}$ to be
the complete $2$-$(v,k,\lambda )$ design, which is (3). Hence, assume that $\Gamma$ is not isomorphic to any of the groups $A_{v}$ or $S_{v}$
respectively. By \cite[Table B.4]{DM} and \cite{At}, and bearing in mind
that $k$ divides the order of $\Gamma$, either $(v,k)=(13,8),(13,9)$ and 
$\Gamma\cong PSL_{3}(3)$, or $(v,k)=(21,16)$ and either $\Gamma
\cong PGL_{2}(7)$ or $PSL_{3}(4)\trianglelefteq \Gamma
\leq P\Gamma L_{3}(4)$. If $(v,k)=(21,16)$ and $\Gamma \cong PG L_{2}(7)$, then $\left\vert \Gamma_{B}\right\vert=16$, where $B$ is any block of $\mathcal{S}$, and hence $\Gamma_{B}$ fixes a point, say $%
x$. So $\Gamma_{B}\leq \Gamma_{x}$, and then each 
$\Gamma_{x}$-orbit is a union of $\Gamma_{B}$%
-orbits. Now, the $\Gamma _{x}$-orbits distinct from $\left\{ x\right\} 
$ have length $4,8,8$, whereas the $\Gamma _{B}$ must have $B$
as an orbit of length $16$, which is a contradiction. Therefore, either $(v,k)=(13,8)$ and $\Gamma \cong PSL_{3}(3)$ or $%
(v,k)=(t^{2}+t+1,t^{2})$ and $PSL_{3}(t)\trianglelefteq \Gamma \leq P\Gamma L_{3}(t)$ with $t=\sqrt{k}=3$ or $4$. 

Assume that the latter occurs. In each case, the
action of $\Gamma$ on the point-set of $\mathcal{S}$ is $2$%
-transitive. Therefore, we may identify the point-set of $\mathcal{S}$ with that of $PG_{2}(t)$, hence $B$ is a suitable $k$-subset of 
$PG_{2}(t)$. In particular, $B$ is not a blocking set of $PG_{2}(t)$ by 
\cite[Corollary 13.12(i)]{Hir}, and hence there is a line $\ell $ of $%
PG_{2}(t)$ such that either $B\cap \ell =\varnothing $ or $\ell
\subset B$. In the latter case, $\ell \subset B=x^{\Gamma _{B}}$ with $x\in B
$. Since $(\ell ^{\Gamma _{B}},x^{\Gamma _{B}})$ is a $1$-design by \cite[1.2.6]{Demb}, it follows
that $\left\vert \ell ^{\Gamma _{B}}\right\vert \cdot (t+1)=\left\vert
x^{\Gamma }\right\vert \cdot u=t^{2}\cdot u$, where $u$ is the number of
lines of $PG_{2}(t)$ containing $x$ and contained in $B$. Then $%
t^{2}\mid \left\vert \ell ^{\Gamma _{B}}\right\vert $ and $5\mid u$, and
actually $\left\vert \ell ^{\Gamma _{B}}\right\vert =t^{2}$ and $u=t+1$
since $\left\vert \ell ^{\Gamma _{B}}\right\vert \leq t^{2}+t+1$. Then, $%
x^{\Gamma _{B}}$ contains all the $t+1$ lines of $PG_{2}(4)$
containing $x$, say $\ell _{1},...,\ell _{t+1}$. Let $\ell ^{\prime }$ be
any line of $PG_{2}(t)$, clearly $\ell ^{\prime }\subset
\bigcup_{i=1}^{t+1}\ell _{i}\subseteq B$, and so $B$ contains all
the $t^{2}+t+1$ lines of $PG_{2}(t)$. Therefore $B$ contains all the $%
t^{2}+t+1$ points of $PG_{2}(t)$, which is not the case since $%
\left\vert B\right\vert =t^{2}$. Thus, $B\cap \ell =\varnothing $ and so $%
\mathcal{S}\cong \overline{PG_{2}(t)}$, hence we obtain (1) and (2) for $t=3$ or $4$, respectively.

Finally, assume that $(v,k)=(13,8)$ and $\Gamma \cong PSL_{3}(3)$. Let $B$ any block of $\mathcal{S}$ Arguing as above, $\Gamma$ acts point-$2$-transitively on $\mathcal{S}$, we may identify the point-set of $\mathcal{S}$ with that of $PG_{2}(3)$, and $B$ is a suitable $8$-subset of $PG_{2}(3)$ that is not a blocking set. Then there is a line $\ell $ of $PG_{2}(3)$ such that either $B\cap \ell =\varnothing $ or $\ell \subset B$. In the latter case, $\ell \subset B=x^{\Gamma _{B}}$ with $x\in \ell$. Since $(\ell ^{\Gamma _{B}},x^{\Gamma _{B}})$ is a $1$-design by \cite[1.2.6]{Demb}, it follows
that $\left\vert \ell ^{\Gamma _{B}}\right\vert \cdot 4=\left\vert
x^{\Gamma }\right\vert \cdot u=8\cdot u$, where $u$ is the number of
lines of $PG_{2}(3)$ containing $x$ and contained in $B$. So, $\left\vert \ell ^{\Gamma _{B}}\right\vert=2u$. If $u=1$, then $B$ consists of two lines and each point of $B$ lies on exactly one of them. So $\left\vert B\right\vert \leq 7$, since any two distinct lines of $PG_{2}(3)$ have always a point in common, a contradiction. Then $u \geq 2$, and hence $B$ contains three lines that are either concurring in a point or lying in a triangular configuration. So, $\left\vert B\right\vert \geq 9$, which is not the case. Thus $B\cap \ell =\varnothing $, and hence there is a unique point $x$ in $PG_{2}(3)$ not in $\ell$ such that $B$ is the complementary set of $\ell \cup \{x\}$ since $\left\vert B\right\vert = 8$. Now, $\Gamma_{B}$ preserves $\ell$ and fixes $x$. Therefore $\Gamma_{B}$ is $GL_{2}(3)$, hence $\mathcal{S}$ is a $2$-$(13,8,462)$ design since $\Gamma \cong PSL_{3}(3)$ acts point $2$-transitively on $\mathcal{S}$. This completes the proof.
\end{proof}


\begin{thebibliography}{10}

\bibitem{OAG}
O.~A. AbuGhneim.
\newblock{All $(64,28,12)$ difference sets and related structures}
\newblock {\em Ars Combin.} 125 (2016) 271-285.

\bibitem{OPS}
O.~A. AbuGhneim, D.~Peifer K.~W. Smith.
\newblock{All $(96,20,4)$ difference sets and related structures}
\newblock {\em Bull. Inst. Combin. Appl.} 85 (2019) 44--59


\bibitem{BLJ} 
T.~Beth, D.~Jungnickel, H.~Lenz, 
\newblock {\em Design Theory, Volume I.}. 
\newblock{Encyclopedia of Mathematics and Its Applications, vol. 69, 2nd edn. Cambridge University Press, Cambridge (1999).}



\bibitem{BHRD}
	J.~N. Bray, D.~F. Holt, M.~Colva Roney-Dougal
   \newblock {\em The maximal subgroups of the low-dimensional finite classical groups}
   \newblock{Cambridge University Press, Cambridge, 2013}
   \newblock{London Mathematical Society Lecture Note Series}
    
\bibitem{BGMV}
S.~Brai\'{c}, A.~Golemac, J.~Mandi\'{c}, T.~Vu\v{c}i\v{c}i\'{c}
\newblock {Primitive symmetric designs with up to 2500 points}
\newblock{ \em J. Combin. Des.} \textbf{19} (2011) 463--474


\bibitem{CK} 
P.~J. Cameron, G. and W.~M. Kantor, 
\newblock{2‐transitive and antiflag transitive collineation groups of finite projective spaces},
\newblock{\em J. Algebra} 60 (1979) 384–422.



\bibitem{COT} 
P.~J. Cameron, G.~R. Omidi, B.~Tayfeh-Rezaie. 
\newblock{$3$-designs from $PGL(2, q)$}, 
\newblock{ \em Electron. J. Combin.} 13, no. 1, Research Paper 50, pp. 11 (2006).



\bibitem{CZ} 
A.~R. Camina P.~H. Zieschang,
\newblock On the normal structure of the flag transitive automorphism groups of $2$-designs, 
\newblock {\em J. London Math. Soc. (2)} \textbf{41} (1990) 555--564.
\bibitem{At}
	J.~H. Conway, R.~T. Curtis, S.~P. Norton, R.~A. Parker, and R.~A. Wilson.
	\newblock{\em Atlas of finite groups}.
	\newblock{Oxford University Press, Eynsham, 1985.}
	\newblock{Maximal subgroups and ordinary characters for simple groups, With
	computational assistance from J. G. Thackray.}
    

\bibitem{Demb}
	P.~Dembowski.
    \newblock {\em Finite geometries}
	\newblock{Springer-Verlag, Berlin-New York, 1968.}
	\newblock{ Ergebnisse der Mathematik und ihrer Grenzgebiete, Band 44.}

\bibitem{DM}
	J.~D. Dixon and B.~Mortimer.
	\newblock {\em Permutation groups}, volume 163 of {\em Graduate Texts in
		Mathematics}.
	\newblock Springer-Verlag, New York, 1996.

 \bibitem{GAP}
	The GAP~Group.
	\newblock {\em GAP -- Groups, Algorithms, and Programming, Version 4.12.2},
	2022.

\bibitem{Go} 
D.~Gorestein.
\newblock {\em Finite Groups}. 
\newblock Chelsea Publishing Company, New York, 1980.

    
\bibitem{HM}
	D.~G. Higman and J.~E. McLaughlin.
	\newblock Geometric {$ABA$}-groups.
	\newblock {\em Illinois J. Math.}, 5:382--397, 1961.


\bibitem{Hir} J.~W. P. Hirshfeld.
\newblock{\em Projective Geometries Over Finite Fields}, 
\newblock Oxford Mathematical Monographs, Oxford University Press, Second edition, 1998.

\bibitem{HP} D.~R. Hughes, F.~C. Piper.
\newblock{\em Design Theory},
\newblock{ Cambridge Univ. Press, Cambridge, 1985.}
    
\bibitem{Holple}
D.~F. Holt, W.~Plesken.
\newblock {\em Perfect Groups}. 
\newblock Clarendon Press, Oxford, 1989. 


\bibitem{Hu} 
Q.~M. Hussain.
\newblock{On the totality of the solutions for the symmetrical incomplete block designs with $\lambda = 2$, $k= 5$ or $6$}, 
\newblock{ \em Sankhya} \textbf{7} (1945) 204–208.

\bibitem{AtMod}
	C.~Jansen, K.~Lux, R.~Parker, and R.~Wilson.
	\newblock {\em An atlas of {B}rauer characters}, volume~11 of {\em London
		Mathematical Society Monographs. New Series}.
	\newblock The Clarendon Press, Oxford University Press, New York, 1995.
	\newblock Appendix 2 by T. Breuer and S. Norton, Oxford Science Publications.


\bibitem{Ka69}
W.~M. Kantor.
\newblock {Automorphism groups of designs}
\newblock {\em Math. Z.}, 109 (1969) 26--252.

\bibitem{Ka72}
W.~M. Kantor.
\newblock {$k$-Homogeneous Groups}
\newblock {\em Math. Z.}, 1124 (1972) 261--265.

\bibitem{Ka75}
W.~M. Kantor.
\newblock {Symplectic Groups, Symmetric Designs, and Line Ovals}
\newblock {\em J. Algebra}, 33 (1975) 43--58.

\bibitem{Ka85}
	W.~M. Kantor.
	\newblock Classification of {$2$}-transitive symmetric designs.
	\newblock {\em Graphs Combin.}, 1(2):165--166, 1985.

\bibitem{KanLib}
	W.~M. Kantor, R.~A. Liebler.
	\newblock The rank {$3$} permutation representations of the finite classical groups,
    \newblock {\em Trans. Amer. Math. Soc.} \textbf{271} (1982) 1--71.

\bibitem{Land}
	E.~S. Lander.
	\newblock {\em Symmetric designs: an algebraic approach}, volume~74 of {\em
		London Mathematical Society Lecture Note Series}.
	\newblock Cambridge University Press, Cambridge, 1983.

\bibitem{LPR}
	M.~Law, C.~E. Praeger, and S.~Reichard.
	\newblock Flag-transitive symmetric {$2\text{-}(96,20,4)$}-designs.
	\newblock {\em J. Combin. Theory Ser. A}, 116(5):1009--1022, 2009.

\bibitem{Lieb}
	M.~W. Liebeck.
	\newblock The affine permutation groups of rank three.
	\newblock {\em Proceedings of the London Mathematical Society},
	3-54(3):477--516, may 1987.

\bibitem{Lu}
H.~L\"{u}neburg.
\newblock {\em Translation Planes},
\newblock{ Springer, Berlin, 1980.} 

\bibitem{MS}
	J.~Mandi\'{c} and A.~\v{S}uba\v{s}i\'{c}.
	\newblock Flag-transitive and point-imprimitive symmetric designs with {$\lambda\leq10$}.
	\newblock {\em J. Combin. Theory Ser. A}, 189:Paper No. 105620, 2022.

\bibitem{Mo0} A.~Montinaro. 
\newblock Classification of the $2$-$(k^{2},k,\lambda )$ design, with $\lambda \mid k$, admitting a flag-transitive automorphism
group of affine type. 
\newblock{ \em J. Comb. Theory, Ser. A} \textbf{195} (2023) 105710.

\bibitem{Mo} 
A.~Montinaro. 
\newblock On the symmetric $2$-$(v,k,\lambda)$ designs with a flag-transitive point-imprimitive automorphism group.
\newblock{ \em J. Algebra} \textbf{653} (2024) 54--101.

\bibitem{Na} 
H.~K. Nandi.
\newblock Enumerations of nonisomorphic solutions of balanced incomplete block designs, 
\newblock{ \em Sankhya} \textbf{7} (1946) 305--312.

\bibitem{ORR}
	E.~O'Reilly-Regueiro.
	\newblock On primitivity and reduction for flag-transitive symmetric designs.
	\newblock {\em J. Combin. Theory Ser. A}, 109(1):135--148, 2005.

\bibitem{DifSet}
    D.~Peifer.
    \newblock {\em DifSets, an algorithm for enumerating all difference sets in a group},
    \newblock { Version 2.3.1 (2019), https://dylanpeifer.github.io/difsets}.

\bibitem{P}
	C.~E. Praeger.
	\newblock The flag-transitive symmetric designs with 45 points, blocks of size
	12, and 3 blocks on every point pair.
	\newblock {\em Des. Codes Cryptogr.}, 44(1-3):115--132, 2007.
	
	\bibitem{PZ}
	C.~E. Praeger and S.~Zhou.
	\newblock Imprimitive flag-transitive symmetric designs.
	\newblock {\em J. Combin. Theory Ser. A}, 113(7):1381--1395, 2006. 

    \bibitem{SP}
	C.~E. Praeger and C.~Schneider.
    \newblock {\em Permutation groups and {C}artesian decompositions}
    \newblock Cambridge University Press, Cambridge, 2018,
    \newblock London Mathematical Society Lecture Note Series.

     \bibitem{sane}
     S.~E.Sane.
    \newblock On a class of symmetric designs. 
    \newblock{In \em Combinatorics and applications (Calcutta, 1982)},
    \newblock Indian Statist. Inst., Calcutta, pp. 292–302

    \bibitem{Design}
     L.~H. Soicher, 
      \newblock{The DESIGN package for GAP, Version 1.3}, http://designtheory.org/software/gapdesign/, 2006.

    \bibitem{TCZ}
    \newblock{Classification of flag-transitive point-primitive non-symmetric $2$-$(v,k,\lambda)$ designs with $v < 100$}
    \newblock{ \em Appl. Math. Comput.}, 430 (2022) 127278.

    \bibitem{ZZ}
    C.~Zhong and S.~Zhou, 
    \newblock On flag-transitive automorphism groups of $2$-designs, 
    \newblock{\em Discrete Math.} \textbf{346} (2023) 113227.
	
\end{thebibliography}
\end{document}